\documentclass[11pt]{article}
\usepackage{amssymb, amsthm, amsmath, amscd, empheq}
\usepackage{tikz-cd}
\usepackage{relsize}
\usepackage{graphicx}

\newtheorem{theorem}{Theorem}[section]
\newtheorem{lemma}[theorem]{Lemma}
\newtheorem{proposition}[theorem]{Proposition}
\theoremstyle{definition}
\newtheorem{definition}[theorem]{Definition}
\newtheorem{corollary}[theorem]{Corollary}

\newtheorem{remark}[theorem]{Remark}
\newtheorem{notation}[theorem]{Notation}
\makeatletter
    
    \newcommand{\Rmnum}[1]{\expandafter\@slowromancap\romannumeral #1@}
  \makeatother

\theoremstyle{definition}
\theoremstyle{definition}
\theoremstyle{definition}

\newcommand{\pf}{{\operatorname{pf}}}
\newcommand{\K}{\mathrm{K}} 
\newcommand{\KK}{\mathrm{KK}} 
\newcommand{\KL}{\mathrm{KL}} 

\newcommand{\Z}{\mathbb{Z}}
\newcommand{\Q}{\mathbb{Q}}
\newcommand{\R}{\mathbb{R}}

\newcommand{\T}{\mathbb{T}}

\newcommand{\Hom}{\operatorname{Hom}}

\newcommand{\ep}{\varepsilon}

\newcommand{\af}{{\alpha}}

\newcommand{\beq}{\begin{eqnarray}}
\newcommand{\eneq}{\end{eqnarray}}

\usepackage[all]{xy}
\usepackage{cleveref}

\title{\uppercase{Higher dimensional Bott classes and the stability of rotation relations}} 
\author{Sayan Chakraborty\footnote{Stat-Math unit, Indian Statistical Institute, 203 Barrackpore Trunk Road, Kolkata, 700108, India. Email: sayan2008@gmail.com}
 \,\, and Jiajie Hua\footnote{College of data science, Jiaxing University,  899 Guangqiong Road, Jiaxing, Zhejiang, 314001,
 China.  Email: jiajiehua@zjxu.edu.cn}}
\date{}

\begin{document}
\maketitle

\begin{abstract}Let $\Theta=(\theta_{jk})_{n\times n}$ be a real skew-symmetric $n\times n$ matrix for $n\geq 2$. Under some mild non-integrality conditions on $\Theta,$ we construct Rieffel-type projections as higher dimensional Bott classes in the $n$-dimensional noncommutative torus $\mathcal{A}_{\Theta}.$ These projections generate $\K_0(\mathcal{A}_\Theta)$ when $\Theta$ is strongly totally irrational.  As an application, when $\Theta$ is strongly totally irrational, we show that:
For any $\varepsilon>0,$  there exists $\delta>0$ (depending only on $\varepsilon$ and $\Theta$) satisfying the following: For any unital simple separable $C^*$-algebra $\mathcal{A}$ with tracial rank at most one, and for any $n$-tuple of unitaries $u_1,u_2,\dots,u_n$ in $\mathcal{A}$, if  $u_1,u_2,\dots,u_n$ satisfy certain trace conditions and
 \begin{eqnarray*}\|u_ku_j-e^{2\pi i\theta_{jk}}u_ju_k\|<\delta,\,j,k=1,2,\dots,n,
 \end{eqnarray*}
   then there exists an $n$-tuple of unitaries $\tilde{u}_1,\tilde{u}_2,\dots,\tilde{u}_n$ in $\mathcal{A}$ such that
 \begin{eqnarray*}\tilde{u}_k\tilde{u}_j=e^{2\pi i\theta_{jk}}\tilde{u}_j\tilde{u}_k\, {\rm and}\, \|\tilde{u}_j-u_j\|<\varepsilon,\, j,k=1,2,\dots,n.
 \end{eqnarray*}
 We also show that these trace conditions are also  necessary in the above application.
\end{abstract}

\


\

\section{Introduction}

The $n$-dimensional noncommutative torus $\mathcal{A}_{\Theta}$ is the universal $C^*$-algebra generated by unitaries $\mathfrak{u}_1,\mathfrak{u}_2,\dots,\mathfrak{u}_n$ subject to the relations
\begin{eqnarray}
\mathfrak{u}_k \mathfrak{u}_j=e^{2\pi i\theta_{jk}}\mathfrak{u}_j\mathfrak{u}_k,\nonumber
\end{eqnarray}
for $j,k=1,2,3,\dots,n,$ where $\Theta=(\theta_{jk})$ is a real skew-symmetric $n\times n$ matrix. The noncommutative tori play a major role in the theory of operator algebras and noncommutative geometry, and serve as key examples of noncommutative spaces.
Rieffel (in \cite{Rieffel-1988}) had constructed projective modules over $n$-dimensional noncommutative tori while in \cite{Elliott}, Elliott computed the $\K$-theory and the range of tracial states of these algebras. Motivated by the work of \cite{ELPW}, in \cite{Chakraborty} the first named author constructed projective modules over a certain field of $C^*$-algebras to provide explicit bases for $\K_0(\mathcal{A}_{\Theta})$ for any higher dimensional noncommutative torus $\mathcal{A}_{\Theta}$. The explicit generators of the $\K_0$-group in \cite{Chakraborty} was given using the projective modules over the $n$-dimensional noncommutative torus $\mathcal{A}_{\Theta}.$ However, for some practical purposes it is desirable to have a description of the explicit generators of the $\K_0$-group in terms of the projections in $\mathcal{A}_{\Theta}.$ Here in this paper, under some  non-integrality conditions (satisfied by many examples) on $\Theta,$ we construct projections in $\mathcal{A}_{\Theta},$ which in turn provides a set of explicit generators of $\K_0(\mathcal{A}_{\Theta}),$ when $\Theta$ is totally irrational.

For a skew-symmetric $n\times n$ matrix $A,$ if we denote the sub-matrix of $A$ consisting of rows and columns indexed by $i_1,i_2,\dots,i_{2p},$ for some numbers
  $1\leq i_1<i_2<\cdots<i_{2p}\leq n,$ $I:=(i_1,i_2,\dots,i_{2p}), |I|:=2p,$ by $M_{I}^A,$ then Elliott's result on the range of the canonical tracial state $\tau_{\Theta}$ of $\mathcal{A}_{\Theta}$ may be given as follows:

  \begin{eqnarray}\label{range_intro}
  (\tau_{\Theta})_*(\K_0(\mathcal{A}_{\Theta}))=\mathbb{Z}+\sum\limits_{0<|I|\leq n}{\rm pf}(M_{I}^{\Theta})\mathbb{Z},
\end{eqnarray}
where $\mathrm{pf}(M_{I}^{\Theta})$  denotes the pfaffian (see Definition~\ref{def:pf}) of the matrix $M_{I}^{\Theta}.$ However, it is not clear whether the numbers $\mathrm{pf}(M_{I}^{\Theta})$ can be realized as traces of projections (which we call Rieffel-type projections, if exist) inside $\mathcal{A}_{\Theta}$ or not. This question, for $n=4$, was raised by Elliott himself in \cite[1.2]{Elliott}. For $n=2,$ under non-integrality assumption on the non-zero entries in $\Theta,$ the well known Rieffel projection (see Definition \ref{def of Rieffel P}) serves the purpose. It is not hard to see that this also works for $n=3,$ also, as for any $n$ and $|I|=2,$ the Rieffel projections give elements in $\K_0(\mathcal{A}_{\Theta}),$ whose traces are $\mathrm{pf}(M_{I}^{\Theta}).$ For $n=4,$ and $|I|=4,$ Boca (see \cite[Remark on page 198]{Boca}) constructed a projection of trace $\mathrm{pf}(M_{I}^{\Theta})$ inside $\mathcal{A}_{\Theta}.$  In this paper, we generalize Boca's idea to construct Rieffel-type projections inside $\mathcal{A}_{\Theta}$ under some mild non-integrality conditions on $\Theta.$ Explicitly we prove,

 \begin{theorem}[cf. Corollary~\ref{even n projection explicit_m=0}]
  For any even number $n=2l\geq2,$ if $\Theta\in \mathcal{T}_n$ satisfies \begin{eqnarray}\label{pf2}\frac{{\rm pf}\left(M^{\Theta}_{(1,2,\dots,n-p-1,n-p)}\right)}{{\rm pf}\left(M^{\Theta}_{(1,2,\dots,n-p-2)}\right)}\in (0,1)\quad {\rm for}\,\, p=n-2j-2,\quad j=1,\dots,l-2,
\end{eqnarray} and  $
 \frac{{\rm pf}\left(M^{\Theta}_{(1,2,\dots,n-1,n)}\right)}{{\rm pf}\left(M^{\Theta}_{(1,2,\dots,n-2)}\right)}\notin \mathbb{Z},$ then
  there exist a (Rieffel-type) projection $p$ inside $\mathcal{A}_{\Theta}$ such that
  \begin{eqnarray}\tau_{\Theta}(p)={\rm pf}(\Theta)+k_{(1,2)}{\rm pf}(M^{\Theta}_{(1,2)})+\cdots+k_{(1,2,\dots,n-2)}{\rm pf}(M^{\Theta}_{(1,2,\dots,n-2)})+k_0\nonumber
   \end{eqnarray}
   for some $k_I,k_0\in \mathbb{Z},$ where $I=(1,2,\dots,n-2s)$ and $s=1,2,\dots,l-1,$ $\tau_{\Theta}$ is the canonical tracial state on $\mathcal{A}_{\Theta}.$
\end{theorem}
With the above theorem in hand, any tracial value in Equation~(\ref{range_intro}) is realized by a projection in $\mathcal{A}_{\Theta}$ under certain assumptions on $\Theta$ (see Theorem~\ref{PI1}). This provides an (generalized) answer to the question of Elliott mentioned before. The assumptions are quite natural to expect for such projections to exist in $\mathcal{A}_\Theta$ as these projections are constructed using an induction argument from the 2-dimensional case.

 When $(\tau_{\Theta})_*$ is injective on $\K_0(A_{\Theta})$, we call such a $\Theta$ totally irrational. Moreover,  when any even dimensional skew-symmetric submatrix of $\Theta$ satisfies (\ref{pf2}), we call it strongly totally irrational (see Definition \ref{a3} for the explicit definition of strongly total irrationality). For a strongly totally irrational $\Theta,$ the above Rieffel-type projections also provide a basis for $\K_0(\mathcal{A}_\Theta)$ (see Theorem~\ref{PI}). In the second part of the paper, we use this fact to the study of the stability of rotation relations, which we describe below.

The notion of stability appears in many forms throughout mathematics, where it
often expresses how much being a little bit wrong matters. Following Hyers and
Ulam (\cite{Hyers}), a general sense of this notion can be expressed as follows: Are
elements that ``almost" satisfy some equation ``close" to some elements that exactly
satisfy the equation? In applied mathematics such a question is important since
we should not expect computer models, or indeed the real world, to give answers
that precisely match what our theory predicts.

An example of a concrete stability problem is the following:
Given $\varepsilon>0$, is there a $\delta>0,$ depending only on $\varepsilon,$ such that if $a$ and $b$ are two $n\times n$ self-adjoint  matrices with $\|a\|,\|b\|\leq 1$ satisfying
$$\|ab-ba\|<\delta,$$
then there exists a pair of $n\times n$ self-adjoint matrices $\tilde{a}$ and $\tilde{b}$  such that
$$\tilde{a}\tilde{b}=\tilde{b}\tilde{a},\,\,\,\,\|a-\tilde{a}\|<\varepsilon\,\,{\rm and}\,\,\|b-\tilde{b}\|<\varepsilon?$$
This is an old and famous question in matrix and operator theory, namely whether any pair of almost commuting self-adjoint matrices are norm close to a pair of exactly commuting self-adjoint matrices (\cite{Berg,Davidson,Rosenthal}), which is popularized by Halmos (\cite{Halmos}).

In the above question, it is important that $\delta$ is a universal constant independent of  the matrix size $n.$
This question was solved affirmatively by Lin in the 1990's (see \cite{Halmos,Lin2}).
The corresponding questions for a pair of unitary matrices and for a triple of self-adjoint matrices are all false, as pointed out by Voiculescu in \cite{Voiculescu1} and \cite{Voiculescu}. However the story does not end here. An obstruction has been found by Exel and Loring in \cite{EL} in the corresponding question for a pair of unitary matrices. The answer becomes yes if this obstruction vanishes (see  \cite{EL,ELP1,ELP,ExL1,ExL2}).

A natural generalization for pairs of almost commuting unitary matrices is to see what happens for pairs of unitaries that almost commute up to a scalar with norm one. It turns out that similar conclusion holds, and in fact one can deal with more general ambient $C^*$-algebras rather than just matrix algebras. More precisely, in \cite{Hua-Lin}, the second named author and  Lin proved the following:

\begin{theorem} \label{HH} Let $\theta$ be a real number in $(-\frac{1}{2},\frac{1}{2})$. For any $\varepsilon>0,$  there exists $\delta>0$, depending only on $\varepsilon$ and $\theta$, such that if $u$ and $v$ are two unitaries in any unital simple separable $C^*$-algebra $\mathcal{A}$ with tracial rank zero satisfying
\beq \|vu-e^{2\pi i\theta}uv\|<\delta\,\,{\rm and}\nonumber \\
\tau(\log(vuv^*u^*))=2\pi i\theta\quad \label{tc}
\eneq
for all tracial states $\tau$ on $\mathcal{A}$, then there exists a pair of unitaries $\tilde{u}$ and $\tilde{v}$ in $\mathcal{A}$ such that
\[
\tilde{v}\tilde{u}=e^{2\pi i\theta}\tilde{u}\tilde{v},\,\,\|u-\tilde{u}\|<\varepsilon\, \,{\rm and}\,\,\|v-\tilde{v}\|<\varepsilon.
\]
\end{theorem}

Note that the trace condition (\ref{tc}) is also necessary.

Let $\theta\in\mathbb{R}$. We call a pair of unitaries $u, v$ with $ vu = e^{2\pi i \theta} uv$ to satisfy the \emph{rotation relation} with respect to $\theta$,
since the universal $C^*$-algebra generated by such unitaries is the rotation algebra.
So  another way to phrase Theorem \ref{HH} is to say that the
rotation relation is stable in a unital simple separable $C^*$-algebras with tracial rank zero, provided that the trace condition (\ref{tc}) is satisfied.

In \cite{Hua-Wang}, the second named author and Wang further studied the stability of the rotation relations
for three unitaries in unital simple separable $C^*$-algebras with tracial rank at most one, and proved the following theorem:
\begin{theorem}[\cite{Hua-Wang}]\label{nondegenerate stability}
Let $\Theta = (\theta_{jk})_{3 \times 3}$ be a nondegenerate real skew-symmetric $3\times3$ matrix (here non-degeneracy is equivalent to $\dim_{\mathbb{Q}}(\rm{span}_{\mathbb{Q}}(1,\theta_{12},
\theta_{13},\theta_{23}))\geq 3,$ see Lemma 3.1 of \cite{Itza-Phillips}), where
$\theta_{jk}\in [0,1)$ for $j,k=1,2,3.$
Then for any $\varepsilon>0$, there exists $\delta > 0$ satisfying the following:
For any unital simple separable $C^*$-algebra $\mathcal{A}$ with tracial rank at most one,
any three unitaries $u_1,u_2,u_3$ in $\mathcal{A}$ such that
\[
\|u_ku_j-e^{2\pi i\, \theta_{jk}}u_ju_k\|<\delta,\,\, j, k = 1, 2, 3,
\]
 there exists a triple of unitaries $\tilde{u}_1,\tilde{u}_2,\tilde{u}_3$ in $\mathcal{A}$ such that
\[
\tilde{u}_k\tilde{u}_j=e^{2\pi i \, \theta_{jk}}\tilde{u}_j\tilde{u}_k\,\,\,\,\mbox{and}\,\,\,\,\|\tilde{u}_j-u_j\|<\varepsilon,\,\,\,j,k=1,2,3
\]
if and only if
\begin{eqnarray}\label{3obs}
\tau(\log_{\theta_{jk}}(u_ku_ju_k^*u_j^*))=2\pi i\theta_{jk}\,\, for\,\, j,k = 1,2,3\,\, and\,\, all\,\, tracial\,\, states\,\, \tau\,\,on\,\, \mathcal{A},\nonumber
\end{eqnarray}
where $\log_{\theta_{jk}}(u_ku_ju_k^*u_j^*)$ is defined as in Definition \ref{logarithm}.
\end{theorem}
In \cite{Hua-IJM} and \cite{Hua-JMAA}, the second named author has also considered the stability of some other relations of  three unitaries in certain classes of  $C^*$-algebras.

In this paper, we will generalize the stability of rotation relations of three unitaries (Theorem \ref{nondegenerate stability}) to the stability of rotation relations of $n$ unitaries for any integer $n\geq 2$ in a unital simple separable $C^*$-algebra $\mathcal{A}$ with tracial rank at most one, when $\Theta$ is strongly totally irrational.

\begin{theorem}[cf. Theorem~\ref{main result}] \label{the1}Let $\Theta=(\theta_{jk})_{n\times n}$ be a strongly totally irrational real skew-symmetric $n\times n$ matrix for $n\geq 2.$ Then for any $\varepsilon>0,$ there exists $\delta>0$ satisfying the following: For any unital simple separable $C^*$-algebra $\mathcal{A}$ with tracial rank at most one, any $n$-tuple of unitaries $u_1,u_2,\dots,u_n$ in $\mathcal{A}$ such that
 \begin{eqnarray}\|u_ku_j-e^{2\pi i\theta_{jk}}u_ju_k\|<\delta,\,\,j,k=1,2,\dots,n,\label{app111}
 \end{eqnarray}
  and
\begin{eqnarray}\label{l1}\tau(R_{I}(u_1,u_2,\dots,u_n))={\rm pf}(M_{I}^{\Theta})+k_I+\sum\limits_{0<|J|< l}{\rm pf}(M_{J}^{\Theta})k_{I_J},
\end{eqnarray}
 for all possible $I=(i_1,i_2,\dots,i_{l})$, $1\leq i_1<i_2<\cdots<i_l\leq n,$ $l=2m,$ $m\in\{1,2,\dots,\lfloor\frac{n}{2}\rfloor\},$ and for all tracial states $\tau$ on $\mathcal{A}$, where $R_{I}(u_1,u_2,\dots,u_n)$ is as in Definition \ref{RI},  and
   $k_{I_J},k_I\in \mathbb{Z}$  are the numbers as in Corollary \ref{PI};
   there exists an $n$-tuple of unitaries $\tilde{u}_1,\tilde{u}_2,\dots,\tilde{u}_n$ in $\mathcal{A}$ such that
 \begin{eqnarray}\tilde{u}_k\tilde{u}_j=e^{2\pi i\theta_{jk}}\tilde{u}_j\tilde{u}_k\, \, {\rm and}\,\, \|\tilde{u}_j-u_j\|<\varepsilon,\,\, j,k=1,2,\dots,n.\label{=111}
 \end{eqnarray}
\end{theorem}
In particular, when $n=2$ and  $\theta_{12}$ is an irrational number, this is a generalization of any unital simple separable $C^*$-algebra with tracial rank zero in Theorem \ref{HH} to any unital simple separable $C^*$-algebra with tracial rank at most one. When $n=3$, this is just Theorem \ref{nondegenerate stability}
in the case that $\Theta$ is totally irrational (note that totally irrationality of $\Theta$ implies that it is nondegenerate).

Finally, we  will also prove that the trace conditions (\ref{l1}) are also  necessary conditions in  Theorem \ref{the1}.

This paper is organized as follows. In Section~\ref{sec:pre}, we list down some notations and known results. In Section~\ref{sec:rie}, we give a very concrete description of higher dimensional Bott classes. Section~\ref{sec:stability} deals with proving the stability of the rotation relations of $n$ unitaries in the class of $C^*$-algebras of tracial rank at most one with certain trace conditions when $\Theta$ is strongly totally irrational (Theorem~\ref{the1}). In Section~\ref{sec:stability} we also show that these trace conditions
are necessary to prove the stability of rotation relations.

\section{Preliminaries}\label{sec:pre}
Let $n\geq2$ be an integer and $\mathcal{T}_n$ denote the space of $n\times n$ real skew-symmetric matrices.

\begin{definition}(\cite{Rieffel-1990}) Let $\Theta=(\theta_{jk})_{n \times n}\in \mathcal{T}_n$. {\it The noncommutative torus} $\mathcal{A}_{\Theta}$ is the universal $C^*$-algebra generated by unitaries $\mathfrak{u}_1,\mathfrak{u}_2,\dots,\mathfrak{u}_n$ subject to the relations
$$\mathfrak{u}_k\mathfrak{u}_j=e^{2\pi i\theta_{jk}}\mathfrak{u}_j\mathfrak{u}_k$$
for $1\leq j,k\leq n.$ (Of course, if all $\theta_{j,k}$ are integers, it is not really noncommutative.) Throughout this paper, we will use $\mathfrak{u}_1,\mathfrak{u}_2,\dots,\mathfrak{u}_n$ to represent the $n$ generators of $\mathcal{A}_\Theta$, sometimes without special emphasis. In particular, given $\theta\in \mathbb{R},$ we also let $\mathcal{A}_\theta$ denote the  universal $C^*$-algebra generated by a pair of unitaries $\mathfrak{u}_1$ and $\mathfrak{u}_2$ subject to $\mathfrak{u}_2\mathfrak{u}_1=e^{2\pi i \theta}\mathfrak{u}_1\mathfrak{u}_2.$ If $\theta$ is irrational (resp. rational), $\mathcal{A}_{\theta}$ is called  the {\it irrational} (resp. {\it rational}) {\it rotation algebra}.
\end{definition}

For any  $\Theta=(\theta_{j,k})_{n \times n}$ in $\mathcal{T}_n,$   $\mathcal{A}_{\Theta}$ has a canonical tracial state $\tau_{\Theta}$ given by the integration over
the canonical action of $\widehat{\mathbb{Z}^n}$ ( see page 4 of \cite{Rieffel-1990} for more details).
We denote this trace by $\tau_{\Theta}$.

\begin{definition}\label{def:nondege}
A skew symmetric real $n \times n$ matrix $\Theta$ is {\it nondegenerate} if whenever $x\in \mathbb{Z}^n$ satisfies $e^{(2\pi i\langle x,\Theta y\rangle)}=1$ for all $y\in \mathbb{Z}^n,$ then $x=0.$ Otherwise, we say $\Theta$ is {\it degenerate}.

\end{definition}

The following theorem shows the structure and the $\K$-theory of $\mathcal{A}_{\Theta}$ when $\Theta$ is nondegenerate.
\begin{theorem}[\cite{Phillips-06}]\label{AT and K}
  Let $\Theta$ be in $\mathcal{T}_n$, with
$n\geq 2$.  The $C^*$-algebra $\mathcal{A}_{\Theta}$ is simple if and only if $\Theta$ is nondegenerate. Moreover, if $\mathcal{A}_{\Theta}$ is simple, then it is a unital  \rm{AT} algebra  and has the unique tracial state $\tau_{\Theta}$,  and $\K_0(\mathcal{A}_{\Theta})\cong \K_1(\mathcal{A}_{\Theta})=\mathbb{Z}^{2^{n-1}}.$
\end{theorem}

\begin{definition}[\cite{ExL2}]\label{Dbott}
{\rm\,
Define
 $$ f_1(e^{2\pi it})=\left\{
\begin{aligned}
1-2t, & & \mbox{if } 0\leq t\leq 1\slash 2, \\
-1+2t, & & \mbox{if }1\slash 2< t\leq 1, \\
\end{aligned}
\right.
$$
$$ g_1(e^{2\pi it})=\left\{
\begin{aligned}
(f_1(e^{2\pi it})-f_1(e^{2\pi it})^2)^{\frac{1}{2}}, & & \mbox{if } 0\leq t\leq 1\slash 2, \\
0,\phantom{(f_1(e^{2\pi it})-f_1(e^{2\pi it})^2)} & & \mbox{if }1\slash 2< t\leq 1, \\
\end{aligned}
\right.
$$
$$ h_1(e^{2\pi it})=\left\{
\begin{aligned}
0,\phantom{(f_1(e^{2\pi it})-f_1(e^{2\pi it})^2)} & & \mbox{if } 0\leq t\leq 1\slash 2, \\
(f_1(e^{2\pi it})-f_1(e^{2\pi it})^2)^{\frac{1}{2}}, & & \mbox{if }1\slash 2< t\leq 1. \\
\end{aligned}
\right.
$$
These are non-negative continuous functions defined on the unit circle $\mathbb{T}$.

Let $\mathcal{A}$ be a unital $C^*$-algebra and let $u,v\in \mathcal{A}$ be two unitaries. Define
$$ e(u,v)=\left(
\begin{aligned}
f_1(v)\phantom{f_1(v)} & & g_1(v)+h_1(v)u^* \\
g_1(v)+uh_1(v) & & 1-f_1(v)\phantom{(u)} \\
\end{aligned}
\right).
$$
It is a self-adjoint element. Furthermore if $vu=uv,$
then $e(u, v)$ is a projection.

 We denote by $C(\mathbb{T}^2)$ the $C^*$-algebra of continuous functions on the two-torus $\mathbb{T}^2$ generated by the two co-ordinate functions $z$ and $w$. In $M_2(C(\T^2))$, $e(z, w)$ is a non-trivial rank one projection.
Then
\begin{eqnarray}\label{b=bott}
b=[e(z,w)]-[\left(
\begin{aligned}
1 & & 0 \\
0 & & 0 \\
\end{aligned}
\right) ]\in \K_0(C(\T^2))
\end{eqnarray}
 is often called the Bott element for $C(\T^2).$ }
\end{definition}

There is a $\delta > 0$ (independent of the unitaries $u, v,$ and
$\mathcal{A}$, see \cite{Lor2} for the existence of such $\delta$) such that if $\|vu-uv\| <\delta,$ then the  spectrum of the element $e(u, v)$ has a gap at
1\slash 2. Let $\chi_{(\frac{1}{2}, \infty)}$ be the characteristic function on $(\frac{1}{2}, \infty).$ The Bott element of $u$ and $v$ is an element in $\K_0(\mathcal{A})$ as defined by
$$\mbox{bott}(u,v) = [\chi_{(1/2,\infty)}(e(u, v))]-[\left(
\begin{aligned}
1 & & 0 \\
0 & & 0 \\
\end{aligned}
\right) ].$$
The reader is referred to \cite{Exel,ExL1,ExL2,Lor} for more information on the Bott element.


\begin{definition}[\cite{Rieffel-PJM-1981}] \label{def of Rieffel P}Let $\theta\in (0,1).$ Choose $\varepsilon$ such that $0<\varepsilon \leq \theta<\theta+\varepsilon\leq 1.$ Set
$$f_2(e^{2\pi it})=\left\{
\begin{aligned}
\varepsilon^{-1}t,\phantom{ttttttttttt}& & 0\leq t\leq \varepsilon,\phantom{tttt}\\
1,\phantom{tttttttttttttt}& & \varepsilon\leq t\leq \theta,\phantom{tttt}\\
\varepsilon^{-1}(\theta+\varepsilon-t),& & \theta\leq t\leq \theta+\varepsilon,\\
0,\phantom{tttttttttttttt} & & \theta+\varepsilon \leq t\leq 1,\\
\end{aligned}\right.
$$
and
$$g_2(e^{2\pi it})=\left\{
\begin{aligned}
0,\phantom{tttttttttttttttttttttttttttt}& & 0\leq t\leq \theta,\phantom{tttt}\\
[f_2(e^{2\pi it})(1-f_2(e^{2\pi it}))]^{1\slash 2},& & \theta\leq t\leq \theta+\varepsilon,\\
0,\phantom{tttttttttttttttttttttttttttt}& & \theta+\varepsilon\leq t\leq 1.\\
\end{aligned}\right.
$$
Then $f_2$ and $g_2$ are the real-valued functions on the circle which satisfy\\
(1) $g_2(e^{2\pi it})\cdot g_2(e^{2\pi i(t-\theta)})=0, $\\
(2) $g_2(e^{2\pi it})\cdot [f_2(e^{2\pi it})+f_2(e^{2\pi i(t+\theta)})]=g_2(e^{2\pi it})$ and\\
(3) $f_2(e^{2\pi it})=[f_2(e^{2\pi it})]^2+[g_2(e^{2\pi it})]^2 +[g_2(e^{2\pi i(t-\theta)})]^2.$

\noindent Let $\mathfrak{u}_1$, $\mathfrak{u}_2$ be the canonical generators of $\mathcal{A}_{\theta}$.  The {\it Rieffel projection} in $\mathcal{A}_\theta$ is the projection
\begin{eqnarray}p_{\theta} = g_2(\mathfrak{u}_1)\mathfrak{u}_2^* + f_2(\mathfrak{u}_1) + \mathfrak{u}_2g_2(\mathfrak{u}_1)\label{Rieffel pro}
\end{eqnarray} (for the details of the above construction, see Theorem 1.1 of \cite{Rieffel-PJM-1981}).
\end{definition}

For $\theta \in (0,1)$, similar to the construction of the Rieffel projection, we give the following definition:

\begin{definition}
Let $\theta \in (0, 1)$. Let $\mathcal{A}$ be a unital $C^*$-algebra and $u, v$ be a pair of unitaries in $\mathcal{A}$. We define $e_{\theta}(u, v)$ to
be the element $g_2(u)v^* + f_2(u) + vg_2(u)$, where $f_2, g_2$ are the functions on the unit circle
defined as in Definition \ref{def of Rieffel P}. In particular, $e_{\theta}(u, v)$ is a projection if $u$ and $v$ satisfy $vu=e^{2\pi i\theta}uv.$
\end{definition}
\noindent It is clear that $e_{\theta}(u, v)$ is always self-adjoint.

\begin{proposition}[Proposition 4.9 of \cite{Hua-Wang}]\label{P_RieffelEgap}
Let $\theta \in (0, 1)$. There exists $\delta > 0$, depending only on $\theta,$ such that for any unital $C^*$-algebra $\mathcal{A}$,
any pair of unitaries $u, v$ in $\mathcal{A}$, if $u$ and $v$ satisfy
$\|vu - e^{2\pi i \theta}uv\| < \delta$,
then
\[
\|(e_{\theta}(u, v))^2 - e_{\theta}(u, v)\| < \frac{1}{4}.
\]
In particular, the spectrum of $e_{\theta}(u, v)$ has a gap at $\frac{1}{2}$.
\end{proposition}

\begin{definition}(\cite{Hua-Wang})\label{D_RieffelProj}
Let $\theta \in [0, 1)$. Let $\mathfrak{u}$, $\mathfrak{v}$ be the canonical generators of $\mathcal{A}_{\theta}$. If $\theta \neq 0$,
we define $b_{\mathfrak{u},\mathfrak{v}} \in \K_0(\mathcal{A}_{\theta})$ to be the equivalent class of the Rieffel projection as constructed in Definition \ref{def of Rieffel P}.
If $\theta = 0$, we let $b_{\mathfrak{u},\mathfrak{v}} \in \K_0(\mathcal{A}_{\theta})$ be the Bott element (see (\ref{b=bott})).
\end{definition}


\begin{definition}
\label{D_RieffelE} Let $\theta \in [0, 1)$.
 If $\theta \neq 0$ let $\delta>0$ be  chosen as in Proposition \ref{P_RieffelEgap},
and let $u, v$ be a pair of unitaries in a unital $C^*$-algebra $\mathcal{A}$ with $ \|vu - e^{2\pi i \theta}uv\| < \delta$. Also let $\chi_{(\frac{1}{2}, \infty)}$ be the characteristic function on $(\frac{1}{2}, \infty).$ We define $R_{\theta}(u, v)=[\chi_{(\frac{1}{2}, \infty )}(e_{\theta}(u, v))].$
 If $\theta=0$, let $\delta>0$ be chosen as in the paragraph after Definition \ref{Dbott},  we  define
$R_{\theta}(u, v)={\rm bott}(u,v).$
\end{definition}
In particular, we have $R_{\theta}(\mathfrak{u}, \mathfrak{v})=b_{\mathfrak{u}, \mathfrak{v}}.$

\begin{notation}For any real number, we use $\lfloor\,\cdot\,\rfloor$ to represent its integer part.
  Let $\mathcal{A}$ be a unital $C^*$-algebra. Denote by $T(\mathcal{A})$ the tracial state space of $\mathcal{A}.$
  The set of all faithful tracial states on $\mathcal{A}$ will be denoted by $T_{f}(\mathcal{A}).$  Denote by $\mathrm{Aff}(T(\mathcal{A}))$ the space of all real affine continuous functions on $T(\mathcal{A}).$ If $\tau\in T(\mathcal{A})$, we will use $\tau^{\oplus k}$ for the trace $\tau\otimes {\rm tr}$ on $M_k(\mathcal{A})$ for all integer $k\geq1,$ where ${\rm tr}$ is the unnormalized trace on the matrix algebra $M_k.$ Denote by $\rho_{\mathcal{A}}:\K_0(\mathcal{A})\rightarrow\mathrm{Aff}(T(\mathcal{A}))$ the order preserving map induced by
  $\rho_{\mathcal{A}}([p])(\tau)=\tau^{\oplus n}(p),$ for all projections $p\in \mathcal{A}\otimes M_n.$ Denote by $\mathcal{A}_{s.a.}$ the set of all self-adjoint elements in $\mathcal{A}$. If $a\in \mathcal{A}_{s.a.}$, let $\hat{a}$ be the real affine function in ${\rm Aff}(T(\mathcal{A}))$ defined by $\hat{a}(\tau)=\tau(a)$ for all
$\tau\in T(\mathcal{A}).$

Denote by $U_n(\mathcal{A})$ the group of unitaries in $M_n(\mathcal{A})$ for $n\geq 1$. We often use $U(\mathcal{A})$ to express $U_1(\mathcal{A})$. Denote by
$U_0(\mathcal{A})$ the subgroup of $U(\mathcal{A})$ consists of the unitaries path connected to $1_\mathcal{A}$. Denote by $CU(\mathcal{A})$ the
closure of the subgroup generated by the commutators of $U(\mathcal{A})$.
Let $U_{\infty}(\mathcal{A})$ be the increasing union of $U_n(\mathcal{A}), n=1, 2, \dots$,  viewed as a topological group
with the inductive limit topology. Define $U_{\infty, 0}(\mathcal{A})$ and $CU_{\infty}(\mathcal{A})$ in a similar fashion. For a unitary $u\in \mathcal{A},$ define $\mathrm{Ad}u(a)=u^*au,$ for all $a\in \mathcal{A}.$ For any $a\in \mathcal{A},$ denote by $\mathrm{spec}(a)$ the spectrum of $a$.
\end{notation}

\begin{definition}[\cite{Hua-Wang}]\label{logarithm}
Let $\theta \in [0, 1)$. Denote by $\log_{\theta}$ the continuous branch of logarithm defined on $F_{\theta}=\{e^{it}:t\in (2\pi\theta-\pi,2\pi\theta+\pi)\}$ with values in $\{ri:r\in (2\pi\theta-\pi,2\pi\theta+\pi)\}$ such that $\log_{\theta}(e^{2\pi i \, \theta}) = 2\pi i \, \theta$.
  Note that if $u$ is any unitary in some $C^*$-algebra $\mathcal{A}$ such that
  $\|u-e^{2\pi i \theta}\|<2$, then ${\rm spec}(u)$ has a gap at $e^{2\pi i \theta + \pi i}$, and thus $\log_{\theta}(u)$
  is well defined. In particular, if $\theta=0,$ we simply write $\log(u)$ for $\log_{0}(u).$
\end{definition}

\begin{theorem}[Theorem 4.14 of \cite{Hua-Wang}]\label{trace formula} Let $\mathcal{A}$ be a unital $C^*$-algebra with $T(\mathcal{A})\neq \emptyset$. Let $\theta \in [0, 1)$, and $\delta$ be chosen as in Definition \ref{D_RieffelE}.
Then for any $u,v\in U(\mathcal{A})$ with $\|vu-e^{2\pi i \theta}uv\|<\delta$, and with $R_{\theta}(u,v)$ defined as in Definition \ref{D_RieffelE}, we have
\begin{eqnarray}\label{Exel formula}
 \rho_{\mathcal{A}}( R_{\theta}(u,v))(\tau)=\frac{1}{2\pi i}\tau(\log_{\theta}(vuv^*u^*)),\,\, \text{for all} ~\, \tau\in T(\mathcal{A}).
\end{eqnarray}
\end{theorem}
The formula (\ref{Exel formula}) is called the {\it generalized Exel trace formula}.

\begin{definition} Denote by $I^{(0)}$ the class of finite dimensional $C^*$-algebras.  For an integer $k\geq 1,$ denote
by $I^{(k)}$ the class of $C^*$-subalgebras of the form $PM_r(C(X))P,$ where $r\geq 1$ is an integer, $X$ is
a finite CW complex of covering dimension at most $k,$ and $P\in M_r(C(X))$ is a projection.
\end{definition}

\begin{definition}Denote by $I_{(k)}$ the class of those $C^*$-algebras which are quotients of $C^*$-algebras
in $I^{(k)}$. Let $\mathcal{C}\in I_{(k)}.$ Then $\mathcal{C} = PM_r(C(X))P,$ where $X$ is a compact subset of a finite CW
complex of covering dimension at most $k$, $r\geq 1$ and $P\in M_r(C(X))$ is a projection.
\end{definition}

We recall the definition of tracial rank of $C^*$-algebras:

\begin{definition}[\cite{LinPLMS}] Let $\mathcal{A}$ be a unital simple $C^*$-algebra.
Then $\mathcal{A}$ is said to have tracial rank at most $k$ if for every $\varepsilon>0$, every
finite subset $\mathcal{F}\subset \mathcal{A}$  and every
nonzero positive element $c\in \mathcal{A},$ there exists a
$C^*$-subalgebra $\mathcal{B}\in I_{(k)}$ with $1_{\mathcal{B}}=p$ such
that:\\
(i) $\|pa-ap\|<\varepsilon$ for all $a\in \mathcal{F}$,\\
(ii) dist$(pap,\mathcal{B})<\varepsilon$ for all $a\in \mathcal{F}$,\\
(iii) $1_{\mathcal{A}}-p$ is Murray-von Neumann equivalent to a projection in $\overline{c\mathcal{A}c}.$
\end{definition}

If $\mathcal{A}$ has tracial rank at most $k$, we write $\mathrm{TR}(\mathcal{A})\leq k.$ If furthermore, $\mathrm{TR}(\mathcal{A})\nleq k-1,$ then we say $\mathrm{TR}(\mathcal{A})= k.$

\begin{definition}
  Let $\mathcal{A}$ be a $C^*$-algebra. Following Dadarlat and Loring (\cite{Dadarlat-Loring-DMJ}), we set
  $$\underline{\K}(\mathcal{A})=\oplus_{n=1}^{\infty}(\K_0(\mathcal{A};\mathbb{Z}/n\mathbb{Z})\oplus \K_1(\mathcal{A};\mathbb{Z}/n\mathbb{Z})).$$
  Let $\mathcal{B}$ be another unital $C^*$-algebra. If, furthermore, $\mathcal{A}$ is assumed to be separable and satisfy the Universal Coefficient Theorem (\cite{Rosenberg-Schochet}), by $\cite{Dadarlat-Loring-DMJ}$,
  $$\KL(\mathcal{A},\mathcal{B})\cong {\rm Hom}_{\wedge}(\underline{\K}(\mathcal{A}),\underline{\K}(\mathcal{B})).$$
  Here $\KL(\mathcal{A},\mathcal{B})=\KK(\mathcal{A},\mathcal{B})/\mathrm{Pext}(\K_*(\mathcal{A}),\K_*(\mathcal{B})),$
  where $\mathrm{Pext}$ is the subgroup of $\mathrm{Ext}_{\mathbb{Z}}^1$ consisting of classes of pure extensions (see \cite{Dadarlat-Loring-DMJ} for details).
\end{definition}

\begin{definition}
Let $\mathcal{A}$ and $\mathcal{B}$ be two unital $C^*$-algebras. Denote by $\KL(\mathcal{A},\mathcal{B})^{++}$ the set of those $\kappa\in \Hom _{\Lambda}(\underline{\K}(\mathcal{A}),\underline{\K}(\mathcal{B}))$ such that
$$\kappa(\K_0(\mathcal{A})_+\backslash \{0\})\subset \K_0(\mathcal{B})_+\backslash \{0\}.$$

Denote by $\KL_e(\mathcal{A}, \mathcal{B})^{++}$ the set of those $\kappa\in  \KL(\mathcal{A}, \mathcal{B})^{++}$ such that $\kappa([1_{\mathcal{A}}]) = [1_{\mathcal{B}}].$
\end{definition}

\begin{remark}
Let $\mathcal{A}$ and $\mathcal{B}$ be two unital $C^*$-algebras.
Let $\varphi \colon \mathcal{A} \rightarrow \mathcal{B}$ be a unital homomorphism. Denote by $\varphi_*$ the induced homomorphism on the $\K$-groups, $\varphi_{*0}$ the induced homomorphism on the $\K_0$-groups and $\varphi_{*1}$ the induced homomorphism on the $\K_1$-groups.
Also denote by $\varphi_\sharp$ the induced map on the tracial state spaces and denote by $\varphi^{\ddagger}$ the induced map
\[
\varphi^{\ddagger} \colon  U(M_{\infty}(\mathcal{A}))/CU(M_{\infty}(\mathcal{A})) \rightarrow U(M_{\infty}(\mathcal{B}))/CU(M_{\infty}(\mathcal{B})).
\]
\end{remark}

Note that if $\mathcal{A}$ is a simple unital separable $C^*$-algebra with tracial rank at most one,
then it has stable rank one (see Theorem 3.6.10 of \cite{Lin-book}) and  there is an isomorphism (see Corollary 3.5 of \cite{Lin-JFA-2010})
\[
U(\mathcal{A})/CU(\mathcal{A}) \cong U(M_{\infty}(\mathcal{A}))/CU(M_{\infty}(\mathcal{A})).
\]

\begin{definition}
  Let $\mathcal{A}$ and $\mathcal{B}$ be two $C^*$-algebras and $L:\mathcal{A}\rightarrow \mathcal{B}$ be a linear map. Let $\eta>0$ and $\mathcal{G}\subset \mathcal{A}$ be a finite subset. We say $L$ is $\mathcal{G}$-$\eta$-{\it multiplicative} if
 \begin{center}
   $\|L(ab)-L(a)L(b)\|<\eta$ for all $a,b\in \mathcal{G}.$
 \end{center}
\end{definition}

Let $\mathcal{C}_n$ be a $C^*$-algebra such that $\K_*(\mathcal{A}; \Z/n\Z) \cong \K_*(\mathcal{A} \otimes \mathcal{C}_n)$.
For convenience, if $L \colon \mathcal{A} \rightarrow \mathcal{B}$ is a linear map, we will use the same symbol $L$
to denote the induced map $L \otimes \mathrm{id}_{n} \colon \mathcal{A} \otimes M_n \rightarrow \mathcal{B} \otimes M_n$
as well as $L \otimes \mathrm{id}_{\mathcal{C}_n} \colon \mathcal{A} \otimes \mathcal{C}_n \rightarrow \mathcal{B} \otimes \mathcal{C}_n$.

It is well known that if $a \in M_n(\mathcal{A})$ is an `almost' projection, then it is norm close to a projection. Two norm close projections are unitarily equivalent.
So $[a] \in \K_0(\mathcal{A})$ is well-defined. Similarly, if $b \in M_n(\mathcal{A})$ is an
`almost' unitary, we shall use $[b]$ to denote the equivalent class in
$\K_1(\mathcal{A})$. If $L \colon \mathcal{A} \rightarrow \mathcal{B}$ is an `almost' homomorphism, we shall
use $L_*$ to denote the induced (partially defined) map on the $\K$-theories.
From \cite{Dadarlat-IJM} or \cite{Lin-book}, we know that for any finite set $\mathcal{P} \subset \underline{\K}(\mathcal{A})$, there is a finite subset $\mathcal{G} \subset \mathcal{A}$ and $\eta > 0$ such that,
for any unital completely positive $\mathcal{G}$-$\eta$-multiplicative linear map $L$, $L_*$ is well defined on $\mathcal{P}$.

If $u$ is a unitary in $U_{\infty}(\mathcal{A})$, we shall use $[u]$ to denote the equivalence class
in $\K_1(\mathcal{A})$ and use $\bar{u}$ to denote the equivalent class in $U_{\infty}(\mathcal{A})/CU_{\infty}(\mathcal{A})$.
If $L \colon \mathcal{A} \rightarrow \mathcal{B}$ is an `almost' homomorphism
so that $L(u)$ is invertible, we define
\[
\langle L(u) \rangle = L(u)(L(u)^*L(u))^{-1/2}.
\]

\section{Higher dimensional Bott classes and construction of Rieffel-type projections}\label{sec:rie}
 As the pfaffian of an even dimensional skew-symmetric matrix will play a central role in the construction of the projections in higher dimensional noncommutative tori, we recall the definition of the pfaffian.

\begin{definition}\label{def:pf} Let $p\geq1$ be an integer.
  The {\it pfaffian} of a $2p\times 2p$ skew-symmetric matrix $A=(a_{ij})$ is a polynomial, denoted by $\mathrm{pf}(A)$, in the entries $a_{ij}$ such that $\mathrm{pf}(A)^2=\det (A)$ and $\mathrm{pf}(J_0'')=1,$ where
  $$J_0''=\left(\begin{matrix}
    J_0' & & & \\
     & J_0'&&\\
     &&\ddots&\\
     &&&J_0'
  \end{matrix}\right)$$ is the block diagonal matrix constructed from $p$ identical $2\times 2$ blocks of the form
  $J_0'=\left(\begin{matrix}
    0 & 1 \\
    -1 & 0
  \end{matrix}\right).$
\end{definition}

It can be shown that $\mathrm{pf}(A)$ always exists and is unique. Also it is well known that for any $2p\times 2p$ matrix $B,$
\begin{equation}\label{eq:pfdet} \mathrm{pf}(BAB^t)=\det(B){\rm pf}(A),\end{equation}
where $B^t$ is the transpose matrix of $B.$

To give some examples,
\begin{eqnarray}\label{f1}{\rm pf}\left(\left(\begin{matrix}
                  0 & \theta_{12} \\
                  -\theta_{12} & 0
                \end{matrix}
\right)\right)=\theta_{12},
\end{eqnarray}

\begin{eqnarray}\label{f2}{\rm pf}\left(\left(\begin{matrix}
                  0 & \theta_{12}&\theta_{13}&\theta_{14} \\
                  -\theta_{12} & 0&\theta_{23}&\theta_{24}\\
                  -\theta_{13} & -\theta_{23}&0&\theta_{34}\\
                  -\theta_{14} & -\theta_{24}&-\theta_{34}&0
                \end{matrix}
\right)\right)=\theta_{12}\theta_{34}-\theta_{13}\theta_{24}+\theta_{14}\theta_{23}.
\end{eqnarray}
More generally, if $n=2m$ for some integer $m\geq1,$ then for
$$A=\left(\begin{matrix}
                  0 & \theta_{12}&\cdots&&\cdots&\theta_{1n} \\
                  -\theta_{12} &\ddots&\ddots&&& \theta_{2n}\\
                  \vdots&\ddots&&&&\\
                  &&&&&\\
                  &&&&\ddots&\vdots\\
                  -\theta_{1(n-1)}&&&\ddots&\ddots& \theta_{(n-1)n}\\
                  -\theta_{1n} &\cdots&&\cdots& -\theta_{(n-1)n}&0
                \end{matrix}
\right),$$
the pfaffian of $A$ is given by $\sum_{\xi}(-1)^{|\xi|}\Pi_{s=1}^m \theta_{\xi(2s-1)\xi(2s)},$ where the sum is taken over all elements $\xi$ of the
permutation group $S_n$ such that
$\xi(2s-1)<\xi(2s)$ for all $1\leq s\leq m$ and $\xi(1)<\xi(3)<\cdots<\xi(2m-1).$

Let $n\geq 2$ be an integer, and let $p$ be an integer such that $1\leq p\leq\frac{n}{2}.$

\begin{definition}\label{2p minor}
  A {\it $2p$-pfaffian minor} (or just {\it pfaffian minor}) of a skew-symmetric $n\times n$ matrix $A$ is the pfaffian of a sub-matrix $M_{I}^A$ of $A$ consisting of rows and columns indexed by $i_1,i_2,\dots,i_{2p}$ for some numbers
  $1\leq i_1<i_2<\cdots<i_{2p}\leq n,$ and $I=(i_1,i_2,\dots,i_{2p}).$ We often use ${\rm pf}_{I}^A$
  as the abbreviation of ${\rm pf}(M_{I}^A)$ without special emphasis.
  We also use $M^{A}_{I,J}$ denote a sub-matrix of $A$ consisting of rows and columns indexed by $i_1,i_2,\dots,i_{n}$ and $j_1,j_2,\dots,j_{m},$ respectively, for some numbers
  $i_1<i_2<\dots<i_{n}, j_1<j_2<\dots<j_{m},$ and $I=(i_1,i_2,\dots,i_{n}), J=(j_1,j_2,\dots,j_{m}).$
\end{definition}
Note that the number of $2p$-pfaffian minors is the combinatorial number $C_{n}^{2p}$
and the number of all pfaffian minors is $2^{n-1}-1.$

We recall the following fact due to Elliott which will play a key role.
\begin{theorem}[Elliott]\label{elliott_image_of_trace}
Let $\Theta\in \mathcal{T}_n.$ Then for all $\tau\in T(\mathcal{A}_{\Theta})$, $\tau_*(\K_0(A_\Theta))$  is the range of the exterior exponential
 $$\operatorname{exp}_{\wedge}:\Lambda^{\operatorname{even}}\Z^n\rightarrow \R.$$
\end{theorem}
 \noindent We refer to (\cite[Theorem 3.1]{Elliott}) for the definition of exterior exponential and the proof of the above theorem. The range of the exterior exponential is well known and is given below as a corollary of the above theorem:

\begin{corollary} For $\Theta\in \mathcal{T}_n,$
$\tau_*(\K_0(\mathcal{A}_{\Theta}))$ is the subgroup of $\mathbb{R}$ generated by $1$ and the numbers
\begin{eqnarray}
\sum_{\xi}(-1)^{|\xi|}\Pi_{s=1}^{p}\theta_{j_{\xi (2s-1)}j_{\xi(2s)}}\nonumber
\end{eqnarray} for
$1\leq j_1<j_2<\cdots<j_{2p}\leq n,$ where the sum is taken over all elements
$\xi$ of the permutation group $S_{2p}$ such that $\xi(2s-1)<\xi(2s)$ for all
$1\leq s\leq p$ and $\xi(1)<\xi(3)<\cdots<\xi(2p-1)$ for any $\tau\in T(\mathcal{A}_{\Theta}).$
\end{corollary}

Noting that $\sum_{\xi}(-1)^{|\xi|}\Pi_{s=1}^{p}\theta_{j_{\xi (2s-1)}j_{\xi(2s)}}$ is exactly the pfaffian of $M_{I}^{\Theta}$,
where $I=(i_1,i_2,\dots,i_{2p})$, we have
\begin{eqnarray}\label{range of trace}
  \tau_*(\K_0(\mathcal{A}_{\Theta}))=\mathbb{Z}+\sum\limits_{0<|I|\leq n}{\rm pf}(M_{I}^{\Theta})\mathbb{Z},\, {\rm for}\,{\rm all}\, \tau\in T(\mathcal{A}_{\Theta}),
\end{eqnarray}
where $|I|=2p$ for $I=(i_1,i_2,\dots,i_{2p}).$


\begin{definition}[Definition 1 of \cite{Farsi-Watling-MA}]\label{df totally irrational}
  We say that $\Theta\in\mathcal{T}_n$ is {\it totally irrational} if ${\rm exp}_{\wedge}$ is an injective map from $\Lambda^{even}\mathbb{Z}^n$ to $\mathbb{R}$ (cf. Section 6 and 7 of \cite{Rieffel-1988}).
\end{definition}

It is clear from Elliott's work that $(\tau_{\Theta})_*$ is injective if and only if $\Theta$ is totally irrational.  Now the range of exp$_{\wedge}$ is given by
$$\mathbb{Z}+\sum\limits_{0<|I|\leq n}{\rm pf}(M_{I}^{\Theta})\mathbb{Z},$$ where $|I|=2p$ for $I=(i_1,i_2,\dots,i_{2p}).$
Thus $\Theta$ is totally irrational if and only if $1,$ $\mathrm{pf}(M_{I}^{\Theta}),$ $0<|I|\leq n$ are rationally independent.

 Note that if $\Theta$ is totally irrational, $\Theta$ is also nondegenerate and $\mathcal{A}_{\Theta}$ is a simple $C^*$-algebra by Theorem \ref{AT and K}.

For $$\Theta=\left(\begin{matrix}
                 \Theta_{11} &  \Theta_{12} \\
               \Theta_{21}& \Theta_{22}
                 \end{matrix}\right)= \left(\begin{matrix}
               \Theta_{11} & \Theta_{12} \\
                -\Theta_{12}^t &\Theta_{22}
                 \end{matrix}\right)\in\mathcal{T}_n,$$
                 where $n=2l,$ $l>1,$
                 and
                 \begin{eqnarray}
  \Theta_{11}=\left(\begin{matrix}
              0 &  \theta_{12} \\
               \theta_{21}& 0
                 \end{matrix}\right)=\left(\begin{matrix}
                 0 &  \theta_{12} \\
              - \theta_{12}& 0
                 \end{matrix}\right)\in\mathcal{T}_2\nonumber
\end{eqnarray} is invertible $2\times2$ matrix,
                 $$\Theta_{22}=\left(\begin{matrix}
                 0 &  \theta_{34}&\cdots&\theta_{3n} \\
              - \theta_{34}& 0&\cdots&\theta_{4n}\\
              \vdots&\vdots&\ddots&\vdots\\
              -\theta_{3n}&-\theta_{4n}&\cdots&0
                 \end{matrix}\right)\in\mathcal{T}_{n-2},$$
                  $$\Theta_{12}=\left(\begin{matrix}
                \theta_{13} &  \theta_{14}&\cdots&\theta_{1n} \\
               \theta_{23}& \theta_{24}&\cdots&\theta_{2n}
                 \end{matrix}\right),$$
     $$\Theta_{21}=\left(\begin{matrix}
                -\theta_{13} & - \theta_{23} \\
                -\theta_{14} & - \theta_{24} \\
                \vdots&\vdots\\
             -\theta_{1n}& -\theta_{2n}
                 \end{matrix}\right),$$
we have  $ \Theta_{11}^{-1}=\left(\begin{matrix}
                 0 &  -\frac{1}{\theta_{12}} \\
              \frac{1}{\theta_{12}}& 0
                 \end{matrix}\right)\in\mathcal{T}_2,$ and
$$\Theta_{22}-\Theta_{21}\Theta_{11}^{-1}\Theta_{12}=\left(\begin{matrix}
                 0 &              \theta_{34}-\frac{-\theta_{23}\theta_{14}+\theta_{13}\theta_{24}}{\theta_{12}}&\cdots&\theta_{3n}-\frac{-\theta_{23}\theta_{1n}+\theta_{13}\theta_{2n}}{\theta_{12}} \\
              - \theta_{34}+\frac{-\theta_{23}\theta_{14}+\theta_{13}\theta_{24}}{\theta_{12}}& 0&\cdots&\theta_{4n}-\frac{-\theta_{24}\theta_{1n}+\theta_{14}\theta_{2n}}{\theta_{12}}\\
              \vdots&\vdots&\ddots&\vdots\\
              -\theta_{3n}+\frac{-\theta_{23}\theta_{1n}+\theta_{13}\theta_{2n}}{\theta_{12}}&-\theta_{4n}+\frac{-\theta_{24}\theta_{1n}+\theta_{14}\theta_{2n}}{\theta_{12}}&\cdots&0
                 \end{matrix}\right).$$
                 Thus we have from (\ref{f1}) and (\ref{f2}) that
                \begin{eqnarray}\label{F(theta)}
              \Theta_{22}-\Theta_{21}\Theta_{11}^{-1}\Theta_{12}=\left(\begin{matrix}
                0 &  \frac{{\rm pf}^{\Theta}_{(1,2,3,4)}}{\theta_{12}}&\cdots&\frac{{\rm pf}^{\Theta}_{(1,2,3,n)}}{\theta_{12}} \\
              - \frac{{\rm pf}^{\Theta}_{(1,2,3,4)}}{\theta_{12}}& 0&\cdots&\frac{{\rm pf}^{\Theta}_{(1,2,4,n)}}{\theta_{12}}\\
              \vdots&\vdots&\ddots&\vdots\\
              -\frac{{\rm pf}^{\Theta}_{(1,2,3,n)}}{\theta_{12}}&-\frac{{\rm pf}^{\Theta}_{(1,2,4,n)}}{\theta_{12}}&\cdots&0
                 \end{matrix}\right).
                 \end{eqnarray}

The following lemma reveals the relationship between the pfaffian of a matrix and the pfaffian of its partitioned matrix.
\begin{lemma}\label{submatrix,x}
For any integer $n\geq 2,$
let
  $$\Theta=\left(\begin{matrix}
                  \Theta_{11} &  \Theta_{12} \\
                 \Theta_{21}& \Theta_{22}
                 \end{matrix}
\right)=\left(\begin{matrix}
                  \Theta_{11} &  \Theta_{12} \\
                 -\Theta_{12}^t& \Theta_{22}
                 \end{matrix}
\right)\in \mathcal{T}_n,$$
where $\Theta_{11}$ is invertible $2\times 2$ matrix, one has
\begin{eqnarray}\label{=2}
{\rm pf}(\Theta_{11}){\rm pf}(M_{I'}^{\Theta_{22}-\Theta_{21}\Theta_{11}^{-1}\Theta_{12}})
={\rm pf}(M_{I}^{\Theta}),
\end{eqnarray}
where $|I|=2l+2$ for $I=(1,2,i_1+2,i_2+2,\dots,i_{2l}+2),$ $I'=(i_1,i_2,\dots,i_{2l})$.
In particular, when $n$ is an even number, we have
\begin{eqnarray}\label{pf1}
  {\rm pf}(\Theta)={\rm pf}(\Theta_{11}){\rm pf}(\Theta_{22}+\Theta^t_{12}\Theta^{-1}_{11}\Theta_{12})={\rm pf}(\Theta_{11}){\rm pf}(\Theta_{22}-\Theta_{21}\Theta_{11}^{-1}\Theta_{12}).
\end{eqnarray}
\end{lemma}
\begin{proof}
By using the following formula:
 $$\left(\begin{matrix}
                  \Theta_{11} &  0 \\
                0&M_{I'}^{\Theta_{22}+\Theta_{12}^t \Theta_{11}^{-1}\Theta_{12}}
                 \end{matrix}
\right)=\left(\begin{matrix}
                   {\rm id} & 0 \\
                (M_{I,(1,2)}^{\Theta})^t\Theta_{11}^{-1}&  {\rm id}
                 \end{matrix}\right)
                 \left(\begin{matrix}
                  \Theta_{11} &  M^{\Theta}_{(1,2),I} \\
                 -(M_{I,(1,2)}^{\Theta})^t& M_{I'}^{\Theta_{22}}
                 \end{matrix}
\right)
\left(\begin{matrix}
                  {\rm id} &  -\Theta_{11}^{-1}M_{(1,2),I}^{\Theta} \\
                0&  {\rm id}
                 \end{matrix}\right),$$
                 and using (\ref{eq:pfdet}), we have
\begin{eqnarray}\label{=1}
{\rm pf}(\Theta_{11}){\rm pf}(M_{I'}^{\Theta_{22}-\Theta_{21}\Theta_{11}^{-1}\Theta_{12}})
={\rm pf}(M_{I}^{\Theta}).
\end{eqnarray}
In particular, when $n$ is an even number,  by taking $I=(1,2,\dots,n)$ in (\ref{=1}), we have
\begin{eqnarray}\label{p1}
  {\rm pf}(\Theta)={\rm pf}(\Theta_{11}){\rm pf}(\Theta_{22}+\Theta^t_{12}\Theta^{-1}_{11}\Theta_{12})={\rm pf}(\Theta_{11}){\rm pf}(\Theta_{22}-\Theta_{21}\Theta_{11}^{-1}\Theta_{12}).\nonumber
\end{eqnarray}
\end{proof}

For the convenience of symbols, we give the following definition:
\begin{definition}
  For $$\Theta=\left(\begin{matrix}
                 \Theta_{11} &  \Theta_{12} \\
               \Theta_{21}& \Theta_{22}
                 \end{matrix}\right)\in\mathcal{T}_n,$$
                 where $\Theta_{11}$ is an invertible $2\times2$ matrix,
                 and $n=2l,$ $l>1,$ define
                 \begin{eqnarray}\label{de F}
                 F(\Theta)=\Theta_{22}+\Theta_{12}^t \Theta_{11}^{-1}\Theta_{12}=
                 \Theta_{22}-\Theta_{21}\Theta_{11}^{-1}\Theta_{12}\in \mathcal{T}_{n-2}.
                 \end{eqnarray}
\end{definition}
Hence from (\ref{pf1}), we have
\begin{eqnarray}\label{q1}
{\rm pf}(\Theta)={\rm pf}(\Theta_{11}){\rm pf}(F(\Theta)).
\end{eqnarray}

If $m$ be an integer less than $l$, we denote by $F^m$ the composition of $F$ taken $m$ times (when makes sense) and $F^0(\Theta):=\Theta.$
Note that $F$ is defined for a $\Theta$ such that $\Theta_{11}$ is invertible, but still $F^m$ may not make sense. The following lemma tells us  when $F^m(\Theta)$ is well defined.
\begin{lemma}\label{h1}
   Let $\Theta\in \mathcal{T}_n$ with $n=2l$ for $l>1.$ If ${\rm pf}^{\Theta}_{(1,2,\dots,2s)}\neq 0$ for $s=1,2,\dots,l-1,$ then
   $F^m(\Theta)$ is well defined for $m=1,2,\dots,l-1.$
\end{lemma}
\begin{proof}
  To show that $F^m(\Theta)$ is well defined for $m=1,2,\dots,l-1,$ it is enough to show that $F^m(\Theta)_{11}$ is invertible for all $m=0,1,2,\dots,l-2.$ Since ${\rm pf}^{\Theta}_{(1,2)}=\theta_{12}\neq0$. Now, by (\ref{F(theta)}) we have
  \begin{eqnarray}\label{F(theta1)}
             F(\Theta)=\Theta_{22}-\Theta_{21}\Theta_{11}^{-1}\Theta_{12}=\left(\begin{matrix}
                0 &  \frac{{\rm pf}^{\Theta}_{(1,2,3,4)}}{\theta_{12}}&\cdots&\frac{{\rm pf}^{\Theta}_{(1,2,3,n)}}{\theta_{12}} \\
              - \frac{{\rm pf}^{\Theta}_{(1,2,3,4)}}{\theta_{12}}& 0&\cdots&\frac{{\rm pf}^{\Theta}_{(1,2,4,n)}}{\theta_{12}}\\
              \vdots&\vdots&\ddots&\vdots\\
              -\frac{{\rm pf}^{\Theta}_{(1,2,3,n)}}{\theta_{12}}&-\frac{{\rm pf}^{\Theta}_{(1,2,4,n)}}{\theta_{12}}&\cdots&0
                 \end{matrix}\right).\nonumber
                 \end{eqnarray}
               Since ${\rm pf}^{\Theta}_{(1,2,3,4)}\neq0,$ we have that $F(\Theta)_{11}$ is also invertible as ${\rm pf}(F(\Theta)_{11})=\frac{{\rm pf}^{\Theta}_{(1,2,3,4)}}{\theta_{12}}\neq 0.$ Now we want to use induction on $m.$
              Assume $F^{j}(\Theta)_{11}$ is invertible for all $j=0,1,2,\dots,m-1$ for some $m\leq l-2,$ then we must prove that $F^m(\Theta)_{11}$ is invertible. Let's use $M^{\Theta}_{k}$ to express $M^{\Theta}_{(1,2,\dots,k)}$ briefly, for $k\leq n.$ Note that, from (\ref{F(theta)}), we have the following identity
              \begin{eqnarray}\label{F1}
              F^{j}(\Theta)_{11}=F^{j}(M^\Theta_{n-2})_{11}
              \end{eqnarray}
              for all $j=0,1,2,\dots,l-2.$
              It follows that there will exist some even number $p,$ depending on $m,$ such that
  \begin{eqnarray}\label{F2}
  F^{m}(M^\Theta_{n-2})_{11}=F^{m}(M^\Theta_{n-p})_{11}=F^{m}(M^\Theta_{n-p}).
  \end{eqnarray}
  In particular, the matrix $F^{m}(M^\Theta_{n-p})$ will be of size $2\times 2.$
  Since by induction hypothesis $F^{j}(\Theta)_{11}$ is invertible for all $j=0,1,2,\ldots,m-1$, $F^{j}(M^\Theta_{n-p})_{11}=F^{j}(\Theta)_{11}$ are also invertible by using (\ref{F1}) for $j=0,1,2,\dots,m-1.$

  Now for the matrix $F^{m-1}(M^\Theta_{n-p}),$ since $F^{m-1}(M^\Theta_{n-p})_{11}$ is invertible, we have from (\ref{pf1})
  \begin{eqnarray}\label{F3}
  {\rm pf}(F^{m-1}(M^\Theta_{n-p}))={\rm pf}(F^{m-1}(M^\Theta_{n-p})_{11}){\rm pf}(F^{m}(M^\Theta_{n-p})).
  \end{eqnarray}
  Since $ {\rm pf}(F^{m-1}(M^\Theta_{n-p}))$ is non-zero, which follows from (\ref{F3})
  by induction and using the fact that ${\rm pf}(M^\Theta_{n-p})$ is non-zero, we have ${\rm pf}(F^{m}(M^\Theta_{n-p}))$ is also non-zero. But
  $F^m(M^\Theta_{n-p})$ is $F^{m}(\Theta)_{11}$. Hence $F^{m}(\Theta)_{11}$ is invertible.

\end{proof}
\begin{remark}
  In fact, in the above lemma, if $\Theta$ is totally irrational, we can get a stronger result, i.e. ${\rm pf}(F^m(\Theta)_{11})$ is irrational for $m=0,1,\dots,l-1.$
   For this let us consider the equation
  \begin{eqnarray}\label{F4}
  {\rm pf}(F^m(\Theta))={\rm pf}(F^m(\Theta)_{11}){\rm pf}(F^{m+1}(\Theta)),\nonumber
  \end{eqnarray}
  thus \begin{eqnarray}\label{F5}
        {\rm pf}(F^{m+1}(\Theta))=\frac{{\rm pf}(F^m(\Theta))}{{\rm pf}(F^m(\Theta)_{11})}.
       \end{eqnarray}
 Combining (\ref{F2}) with (\ref{F5}), we get
  $${\rm pf}(F^{m+1}(\Theta))=\frac{{\rm pf}(F^m(\Theta))}{{\rm pf}(F^m(M^\Theta_{n-p}))}.$$
  Now using  (\ref{q1}) repeatedly for all $F^{m+1}(\Theta)$ and $F^m(M^\Theta_{n-p}),$
  $${\rm pf}(F^{m+1}(\Theta))=\frac{{\rm pf}(\Theta)}{{\rm pf}(M^\Theta_{n-p})}.$$
  Similarly,
  $${\rm pf}(F^{m}(\Theta))=\frac{{\rm pf}(\Theta)}{{\rm pf}(M^\Theta_{n-p-2})}.$$
  So we have
  \begin{eqnarray}\label{a1}
  {\rm pf}(F^{m}(\Theta)_{11})=\frac{{\rm pf}(M^\Theta_{n-p})}{{\rm pf}(M^\Theta_{n-p-2})}
  \end{eqnarray}
 which is irrational since $\Theta$ is totally irrational.
\end{remark}

The following lemma
gives an explicit expression of $F^m(\Theta)$ for $m=1,2,\dots,l-1.$
\begin{lemma}\label{h2}
  Let $\Theta\in \mathcal{T}_n$ with $n=2l$ for $l>1.$ If ${\rm pf}^{\Theta}_{(1,2,\dots,2s)}\neq 0$ for $s=1,2,\dots,l-1,$ then
  \begin{eqnarray}\label{F^ii}
  F^m(\Theta)=\left(
  \begin{array}{cccc}
  0&\frac{{\rm pf}^{\Theta}_{(1,2,\dots,n-p-1,n-p)}}{{\rm pf}^{\Theta}_{(1,2,\dots,n-p-2)}}&\cdots&\frac{{\rm pf}^{\Theta}_{(1,2,\dots,n-p-1,n)}}{{\rm pf}^{\Theta}_{(1,2,\dots,n-p-2)}}\\
  -\frac{{\rm pf}^{\Theta}_{(1,2,\dots,n-p-1,n-p)}}{{\rm pf}^{\Theta}_{(1,2,\dots,n-p-2)}}&0&\cdots& \frac{{\rm pf}^{\Theta}_{(1,2,\dots,n-p,n)}}{{\rm pf}^{\Theta}_{(1,2,\dots,n-p-2)}}\\
  \vdots&\vdots&\ddots&\vdots\\
  -\frac{{\rm pf}^{\Theta}_{(1,2,\dots,n-p-1,n)}}{{\rm pf}^{\Theta}_{(1,2,\dots,n-p-2)}}&
  -\frac{{\rm pf}^{\Theta}_{(1,2,\dots,n-p,n)}}{{\rm pf}^{\Theta}_{(1,2,\dots,n-p-2)}}&
  \ldots&0
  \end{array}\right)
  \end{eqnarray}
  and in particular,
  \begin{eqnarray}\label{g1}
  F^m(\Theta)_{jk}=\frac{{\rm pf}^{\Theta}_{(1,2,\dots,s',s'+j,s'+k)}}{{\rm pf}^{\Theta}_{(1,2,\dots,s')}},\quad p=n-2m-2,\quad s'=n-p-2=2m,
  \end{eqnarray}
  for $m=1,\dots,l-1.$
\end{lemma}
\begin{proof}
  We prove that (\ref{F^ii}) holds by using induction on $m.$ Since ${\rm pf}^{\Theta}_{(1,2)}=\theta_{12}\neq 0,$ we have from (\ref{F(theta)}) that
  \begin{eqnarray}\label{F(theta2)}
             F(\Theta)=\Theta_{22}-\Theta_{21}\Theta_{11}^{-1}\Theta_{12}=\left(\begin{matrix}
                0 &  \frac{{\rm pf}^{\Theta}_{(1,2,3,4)}}{\theta_{12}}&\cdots&\frac{{\rm pf}^{\Theta}_{(1,2,3,n)}}{\theta_{12}} \\
              - \frac{{\rm pf}^{\Theta}_{(1,2,3,4)}}{\theta_{12}}& 0&\cdots&\frac{{\rm pf}^{\Theta}_{(1,2,4,n)}}{\theta_{12}}\\
              \vdots&\vdots&\ddots&\vdots\\
              -\frac{{\rm pf}^{\Theta}_{(1,2,3,n)}}{\theta_{12}}&-\frac{{\rm pf}^{\Theta}_{(1,2,4,n)}}{\theta_{12}}&\cdots&0
                 \end{matrix}\right).\nonumber
                 \end{eqnarray}
 So  (\ref{F^ii}) holds for $m=1.$ Now assume that (\ref{F^ii}) holds for $m-1$. Then we must prove that the same holds for $m.$ Hence we have
 \begin{eqnarray}\label{g2}
 F^{m-1}(\Theta)_{jk}=\frac{{\rm pf}^{\Theta}_{(1,2,\dots,s,s+j,s+k)}}{{\rm pf}^{\Theta}_{(1,2,\dots,s)}},\quad s=2(m-1),\nonumber
 \end{eqnarray}
  and we have to show that
  \begin{eqnarray}\label{g3}
 F^{m}(\Theta)_{j'k'}=\frac{{\rm pf}^{\Theta}_{(1,2,\dots,s',s'+j',s+k')}}{{\rm pf}^{\Theta}_{(1,2,\dots,s')}},\quad s'=2m.\nonumber
 \end{eqnarray}

Let $\sigma_n$ denote the permutation in symmetric group $\mathcal{S}_n$ of degree $n$ which interchanges $1$ and $j'$, and $2$ and $k'$. Also let $\sigma_n(\Theta)$ denote the matrix obtained from conjugating $\Theta$ with the permutation matrix corresponding to
the permutation $\sigma_n.$  Now $F^m(\Theta)$ is a $(n-2m)\times (n-2m)$ matrix. Consider $\sigma_{n-2m}(F^m(\Theta)).$ It is clear that $\sigma_{n-2m}(F^m(\Theta))_{12}=F^m(\Theta)_{j'k'}.$ So we need to determine
$\sigma_{n-2m}(F^m(\Theta))_{12}.$  Now $\sigma_{n-2}(F(\Theta))_{12}=\frac{{\rm pf}^{\Theta}_{(1,2,2+j',2+k')}}{\theta_{12}}.$ If we denote the permutation in $\mathcal{S}_n$ which interchanges $3$ and $2+j'$, and $4$ and $2+k'$ by $\sigma'_n,$ then $F(\sigma'_n(\Theta))_{12}=\frac{{\rm pf}^{\Theta}_{(1,2,2+j',2+k')}}{\theta_{12}}.$
Hence we get
\begin{eqnarray}\label{g4}
\sigma_{n-2}(F(\Theta))_{12}=F(\sigma'_n(\Theta))_{12}.
\end{eqnarray}
Putting $\Theta=F^{m-1}(\Theta)$ in the above equation (\ref{g4}), we get
\begin{eqnarray}\label{g5}
\sigma_{n-2m}(F^m(\Theta))_{12}=F(\sigma'_{n-2(m-1)}(F^{m-1}(\Theta)))_{12}.\end{eqnarray}
Now by induction hypothesis,
\begin{eqnarray}\label{g6}
 F^{m-1}(\Theta)_{jk}=\frac{{\rm pf}^{\Theta}_{(1,2,\dots,s,s+j,s+k})}{{\rm pf}^{\Theta}_{(1,2,\dots,s)}},\quad s=2(m-1).
 \end{eqnarray}
Denoting $F^{m-1}(\Theta)$ by $\widetilde{\Theta},$ we have \begin{eqnarray}F(\widetilde{\Theta})_{12}=\frac{{\rm pf}^{\widetilde{\Theta}}_{(1,2,3,4)}}{\widetilde{\Theta}_{12}}.\nonumber
\end{eqnarray}
But from (\ref{a1}) we have
\begin{eqnarray}F(\widetilde{\Theta})_{12}=\frac{{\rm pf}^{\widetilde{\Theta}}_{(1,2,3,4)}}{\widetilde{\Theta}_{12}}=\frac{{\rm pf}(M^\Theta_{n-p})}{{\rm pf}(M^\Theta_{n-p-2})}.\nonumber
\end{eqnarray}
 Now let us look at the permutation
$\rho$ in $\mathcal{S}_n$ which just interchanges $s'+1$ and $s'+j'$, and $s'+2$ and $s'+k'$, and consider ${\rm pf}(M_{s'}^{\rho(\Theta)})={\rm pf}(M_{s'}^{\Theta})\neq0$. So from (\ref{a1}),
\begin{eqnarray}F^m(\rho(\Theta))_{12}=\frac{{\rm pf}(M^{\rho(\Theta)}_{n-p})}{{\rm pf}(M^{\rho(\Theta)}_{n-p-2})}=\frac{{\rm pf}(M^{\rho(\Theta)}_{s'+2})}{{\rm pf}(M^{\rho(\Theta)}_{s'})},\nonumber
\end{eqnarray} where $s'=n-p-2=2m$. But
\begin{eqnarray}\label{g7}
F^m(\rho(\Theta))_{12}=\frac{{\rm pf}(M^{\rho(\Theta)}_{s'+2})}{{\rm pf}(M^{\rho(\Theta)}_{s'})}
=\frac{{\rm pf}^{\Theta}_{(1,2,\dots,s',s'+j',s'+k')}}{{\rm pf}^{\Theta}_{(1,2,\dots,s')}},\quad s'=2m.
\end{eqnarray}
But an easy observation from (\ref{g6}) shows that
\begin{eqnarray}\label{g8}
F^m(\rho(\Theta))_{12}=F(F^{m-1}(\rho(\Theta)))_{12}=
F(\sigma'_{n-2(m-1)}(F^{m-1}(\Theta)))_{12}.
\end{eqnarray}
Combining (\ref{g5}), (\ref{g7}) with (\ref{g8}),
we have
\begin{eqnarray*}
\sigma_{n-2m}(F^m(\Theta))_{12}=F(\sigma'_{n-2(m-1)}(F^{m-1}(\Theta)))_{12}=F^m(\rho(\Theta))_{12}
=\frac{{\rm pf}^{\Theta}_{(1,2,\dots,s',s'+j',s'+k')}}{{\rm pf}^{\Theta}_{(1,2,\dots,s')}}, \quad s'=2m.
\end{eqnarray*}
Hence we prove the conclusion.
\end{proof}

\begin{remark}\label{F()simple}
By using the above lemma, we can know more about the properties of $F^m(\Theta)$ and $\mathcal{A}_{F^{m}(\Theta)}$ when $\Theta\in\mathcal{T}_n$
is totally irrational with $n=2l\geq2$.
For example,  the entries (above the diagonal) of $F^m(\Theta)$ are all irrational and independent over $\mathbb{Q}$ for $m=0,1,\dots,l-1.$  Hence $\mathcal{A}_{F^{m}(\Theta)}$ is a simple $C^*$-algebra and has a unique tracial state $\tau_{F^{m}(\Theta)}$ for $m=0,1,\dots,l-1$.
This is because when $\Theta$ is totally irrational, we have that ${\rm pf}^{\Theta}_{(1,2,\dots,2s)}\neq 0$ for $s=1,2,\dots,l-1$. By Lemma \ref{h1} and Lemma \ref{h2}, we know that $ F^m(\Theta)$ is well-defined and
 \begin{eqnarray}\label{F^i}
  F^m(\Theta)=\left(
  \begin{array}{cccc}
  0&\frac{{\rm pf}^{\Theta}_{(1,2,\dots,n-p-1,n-p)}}{{\rm pf}^{\Theta}_{(1,2,\dots,n-p-2)}}&\cdots&\frac{{\rm pf}^{\Theta}_{(1,2,\dots,n-p-1,n)}}{{\rm pf}^{\Theta}_{(1,2,\dots,n-p-2)}}\\
  -\frac{{\rm pf}^{\Theta}_{(1,2,\dots,n-p-1,n-p)}}{{\rm pf}^{\Theta}_{(1,2,\dots,n-p-2)}}&0&\cdots& \frac{{\rm pf}^{\Theta}_{(1,2,\dots,n-p,n)}}{{\rm pf}^{\Theta}_{(1,2,\dots,n-p-2)}}\\
  \vdots&\vdots&\ddots&\vdots\\
  -\frac{{\rm pf}^{\Theta}_{(1,2,\dots,n-p-1,n)}}{{\rm pf}^{\Theta}_{(1,2,\dots,n-p-2)}}&
  -\frac{{\rm pf}^{\Theta}_{(1,2,\dots,n-p,n)}}{{\rm pf}^{\Theta}_{(1,2,\dots,n-p-2)}}&
  \ldots&0
  \end{array}\right)\nonumber
  \end{eqnarray}
  and in particular,
  \begin{eqnarray}\label{g1}
  F^m(\Theta)_{jk}=\frac{{\rm pf}^{\Theta}_{(1,2,\dots,s',s'+j,s'+k)}}{{\rm pf}^{\Theta}_{(1,2,\dots,s')}},\quad p=n-2m-2,\quad s'=n-p-2=2m,
  \end{eqnarray}
  for $m=1,\dots,l-1.$
Now since $\Theta$ is totally irrational, the numbers ${\rm pf}^{\Theta}_{(1,2,\dots,n-p-2,j,k)}$, $n-p-1\leq j<k\leq n$  along with ${\rm pf}^{\Theta}_{(1,2,\dots,n-p-2)},$ are irrational and independent over $\mathbb{Q}$. This means that the numbers $\frac{{\rm pf}^{\Theta}_{(1,2,\dots,n-p-2,j,k)}}{{\rm pf}^{\Theta}_{(1,2,\dots,n-p-2)}}$ are irrational numbers. Next, to show that these numbers are independent over $\mathbb{Q}$, let us write
$$\sum\limits_{n-p-1\leq j<k\leq n}c_{j,k} \frac{{\rm pf}^{\Theta}_{(1,2,\dots,n-p-2,j,k)}}{{\rm pf}^{\Theta}_{(1,2,\dots,n-p-2)}}=0,\quad c_{j,k}\in \mathbb{Q}.$$
This implies
$$\sum\limits_{n-p-1\leq j<k\leq n}c_{j,k}{\rm pf}^{\Theta}_{(1,2,\dots,n-p-2,j,k)}=0,\quad c_{j,k}\in \mathbb{Q}.$$
But since the numbers ${\rm pf}^{\Theta}_{(1,2,\dots,n-p-2,j,k)},$ $n-p-1\leq j<k\leq n$ are independent over $\mathbb{Q},$ $c_{j,k}$'s are all zero. Hence the numbers $\frac{{\rm pf}^{\Theta}_{(1,2,\dots,n-p-2,j,k)}}{{\rm pf}^{\Theta}_{(1,2,\dots,n-p-2)}}$, $n-p-1\leq j<k\leq n$ are rationally
independent. So we have shown that the entries (above the diagonal) of $F^m(\Theta)$ are
irrational and rationally independent. Now, from this, it is easy to see that
$F^m(\Theta)$  is non-degenerate in the sense of Definition~\ref{def:nondege}. Hence $\mathcal{A}_{F^m(\Theta)}$ is simple and has a unique tracial state $\tau_{F^m(\Theta)}$ for $m=0,1,\dots,l-1$.
\end{remark}

In the following we shall construct Rieffel-type projections for the
higher dimensional noncommutative tori.

We fix a number $p$ with $1\leq p\leq n\slash 2$ and let $q\geq0$ be an integer such that $n=2p+q.$ Let us write $\Theta\in \mathcal{T}_n$ as
$$\left(\begin{matrix}
                  \Theta_{11} & \Theta_{12} \\
                  \Theta_{21} & \Theta_{22}
                \end{matrix}
\right),$$
partitioned into four sub-matrices $\Theta_{11},\Theta_{12},\Theta_{21},\Theta_{22},$ and assume $\Theta_{11}$ to be an invertible $2p\times 2p$ matrix. Let
$$\Theta'=\left(\begin{matrix}
                  \Theta_{11}^{-1} & -\Theta_{11}^{-1}\Theta_{12} \\
                  \Theta_{21}\Theta_{11}^{-1} & \Theta_{22}-\Theta_{21}\Theta_{11}^{-1}\Theta_{12}
                \end{matrix}
\right).$$

Set $\mathcal{A}=\mathcal{A}_{\Theta}$ and $\mathcal{B}=\mathcal{A}_{\Theta'}.$
Let $\mathcal{M}$ be the group $\mathbb{R}^{p}\times \mathbb{Z}^q,$ $G=\mathcal{M}\times\widehat{\mathcal{M}}$ and
$\langle\cdot,\cdot\rangle$ be the natural pairing between $\mathcal{M}$ and its dual group $\widehat{\mathcal{M}}$ (our notation does not distinguish between the pairing of a group and its dual group, and the standard inner product on a linear space). Consider the Schwartz space $\mathcal{E}^{\infty}=\mathcal{S}(\mathcal{M})$ consisting of smooth and rapidly decreasing complex-valued functions on $\mathcal{M}.$

Denote by $\mathcal{A}^{\infty}=\mathcal{A}_{\Theta}^{\infty}$ and $\mathcal{B}=\mathcal{A}_{\Theta'}^{\infty}$ the dense sub-algebras of $\mathcal{A}$ and $\mathcal{B}$, respectively, consisting of formal series with rapidly decaying coefficients. Let us consider the following $(2p+2q)\times(2p+q)$
real valued matrix:
\begin{eqnarray}T=\left(\begin{matrix}
                  T_{11} & 0 \\
                 0 & {\rm id}_q\\
                 T_{31}&T_{32}
                \end{matrix}
\right),\nonumber
\end{eqnarray}
where $T_{11}$ is an invertible matrix such that $T^{t}_{11}J_0 T_{11}=\Theta_{11},$ $J_0=\left(\begin{matrix}
                  0 & {\rm id}_{p} \\
                 -{\rm id}_{p}& 0
                 \end{matrix}
\right),$
$T_{31}=\Theta_{21}$ and $T_{32}$ is any $q\times q$ matrix such that
$\Theta_{22}=T_{32}-T_{32}^t.$ For our purposes, we take $T_{32}=\Theta_{22}\slash 2.$

We also define the following $(2p+2q)\times (2p+q)$ real valued matrix:
$$S=\left(\begin{matrix}
                  J_0(T_{11}^t)^{-1} &  -J_0(T_{11}^t)^{-1} T_{31}^t \\
                 0& {\rm id}_q\\
                 0& T_{32}^t
                 \end{matrix}
\right).$$
Let
$$J=\left(\begin{matrix}
                  J_0 &  0&0 \\
                 0& 0&{\rm id}_q\\
                 0& -{\rm id}_q&0
                 \end{matrix}
\right)$$
and $J'$ be the matrix obtained from $J$ by replacing the negative entries of it by zeroes.
Note that $T$ and $S$ can be thought as maps from $(\mathbb{R}^n)^*$
to $\mathbb{R}^{p}\times (\mathbb{R}^p)^*\times \mathbb{R}^{q}\times (\mathbb{R}^q)^*$ (see the definition of an embedding map in Definition 2.1 of \cite{Li-2004}), and $S(\mathbb{Z}^n),T(\mathbb{Z}^n)\subset \mathbb{R}^{p}\times (\mathbb{R}^p)^*\times \mathbb{Z}^{q}\times (\mathbb{R}^q)^*$.
Then we can think of $S(\mathbb{Z}^n),T(\mathbb{Z}^n)$ as in $G$ via composing
$S|_{\mathbb{Z}^n},T|_{\mathbb{Z}^n}$ with the natural covering map
$\mathbb{R}^{p}\times (\mathbb{R}^p)^*\times \mathbb{Z}^{q}\times (\mathbb{R}^q)^*\rightarrow G.$ Let $P'$ and $P''$ be the canonical projections of
$G$ to $\mathcal{M}$ and $\widehat{\mathcal{M}},$ respectively, and let
$$T'=P'\circ T,\quad T''=P''\circ T,\quad S'=P'\circ S,\quad S''=P''\circ S.$$
Then the following set of formulas define a $\mathcal{B}^{\infty}$-$\mathcal{A}^{\infty}$
bimodule structure on $\mathcal{E}^{\infty}$:
\begin{empheq}[left=\empheqlbrace]{align}
  &(fU_{l}^{\Theta})(x)=e^{2\pi i\langle -T(l),J'T(l)\slash2\rangle}\langle x,T''(l)\rangle f(x-T'(l)), \label{eq:module1}\\
&\langle f,g\rangle_{\mathcal{A}^{\infty}}(l)=e^{2\pi i\langle -T(l),J'T(l)\slash2\rangle}\int_{G}\langle x,-T''(l)\rangle g(x+T'(l))\overline{f(x)}dx,\label{eq:module2}\\
&(U_{l}^{\Theta'}f)(x)=e^{2\pi i\langle -S(l),J'S(l)\slash2\rangle}\langle x,S''(l)\rangle f(x-S'(l)), \label{eq:module3}\\
& _{\mathcal{B}^{\infty}}\langle f,g\rangle(l)=Ke^{2\pi i\langle S(l),J'S(l)\slash2\rangle}\int_{G}\langle x,S''(l)\rangle \overline{g(x+S'(l))}f(x)dx, \label{eq:module4}
\end{empheq}
where $U_l^{\Theta},U_{l}^{\Theta'}$ denote the canonical unitaries with respect to the group element $l\in \mathbb{Z}^n$ in $\mathcal{A}^{\infty}$ and $\mathcal{B}^{\infty}$, respectively, and $K$ is a positive constant. See Proposition 2.2 in \cite{Li-2004} for the following well-known result.

\begin{theorem}
  The smooth module $\mathcal{E}^{\infty}$, with above structures, is an $\mathcal{B}^{\infty}$-$\mathcal{A}^{\infty}$ Morita equivalence bimodule which can be extended to a strong Morita equivalence between $\mathcal{B}$ and $\mathcal{A}$.
\end{theorem}

Let $\mathcal{E}$ denote the completion of $\mathcal{E}^{\infty}$ with respect to
the $C^*$-valued inner products given above. This becomes a finitely generated projective module over $\mathcal{A}$. This class is called the {\it Bott class}. The Bott class  contributes to a part of a generating set of the $\K_0$-group of $\mathcal{A}_{\Theta}$.
Next, we will construct a specific projection to describe the Bott class.

\begin{theorem}\label{even n projection}
  For any even number $n=2l\geq2,$ if $\Theta\in \mathcal{T}_n$ satisfy ${\rm pf}(F^j(\Theta)_{11})\in (0,1)$ for $j=0,1,\dots,l-2$, and ${\rm pf}(F^{l-1}(\Theta)_{11})\notin \mathbb{Z}$, then
  there exist a (Rieffel-type) projection $p_m$ inside $\mathcal{A}_{F^m(\Theta)}$ such that
  \begin{eqnarray}\tau_{F^m(\Theta)}(p_m)={\rm pf}(F^m(\Theta))+k_{(1,2)}^{(m)}{\rm pf}(M^{F^m(\Theta)}_{(1,2)})+\cdots+k_{(1,2,\dots,n-2(m+1))}^{(m)}{\rm pf}(M^{F^m(\Theta)}_{(1,2,\dots,n-2(m+1))})+k^{(m)}_0\nonumber
   \end{eqnarray}
   for some $k_I^{(m)},k^{(m)}_0\in \mathbb{Z}$ and $m=0,1,2,\dots,l-1,$ where $I=(1,2,\dots,n-2s)$ and $s=m+1,m+2,\dots,l-1,$ $\tau_{F^m(\Theta)}$ is the canonical tracial state on $\mathcal{A}_{F^m(\Theta)}.$
\end{theorem}
\begin{proof}
We do the proof using recursion on $m.$ For $m=l-1$, $F^{l-1}(\Theta)$ is a $2\times2$ matrix and  ${\rm pf}(F^{l-1}(\Theta)_{11})={\rm pf}(F^{l-1}(\Theta))\notin \mathbb{Z}$, construction of such projection is well-known
  and the projection is known as Rieffel projection with trace ${\rm pf}(F^{l-1}(\Theta))+k^{(l-1)}_0,$ for some $k^{(l-1)}_0\in\mathbb{Z}$ (see (\ref{Rieffel pro})).

Now suppose that for some $m\in \{1,2,\dots,l-1\}$, there is such a projection in
$\mathcal{A}_{F^m(\Theta)}$ such that \begin{eqnarray}\label{u1}
\tau_{F^m(\Theta)}(p_m)&=&{\rm pf}(F^m(\Theta))+k_{(1,2)}^{(m)}{\rm pf}(M^{F^m(\Theta)}_{(1,2)})+\cdots\nonumber\\
&&+k_{(1,2,\dots,n-2(m+1))}^{(m)}{\rm pf}(M^{F^m(\Theta)}_{(1,2,\dots,n-2(m+1))})+k^{(m)}_0\label{u1}
\end{eqnarray} for some $k_I^{(m)},k^{(m)}_0\in \mathbb{Z},$ where $I=(1,2,\dots,n-2s)$ and $s=m+1,m+2,\dots,l-1.$

   Next we want to prove that there is a projection $p_{m-1}$ in $\mathcal{A}_{F^{m-1}(\Theta)}$ with $\tau_{F^{m-1}(\Theta)}(p_{m-1})={\rm pf}(F^{m-1}(\Theta))+k_{(1,2)}^{(m-1)}{\rm pf}(M^{F^{m-1}(\Theta)}_{(1,2)})+\cdots+k_{(1,2,\dots,n-2m))}^{(m-1)}{\rm pf}(M^{F^{m-1}(\Theta)}_{(1,2,\dots,n-2m)})+k^{(m-1)}_0$ for some $k_I^{(m-1)},k^{(m-1)}_0\in \mathbb{Z},$ where $I=(1,2,\dots,n-2s)$ and $s=m,m+1,\dots,l-1.$ Now we follow the method of page 198-199 in \cite{Boca} to construct such a projection. Write
   $$F^{m-1}(\Theta)=\left(\begin{matrix}
                  F^{m-1}(\Theta)_{11} &  F^{m-1}(\Theta)_{12} \\
                 F^{m-1}(\Theta)_{21}& F^{m-1}(\Theta)_{22}
                 \end{matrix}
\right)=\left(\begin{matrix}
                  F^{m-1}(\Theta)_{11} &  F^{m-1}(\Theta)_{12} \\
                 -F^{m-1}(\Theta)_{12}^t& F^{m-1}(\Theta)_{22}
                 \end{matrix}
\right)\in \mathcal{T}_{n-2(m-1)},$$
where $F^{m-1}(\Theta)_{11}$ is a $2\times 2$ block.
 From the previous discussion of this section, $\mathcal{A}_{F^{m-1}(\Theta)}$ is strong Morita equivalent to $\mathcal{A}_{F^{m-1}(\Theta)'},$
where
\begin{eqnarray}F^{m-1}(\Theta)'&=&\left(\begin{matrix}
                  F^{m-1}(\Theta)_{11}^{-1} & -F^{m-1}(\Theta)_{11}^{-1}F^{m-1}(\Theta)_{12} \\
                 F^{m-1}(\Theta)_{21}F^{m-1}(\Theta)_{11}^{-1} & F^{m-1}(\Theta)_{22}-F^{m-1}(\Theta)_{21}F^{m-1}(\Theta)_{11}^{-1}F^{m-1}(\Theta)_{12}
                \end{matrix}
\right)\nonumber\\
&=&\left(\begin{matrix}
                  F^{m-1}(\Theta)_{11}^{-1} & -F^{m-1}(\Theta)_{11}^{-1}F^{m-1}(\Theta)_{12} \\
                 F^{m-1}(\Theta)_{21}F^{m-1}(\Theta)_{11}^{-1} & F^{m}(\Theta)
                \end{matrix}
\right).\nonumber
\end{eqnarray}
Denote the Rieffel projection in $\mathcal{A}_{F^{m-1}(\Theta)_{11}}$, by $e$ and  $\tau_{F^{m-1}(\Theta)}(e)={\rm pf}(F^{m-1}(\Theta)_{11})$ (see Appendix for the construction of such $e$, here we use the assumption that ${\rm pf}(F^{m-1}(\Theta)_{11})$ is in $(0,1)$). It follows that $\mathcal{A}_{F^{m-1}(\Theta)'}\cong e\mathcal{A}_{F^{m-1}(\Theta)}e,$ and we denote this isomorphism by $\psi$. By (\ref{u1}) there exists a projection $e'\in \mathcal{A}_{F^{m}(\Theta)}\subset \mathcal{A}_{F^{m-1}(\Theta)'}$ with
\begin{eqnarray}\tau_{F^{m}(\Theta)}(e')&=&{\rm pf}(F^m(\Theta))+k_{(1,2)}^{(m)}{\rm pf}(M^{F^m(\Theta)}_{(1,2)})+\cdots\nonumber\\
&&+k_{(1,2,\dots,n-2(m+1))}^{(m)}{\rm pf}(M^{F^m(\Theta)} _{(1,2,\dots,n-2(m+1))})+k^{(m)}_0\label{m13}
\end{eqnarray}
 for some $k_I^{(m)},k^{(m)}_0\in \mathbb{Z},$ where $I=(1,2,\dots,n-2s)$ and $s=m+1,m+2,\dots,l-1.$
Now for the tracial state $\tau_{F^{m-1}(\Theta)}$ on $\mathcal{A}_{F^{m-1}(\Theta)}$, since $\psi(1_{\mathcal{A}_{F^{m-1}(\Theta)'}})=e$ and $\tau_{F^{m-1}(\Theta)}(e)={\rm pf}(F^{m-1}(\Theta)_{11})$, we have $$\frac{1}{{\rm pf}(F^{m-1}(\Theta)_{11})}\tau_{F^{m-1}(\Theta)}\circ \psi$$ is a tracial state on $\mathcal{A}_{F^{m-1}(\Theta)'}.$ Note $e'\in \mathcal{A}_{F^{m}(\Theta)}\subset \mathcal{A}_{F^{m-1}(\Theta)'},$ so
$\tau_{F^{m-1}(\Theta)'}(e')=\tau_{F^{m}(\Theta)}(e').$
Then using the fact that all tracial states of $\mathcal{A}_{F^{m-1}(\Theta)'}$ induce the same map on K$_0(\mathcal{A}_{F^{m-1}(\Theta)'})$, and (\ref{m13}), we get
\begin{eqnarray}
&&\frac{1}{{\rm pf}(F^{m-1}(\Theta)_{11})}\tau_{F^{m-1}(\Theta)}\circ \psi (e')\nonumber\\
&=&\tau_{F^{m-1}(\Theta)'}(e')=\tau_{F^{m}(\Theta)}(e')\nonumber\\
&=&{\rm pf}(F^m(\Theta))+k_{(1,2)}^{(m)}{\rm pf}(M^{F^m(\Theta)}_{(1,2)})+\cdots\nonumber\\
&&+k_{(1,2,\dots,n-2(m+1))}^{(m)}{\rm pf}(M^{F^m(\Theta)} _{(1,2,\dots,n-2(m+1))})+k^{(m)}_0\label{m11}
\end{eqnarray}
for some $k_I^{(m)},k^{(m)}_0\in \mathbb{Z},$ where $I=(1,2,\dots,n-2s)$ and $s=m+1,m+2,\dots,l-1.$ Let $p_{m-1}=\psi(e'),$
from (\ref{m11}) we get
\begin{eqnarray}
\tau_{F^{m-1}(\Theta)}(p_{m-1})&=&\tau_{F^{m-1}(\Theta)}
(\psi(e'))\nonumber\\
&=&{\rm pf}(F^{m-1}(\Theta)_{11})(\frac{1}{{\rm pf}(F^{m-1}(\Theta)_{11})}\tau_{F^{m-1}(\Theta)}\circ \psi (e'))\nonumber\\
&=&{\rm pf}(F^{m-1}(\Theta)_{11})({\rm pf}(F^m(\Theta))+k_{(1,2)}^{(m)}{\rm pf}(M^{F^m(\Theta)}_{(1,2)})+\cdots\nonumber\\
&&+k_{(1,2,\dots,n-2(m+1))}^{(m)}{\rm pf}(M^{F^m(\Theta)} _{(1,2,\dots,n-2(m+1))})+k^{(m)}_0).\nonumber
\end{eqnarray}
But from (\ref{=2}) and (\ref{pf1}), we have
\begin{eqnarray}
&&{\rm pf}(F^{m-1}(\Theta)_{11})({\rm pf}(F^m(\Theta))+k_{(1,2)}^{(m)}{\rm pf}(M^{F^m(\Theta)}_{(1,2)})\nonumber\\
&&+\cdots+k_{(1,2,\dots,n-2(m+1))}^{(m)}{\rm pf}(M^{F^m(\Theta)}_{(1,2,\dots,n-2(m+1))})+k^{(m)}_0)\nonumber\\
&=&{\rm pf}(F^{m-1}(\Theta))+k_{(1,2)}^{(m)}{\rm pf}(M^{F^{m-1}(\Theta)}_{(1,2,3,4)})\nonumber\\
&&+\cdots+k_{(1,2,\dots,n-2(m+1))}^{(m)}{\rm pf}(M^{F^{m-1}(\Theta)}_{(1,2,\dots,n-2m)})+k^{(m)}_0{\rm pf}(M^{F^{m-1}(\Theta)}_{(1,2)}).\nonumber
\end{eqnarray}
Put $k_0^{(m)}=k_{(1,2)}^{m-1}$, $k_{(1,2,3,4)}^{(m-1)}=k_{(1,2)}^{(m)}$,\dots,
$k_{(1,2,\dots,n-2m)}^{(m-1)}=k_{(1,2,\dots,n-2(m+1))}^{(m)},$ then
there is a projection $p_{m-1}$ in $\mathcal{A}_{F^{m-1}(\Theta)}$ with \begin{eqnarray}\tau_{F^{m-1}(\Theta)}(p_{m-1})&=&{\rm pf}(F^{m-1}(\Theta))+k_{(1,2)}^{(m-1)}{\rm pf}(M^{F^{m-1}(\Theta)}_{(1,2)})+\cdots\nonumber\\
&&+k_{(1,2,\dots,n-2m))}^{(m-1)}{\rm pf}(M^{F^{m-1}(\Theta)}_{(1,2,\dots,n-2m)})+k^{(m-1)}_0\nonumber
\end{eqnarray}for some $k_I^{(m-1)},k^{(m-1)}_0\in \mathbb{Z},$ where $I=(1,2,\dots,n-2s)$ and $s=m,m+1,\dots,l-1.$

\end{proof}

The above theorem tells us how to construct a higher dimensional projection with corresponding trace values from a low dimensional projection under certain conditions. We call this type of projection a Rieffel-type projection. By Lemma \ref{h2}, we know that
\begin{eqnarray}{\rm pf}(F^j(\Theta)_{11})=\frac{{\rm pf}^{\Theta}_{(1,2,\dots,n-p-1,n-p)}}{{\rm pf}^{\Theta}_{(1,2,\dots,n-p-2)}},\quad  p=n-2j-2,\quad j=1,\dots,l-1.
\end{eqnarray}
Therefore, the conditions of the above theorem can also be described as:
\begin{eqnarray}\frac{{\rm pf}^{\Theta}_{(1,2,\dots,n-p-1,n-p)}}{{\rm pf}^{\Theta}_{(1,2,\dots,n-p-2)}}\in (0,1)\quad {\rm for}\,\, p=n-2j-2,\quad j=1,\dots,l-2,
\end{eqnarray}
 and
 \begin{eqnarray}\label{a4}
 \frac{{\rm pf}^{\Theta}_{(1,2,\dots,n-1,n)}}{{\rm pf}^{\Theta}_{(1,2,\dots,n-2)}}\notin \mathbb{Z}.
 \end{eqnarray}We record this fact as a corollary below.

 \begin{corollary}\label{even n projection explicit}
  For any even number $n=2l\geq2,$ if $\Theta\in \mathcal{T}_n$ satisfies \begin{eqnarray}\frac{{\rm pf}^{\Theta}_{(1,2,\dots,n-p-1,n-p)}}{{\rm pf}^{\Theta}_{(1,2,\dots,n-p-2)}}\in (0,1)\quad {\rm for}\,\, p=n-2j-2,\quad j=1,\dots,l-2,
\end{eqnarray} and  $
 \frac{{\rm pf}^{\Theta}_{(1,2,\dots,n-1,n)}}{{\rm pf}^{\Theta}_{(1,2,\dots,n-2)}}\notin \mathbb{Z},$ then
  there exist a (Rieffel-type) projection $p_m$ inside $\mathcal{A}_{F^m(\Theta)}$ such that
  \begin{eqnarray}\tau_{F^m(\Theta)}(p_m)={\rm pf}(F^m(\Theta))+k_{(1,2)}^{(m)}{\rm pf}(M^{F^m(\Theta)}_{(1,2)})+\cdots+k_{(1,2,\dots,n-2(m+1))}^{(m)}{\rm pf}(M^{F^m(\Theta)}_{(1,2,\dots,n-2(m+1))})+k^{(m)}_0\nonumber
   \end{eqnarray}
   for some $k_I^{(m)},k^{(m)}_0\in \mathbb{Z}$ and $m=0,1,2,\dots,l-1,$ where $I=(1,2,\dots,n-2s)$ and $s=m+1,m+2,\dots,l-1,$ $\tau_{F^m(\Theta)}$ is the canonical tracial state on $\mathcal{A}_{F^m(\Theta)}.$
\end{corollary}

In particular, when $m=0,$ we have the following:

 \begin{corollary}\label{even n projection explicit_m=0}
  For any even number $n=2l\geq2,$ if $\Theta\in \mathcal{T}_n$ satisfies \begin{eqnarray}\frac{{\rm pf}^{\Theta}_{(1,2,\dots,n-p-1,n-p)}}{{\rm pf}^{\Theta}_{(1,2,\dots,n-p-2)}}\in (0,1)\quad {\rm for}\,\, p=n-2j-2,\quad j=1,\dots,l-2,
\end{eqnarray} and  $
 \frac{{\rm pf}^{\Theta}_{(1,2,\dots,n-1,n)}}{{\rm pf}^{\Theta}_{(1,2,\dots,n-2)}}\notin \mathbb{Z},$ then
  there exist a (Rieffel-type) projection $p=p_0$ inside $\mathcal{A}_{\Theta}$ such that
  \begin{eqnarray}\tau_{\Theta}(p)={\rm pf}(\Theta)+k_{(1,2)}{\rm pf}(M^{\Theta}_{(1,2)})+\cdots+k_{(1,2,\dots,n-2)}{\rm pf}(M^{\Theta}_{(1,2,\dots,n-2)})+k_0\nonumber
   \end{eqnarray}
   for some $k_I,k_0\in \mathbb{Z},$ where $I=(1,2,\dots,n-2s)$ and $s=1,2,\dots,l-1,$ $\tau_{\Theta}$ is the canonical tracial state on $\mathcal{A}_{\Theta}.$
\end{corollary}

 Let us now show that under suitable conditions on $\Theta,$ the numbers coming in the right hand side of the equation~(\ref{range of trace}) may be realized as traces of the Rieffel-type projections in $\mathcal{A}_\Theta.$
\begin{theorem}\label{PI1}
Let $\Theta\in \mathcal{T}_n$ with $n\geq 2$. If
\begin{eqnarray}\label{a21}
 {\rm pf}(F^{j}(M^{\Theta}_I)_{11})\in (0,1),\quad {\text and}\,\, {\rm pf}(F^{m-1}(M^{\Theta}_I)_{11})\notin \mathbb{Z}
 \end{eqnarray}
for all $I$ with $2\leq|I|\leq 2\lfloor\,\frac{n}{2}\,\rfloor$, $|I|=l=2m,$ and for all $j=0,1,\dots,m-2,$
  then there exist  (Rieffel-type) projections $P_{I},$ such that
  \begin{eqnarray}\label{pI1}\tau_{\Theta}(P_{I})={\rm pf}(M_{I}^{\Theta})+\sum\limits_{0<|J|< l}{\rm pf}(M_{J}^{\Theta})k_{I_J}+k_{I_0}
  \end{eqnarray}
  for some $k_{I_J},k_{I_0}\in \mathbb{Z},$
where  $J=(i_1,i_2,\dots,i_{s})$ and $s=2,4,\dots,2m-2,$ $\tau_{\Theta}$ is the canonical tracial state on $\mathcal{A}_{\Theta}.$
\end{theorem}
\begin{proof}
  We use $M_{I}^{\Theta}$ instead of $\Theta$ in Corollary~\ref{even n projection explicit_m=0} to get $\tau_{\Theta}(P_{I})={\rm pf}(M_{I}^{\Theta})+\sum\limits_{0<|J|< l}{\rm pf}(M_{J}^{\Theta})k_{I_J}+k_{I_0}$.
\end{proof}

The conditions in the above theorem are satisfied by many examples. For example, if $n=2, 3$ any antisymmetric matrix with non-integer entries will work. For $n=4$ let us consider the following class of matrices: $\Theta=\left(
  \begin{matrix}
    0 & \gamma & \gamma^2 & \gamma^8 \\
    -\gamma & 0 & \gamma^4 & \gamma^{16}\\
    -\gamma^2 & -\gamma^4  & 0 & \gamma^{32} \\
    -\gamma^8 & -\gamma^{16} & -\gamma^{32} & 0  \end{matrix}
  \right), \gamma \in (0,1).$ An easy computation shows that   \begin{eqnarray*}
  {\rm pf}(F^1(\Theta)_{11})&\overset{(\ref{f2})}{=}& \frac{\gamma\cdot\gamma^{32}-\gamma^{2}\cdot\gamma^{16}+\gamma^{8}\cdot\gamma^{4}}{\gamma}\\
  &=&\gamma^{11}(\gamma^{21}-\gamma^6+1)\notin \Z.
   \end{eqnarray*}
   Hence this class satisfies the conditions in the above theorem.

We will now see that the above projections form a generating set of $\K_0(\mathcal{A}_\Theta),$ when $\Theta$ is totally irrational, under the assumptions of the above theorem.

\begin{definition}\label{a3}
 We say that $\Theta$ is {\it strongly totally irrational} if $\Theta=(\theta_{jk})\in \mathcal{T}_n$ is totally irrational matrix such that
 \begin{eqnarray}\label{a2}
 {\rm pf}(F^{j}(M^{\Theta}_I)_{11})\in (0,1)\quad\quad (when \,\,n\geq4)
 \end{eqnarray}
for all $I$ with $4\leq|I|\leq 2\lfloor\,\frac{n}{2}\,\rfloor$, and for all $j=0,1,\dots,m-2,$ where $|I|=2m.$ For $n=2,3$, we do not need the condition (\ref{a2}) to hold. (Here we don't need condition ${\rm pf}(F^{m-1}(M^{\Theta}_I)_{11})\notin \mathbb{Z}$ because it is satisfied automatically by (\ref{a4}) and $\Theta$'s totally irrationality.)
\end{definition}
Note that since $e^{2\pi i\theta}=e^{2\pi i(\theta+k)}$ for any $k\in\mathbb{Z},$ we can perturb $\Theta$ such a way that the condition (\ref{a2}) holds for $n=4,5.$ Indeed, when $n=4,$ take $|I|=4.$ Then $M^{\Theta}_I=\Theta$ and $j=0$ so that ${\rm pf}(F^{0}(\Theta)_{11})={\rm pf}(\Theta_{11})=\theta_{12}.$ If $\theta_{12}$ is not in $(0,1),$ we can find some $k\in\mathbb{Z}$ such that $\theta_{12}+k$ is in $(0,1)$. Thus we can replace $\theta_{12}$ with $\theta_{12}+k$ in $\Theta$ (the corresponding noncommutative torus  remains unchanged). When $n=5$, all possible $M^{\Theta}_I$'s are $M^{\Theta}_{(1,2,3,4)}$, $M^{\Theta}_{(2,3,4,5)}$, $M^{\Theta}_{(1,3,4,5)}$, $M^{\Theta}_{(1,2,4,5)}$, $M^{\Theta}_{(1,2,3,5)},$ and $j=0.$ As all possible ${\rm pf}(F^{j}(M^{\Theta}_I)_{11})$'s are $\theta_{12},\theta_{23},\theta_{13},$ we can  deal with them as in the case of $n = 4.$

For $n=6,$ we now determine a class of totally irrational matrices for which the condition (\ref{a2}) holds.

\begin{proposition}\label{h3}
  Let $\gamma$ be a transcendental number in $(0,1)$. Then for the matrix
  \begin{eqnarray*}
  \Theta=\left(
  \begin{matrix}
    0 & \gamma & \gamma^2 & \gamma^8 & \gamma^{64} & \gamma^{1024} \\
    -\gamma & 0 & \gamma^4 & \gamma^{16} & \gamma^{128} & \gamma^{2048} \\
    -\gamma^2 & -\gamma^4  & 0 & \gamma^{32} & \gamma^{256} & \gamma^{4096} \\
    -\gamma^8 & -\gamma^{16} & -\gamma^{32} & 0 & \gamma^{512} & \gamma^{8192} \\
    -\gamma^{64} & -\gamma^{128} & - \gamma^{256} & -\gamma^{512} & 0 & \gamma^{16384} \\
    -\gamma^{1024} & -\gamma^{2048} & -\gamma^{4096} & -\gamma^{8192}& -\gamma^{16384} & 0
  \end{matrix}
  \right),
  \end{eqnarray*} the condition (\ref{a2}) holds and $\Theta$ is totally irrational. Hence  $\Theta$ is strongly totally irrational.
\end{proposition}
\begin{proof}Since $\gamma$ is transcendental, it is well known that the numbers $\gamma,\gamma^2,\dots,\gamma^{2^n},$ as well as any product the numbers are linear independent over $\mathbb{Q}$ for any $n\in \mathbb{N}.$ So we have that $\Theta$ is totally irrational.
     For any $|I|=4$, the condition (\ref{a2})  trivially holds for $j=0.$ For $|I|=6,$ namely $M^{\Theta}_{I}=\Theta,$ again the condition (\ref{a2})  trivially holds for $j=0.$ For $j=1,$
   \begin{eqnarray*}
  F^1(\Theta)&=&
   \Theta_{22}-\Theta_{21}\Theta_{11}^{-1}\Theta_{12}\\
   &\overset{(\ref{F(theta)})}{=}&\left(\begin{matrix}
                0 &  \frac{{\rm pf}^{\Theta}_{(1,2,3,4)}}{\gamma}&\frac{{\rm pf}^{\Theta}_{(1,2,3,5)}}{\gamma}&\frac{{\rm pf}^{\Theta}_{(1,2,3,6)}}{\gamma} \\
              - \frac{{\rm pf}^{\Theta}_{(1,2,3,4)}}{\gamma}& 0&\frac{{\rm pf}^{\Theta}_{(1,2,4,5)}}{\gamma}&\frac{{\rm pf}^{\Theta}_{(1,2,4,6)}}{\gamma}\\
             - \frac{{\rm pf}^{\Theta}_{(1,2,3,5)}}{\gamma}&- \frac{{\rm pf}^{\Theta}_{(1,2,4,5)}}{\gamma}&0&\frac{{\rm pf}^{\Theta}_{(1,2,5,6)}}{\gamma}\\
              -\frac{{\rm pf}^{\Theta}_{(1,2,3,6)}}{\gamma}&-\frac{{\rm pf}^{\Theta}_{(1,2,4,6)}}{\gamma}&-\frac{{\rm pf}^{\Theta}_{(1,2,5,6)}}{\gamma}&0
                 \end{matrix}\right).
   \end{eqnarray*}
Then
   \begin{eqnarray*}
  {\rm pf}(F^1(\Theta)_{11})&=&\frac{{\rm pf}^{\Theta}_{(1,2,3,4)}}{\gamma}\\
  &\overset{(\ref{f2})}{=}& \frac{\gamma\cdot\gamma^{32}-\gamma^{2}\cdot\gamma^{16}+\gamma^{8}\cdot\gamma^{4}}{\gamma}\\
  &=&\gamma^{11}(\gamma^{21}-\gamma^6+1)\in (0,1).
   \end{eqnarray*}

\end{proof}

In Appendix \uppercase\expandafter{\romannumeral1},  we provide a large class of examples of $n\times n$ skew-symmetric matrices, generalising  Proposition \ref{h3}, which satisfy the conditions of Theorem \ref{PI1}. These skew-symmetric matrices can also be made totally irrational (see Corollary~\ref{a}).

\begin{theorem}\label{PI}
  For any integer $n\geq 2,$ let $\Theta\in \mathcal{T}_n$ be strongly totally irrational, then there exist  (Rieffel-type) projections $P_{I},$ $I=(i_1,i_2,\dots,i_{l})$, $1\leq i_1<i_2<\cdots<i_l\leq n,$ $l=2m,$ $m=1,2,\dots,\lfloor\frac{n}{2}\rfloor,$ inside $\mathcal{A}_{\Theta},$ such that
  \begin{eqnarray}\label{pI1}\tau_{\Theta}(P_{I})={\rm pf}(M_{I}^{\Theta})+\sum\limits_{0<|J|< l}{\rm pf}(M_{J}^{\Theta})k_{I_J}+k_{I_0}
  \end{eqnarray}
  for some $k_{I_J},k_{I_0}\in \mathbb{Z},$
where  $J=(i_1,i_2,\dots,i_{s})$ and $s=2,4,\dots,2m-2,$ $\tau_{\Theta}$ is the canonical tracial state on $\mathcal{A}_{\Theta}.$  Moreover, the generators of $\K_0(\mathcal{A}_{\Theta})$ are given by $[1]$ and $\{P_{I}|I=(i_1,i_2,\dots,i_{l})$, $1\leq i_1<i_2<\cdots<i_l\leq n, l=2m, m=1,2,\dots,\lfloor\frac{n}{2}\rfloor\}$.
\end{theorem}
\begin{proof}By Theorem \ref{PI1} and from the definition of strongly totally irrationality,  we know that those projections $\{P_{I}\}$ exist.
  Note that since $\Theta$ is totally irrational, $(\tau_{\Theta})_*$ is injective. So $[1]$ and $\{P_{I}|I=(i_1,i_2,\dots,i_{l}), 1\leq i_1<i_2<\cdots<i_l\leq n, l=2m, m=1,2,\dots,\lfloor\frac{n}{2}\rfloor\}$ are the generators of $\K_0(\mathcal{A}_{\Theta})$ by (\ref{range of trace}).
\end{proof}

In Appendix \uppercase\expandafter{\romannumeral2}, as an example, we will give explicit expressions of the Rieffel-type projections in $\mathcal{A}_{\Theta}$ when $n=4.$

\section{Stability of some relations in $C^*$-algebras of tracial rank at most one}\label{sec:stability}

From the above discussion, when $\Theta\in\mathcal{T}_n$ is strongly totally irrational, we know that $P_{I}$ is in $\mathcal{A}_{\Theta}$  and
$P_{I}$ can be expressed in terms of $\mathfrak{u}_{i_1},\mathfrak{u}_{i_2},\dots,\mathfrak{u}_{i_l}$ by Theorem \ref{PI}, where $I=(i_1,i_2,\dots,i_{l})$, $1\leq i_1<i_2<\cdots<i_l\leq n,$ $l=2m,$ $m=1,2,\dots,\lfloor\frac{n}{2}\rfloor.$ Hence we can assume that
\begin{eqnarray}\label{coefficient}
P_{I}=\sum\limits_{k_1=-\infty}^{+\infty}\sum\limits_{k_2=-\infty}^{+\infty}\cdots \sum\limits_{k_l=-\infty}^{+\infty}a_{I,k_1,k_2,\dots,k_l}\mathfrak{u}_{i_1}^{k_1}\mathfrak{u}_{i_2}^{k_2}\cdots \mathfrak{u}_{i_l}^{k_l},\quad a_{I,k_1,k_2,\dots,k_l}\in\mathbb{C}
\end{eqnarray}
where $I=(i_1,i_2,\dots,i_{l})$, $1\leq i_1<i_2<\cdots<i_l\leq n,$ $l=2m,$ $m=1,2,\dots,\lfloor\frac{n}{2}\rfloor$.

The following lemma is well known (we refer to the proof of Lemma 3 of \cite{Hua-FANG-XU}).
\begin{lemma}\label{hfx}
Let $f\in C([-\frac{1}{8},\frac{1}{8}]\bigcup [\frac{7}{8},\frac{9}{8}]).$ Then, for any $\epsilon>0,$ there exists $\delta'>0$ such that for any unital $C^*$-algebra $A$ and two self-adjoint elements $a,b\in A$ with ${\rm spec}(a),{\rm spec}(b)\subset [-\frac{1}{8},\frac{1}{8}]\bigcup [\frac{7}{8},\frac{9}{8}],$ if $\|a-b\|<\delta',$ then $\|f(a)-f(b)\|<\epsilon.$
  \end{lemma}
From the lemma above, let $f=\chi_{(\frac{1}{2},+\infty)}$ be the characteristic function on $(\frac{1}{2},\infty)$, for $\epsilon=\frac{1}{6}$, there exists $\delta'>0$ such that for any unital $C^*$-algebra $A$ and two self-adjoint elements $a,b\in A$ with ${\rm spec}(a),{\rm spec}(b)\subset [-\frac{1}{8},\frac{1}{8}]\bigcup [\frac{7}{8},\frac{9}{8}],$ if $\|a-b\|<\delta',$ then $\|\chi_{(\frac{1}{2},+\infty)}(a)-\chi_{(\frac{1}{2},+\infty)}(b)\|<1.$

Next we can choose a sufficiently large positive integer  $N$ such that
 \begin{eqnarray}
\| P_{I}-\frac{1}{2}(\sum\limits_{k_1=-N}^{N}\sum\limits_{k_2=-N}^{N}\cdots \sum\limits_{k_l=-N}^{N}a_{I,k_1,k_2,\dots,k_l}\mathfrak{u}_{i_1}^{k_1}\mathfrak{u}_{i_2}^{k_2}\cdots \mathfrak{u}_{i_l}^{k_l}\nonumber\\+(\sum\limits_{k_1=-N}^{N}\sum\limits_{k_2=-N}^{N}\cdots \sum\limits_{k_l=-N}^{N}a_{I,k_1,k_2,\dots,k_l}\mathfrak{u}_{i_1}^{k_1}\mathfrak{u}_{i_2}^{k_2}\cdots \mathfrak{u}_{i_l}^{k_l})^*)\|<\delta'<\frac{1}{8},\label{rep}
\end{eqnarray}
for all $P_I$ and $a_{I,k_1,k_2,\dots,k_l}\in \mathbb{C}$ is as in (\ref{coefficient}).
Let
\begin{eqnarray}
  \widetilde{P}_{I}=\frac{1}{2}(\sum\limits_{k_1=-N}^{N}\sum\limits_{k_2=-N}^{N}\cdots \sum\limits_{k_l=-N}^{N}a_{I,k_1,k_2,\dots,k_l}\mathfrak{u}_{i_1}^{k_1}\mathfrak{u}_{i_2}^{k_2}\cdots \mathfrak{u}_{i_l}^{k_l}\nonumber\\+(\sum\limits_{k_1=-N}^{N}\sum\limits_{k_2=-N}^{N}\cdots \sum\limits_{k_l=-N}^{N}a_{I,k_1,k_2,\dots,k_l}\mathfrak{u}_{i_1}^{k_1}\mathfrak{u}_{i_2}^{k_2}\cdots \mathfrak{u}_{i_l}^{k_l})^*).\label{Nk}
\end{eqnarray}
It is obvious that $\widetilde{P}_{I}$ is a self-adjoint element, and  $\mathrm{spec}(\widetilde{P}_{I})\subset (-\frac{1}{8},\frac{1}{8})\bigcup (\frac{7}{8},\frac{9}{8})$.
Thus by Lemma \ref{hfx} we have \begin{eqnarray}
 \|\chi_{(\frac{1}{2},+\infty)}(P_{I})-\chi_{(\frac{1}{2},+\infty)}(\widetilde{P}_{I})\|<1,
\end{eqnarray}
i.e.
\begin{eqnarray}\label{p=p}
 \|P_{I}-\chi_{(\frac{1}{2},+\infty)}(\widetilde{P}_{I})\|<1.
\end{eqnarray}
Thus by Lemma 2.5.1 of \cite{Lin-book},  we know that \begin{eqnarray}[P_{I}]=[\chi_{(\frac{1}{2},+\infty)}(\widetilde{P}_{I})]\label{b1}
\end{eqnarray} in $\K_0(A_{\Theta}).$

\begin{definition}\label{eI}
  Let $\Theta=(\theta_{jk})\in \mathcal{T}_n$ be strongly totally irrational. Let $\mathcal{A}$ be a unital $C^*$-algebra and $u_1,u_2,\dots,u_n$ be an $n$-tuple of unitaries in $\mathcal{A}.$ We define \begin{eqnarray}e_{I}(u_1,u_2,\dots,u_n)&=&\frac{1}{2}(\sum\limits_{k_1=-N}^{N}\sum\limits_{k_2=-N}^{N}\cdots \sum\limits_{k_l=-N}^{N}a_{I,k_1,k_2,\dots,k_l}u_{i_1}^{k_1}u_{i_2}^{k_2}\cdots u_{i_l}^{k_l}\nonumber\\
  &&+(\sum\limits_{k_1=-N}^{N}\sum\limits_{k_2=-N}^{N}\cdots \sum\limits_{k_l=-N}^{N}a_{I,k_1,k_2,\dots,k_l}u_{i_1}^{k_1}u_{i_2}^{k_2}\cdots u_{i_l}^{k_l})^*),\nonumber
  \end{eqnarray}
where $I=(i_1,i_2,\dots,i_{l})$, $1\leq i_1<i_2<\cdots<i_l\leq n,$ $l=2m,$ $m=1,2,\dots,\lfloor\frac{n}{2}\rfloor,$ $a_{I,k_1,k_2,\dots,k_l}$ is as in (\ref{coefficient}), and $N$ is as in (\ref{rep}). In the absence of ambiguity, we often use $e_I$ to express $e_{I}(u_1,u_2,\dots,u_n)$ briefly. It is clear that $e_{I}$ is always self-adjoint.
\end{definition}
\begin{remark}
  In the above definition, we do not directly define $e_{I}$ in the form of $P_I$, but instead use the form of $\widetilde{P}_I$. This is because $e_{I}$ may not converge if we use the form of $P_I$.
\end{remark}

\begin{remark}\label{remark1}
  If $u_1,u_2,\dots,u_n$ in a unital $C^*$-algebra $\mathcal{A}$ satisfy $u_k u_j=e^{2\pi i\theta_{jk}}u_j u_k$ for $j,k=1,2,\dots,n$, then $\chi_{(\frac{1}{2},+\infty)}(e_{I})$ is a projection in $\mathcal{A}$ for every $I.$ This is because $\mathcal{A}_{\Theta}$ is the universal $C^*$-algebra generated by $\mathfrak{u}_1,\mathfrak{u}_2,\dots,\mathfrak{u}_n$ subject to the relations $\mathfrak{u}_k \mathfrak{u}_j=e^{2\pi i\theta_{jk}}\mathfrak{u}_j \mathfrak{u}_k$, $j,k=1,2,\dots,n$, there exists a homomorphism $\varphi:\mathcal{A}_{\Theta}\rightarrow \mathcal{A}$ such that $\varphi(\mathfrak{u}_j)=u_j,\, j=1,2,\dots,n.$ Thus
  $\varphi(\widetilde{P}_{I})=e_{I}$ for any possible $I$ defined as in Definition \ref{eI}, we have $\varphi(\chi_{(\frac{1}{2},+\infty)}(\widetilde{P}_{I}))=\chi_{(\frac{1}{2},+\infty)}(e_{I}).$ It follows that $\chi_{(\frac{1}{2},+\infty)}(e_{I})$ is a projection in $\mathcal{A}.$ In particular, by (\ref{b1}) we have \begin{eqnarray}[\chi_{(\frac{1}{2},+\infty)}(e_{I}(\mathfrak{u}_1,\dots,\mathfrak{u}_n))]=[\chi_{(\frac{1}{2},+\infty)}(\widetilde{P}_{I})]=[P_{I}]
  \end{eqnarray}
  in K$_0(\mathcal{A}_{\Theta})$ for any possible $I$ defined in Definition \ref{eI}.
\end{remark}

\begin{proposition}\label{gap}
 Let $\Theta=(\theta_{jk})\in \mathcal{T}_n.$ There exists a $\delta>0$ such that: For any unital $C^*$-algebra $\mathcal{A},$ any $n$-tuple of unitaries $u_1,u_2,\dots,u_n$ in $\mathcal{A}$  satisfying $\|u_ku_j-e^{2\pi i\theta_{jk}}u_ju_k\|<\delta$ for $j,k=1,2,\dots,n,$ we have
 $\|(e_I(u_1,u_2,\dots,u_n))^2-e_I(u_1,u_2,\dots,u_n)\|<\frac{1}{4}$ for any possible $I$ defined in Definition \ref{eI}. In particular, $\mathrm{spec}(e_I)$ has a gap at $\frac{1}{2}$.
\end{proposition}
\begin{proof}
  Suppose the statement is false. Let $(\delta_m)_{m=1}^{\infty}$ be a sequence of positive numbers decreasing to $0.$ Then for any positive integer $m,$ there exists a unital $C^*$-algebra $\mathcal{A}_m$, a $n$-tuple of unitaries $\tilde{u}_{1m},\tilde{u}_{2m},\dots,\tilde{u}_{nm}$ in $\mathcal{A}_m$
  such that  $\|u_{km}u_{jm}-e^{2\pi i\theta_{jk}}u_{jm}u_{km}\|<\delta_m,$ but
  $$\|(e_I(\tilde{u}_{1m},\tilde{u}_{2m},\dots,\tilde{u}_{nm}))^2-e_I(\tilde{u}_{1m},\tilde{u}_{2m},\dots,\tilde{u}_{nm})\|\geq \frac{1}{4}.$$
  Let $\mathcal{B}=\prod_{m=1}^{\infty}\mathcal{A}_m\slash \oplus_{m=1}^{\infty}\mathcal{A}_m.$ Let $\pi:\prod_{m=1}^{\infty}\mathcal{A}_m\rightarrow \mathcal{B}$ be the canonical quotient map.
  Let $\tilde{u}_j=(\tilde{u}_{j1},\tilde{u}_{j2},\dots,\tilde{u}_{jm},\dots)$ for $j=1,2,\dots,n.$ Then $\pi(\tilde{u}_j)$ satisfies
  $\pi(\tilde{u}_k)\pi(\tilde{u}_j)=e^{2\pi i\theta_{jk}}\pi(\tilde{u}_j)\pi(\tilde{u}_k),$
  hence there is a homomorphism
  $$\phi:\mathcal{A}_{\Theta}\rightarrow \mathcal{B},\quad \phi(\mathfrak{u}_j)=\pi(\tilde{u}_j)\quad {\rm for}\, j=1,2,\dots,n,$$
  where $\mathfrak{u}_j,j=1,2,\dots,n$ are the canonical generators of $\mathcal{A}_{\Theta}.$ In particular, the element
  $$\pi(e_{I}(\tilde{u}_{11},\dots,\tilde{u}_{n1}),e_{I}(\tilde{u}_{12},
  \dots,\tilde{u}_{n2}),\dots)=\pi(e_I(\tilde{u}_1,\dots,\tilde{u}_n))=\phi(e_I(\mathfrak{u}_1,\dots,\mathfrak{u}_n))$$
and ${\rm spec}(\phi(e_I(\mathfrak{u}_1,\dots,\mathfrak{u}_n)))={\rm spec}(\phi(\widetilde{P}_I)\subset (-\frac{1}{8},\frac{1}{8})\bigcup (-\frac{7}{8},\frac{9}{8})$. But this implies that
$$\lim\limits_{m\rightarrow\infty}\|(e_I(\tilde{u}_{1m},\dots,\tilde{u}_{1m}))^2-e_I(\tilde{u}_{1m},\dots,\tilde{u}_{nm})\|<\frac{9}{64}<\frac{1}{4},$$
which is a contradiction.
\end{proof}

\begin{definition}\label{RI}
  Let $\Theta=(\theta_{jk})\in \mathcal{T}_n$ be strongly totally irrational. Let $\delta>0$ be chosen as in Proposition \ref{gap}. Let $\mathcal{A}$ be a unital $C^*$-algebra and let $u_1,\dots,u_n$ be an $n$-tuple of unitaries in $\mathcal{A}$ with $\|u_ku_j-e^{2\pi i\theta_{jk}}u_ju_k\|<\delta$ for $j,k=1,2,\dots,n.$ Let $\chi_{(\frac{1}{2},\infty)}$ be the characteristic function on $(\frac{1}{2},\infty).$ We define $R_{I}(u_1,\dots,u_n)=\chi_{(\frac{1}{2},\infty)}(e_{I}(u_1,u_2,\dots,u_n)).$ In particular, when $n=2,$ we have $R_{\theta_{12}}(u_1,u_2)=[R_{(1,2)}(u_1,u_2)],$ where $R_{\theta_{12}}(u_1,u_2)$ is defined as in Definition \ref{D_RieffelE}.
\end{definition}

Let $\mathcal{C} = PM_n(C(X))P$ be a homogenous algebra, where $X$ is a compact metric space
and $P$ is a projection in $M_n(C(X))$. Then for any tracial state $\tau$ on $\mathcal{C}$,
there is a probability measure $\mu$ on $X$ such that
\[
\tau(f) = \int_{x \in X} \mathrm{tr}_x(f(x))\,d\mu,
\]
where $\mathrm{tr}_x$ is the normalized trace on the fiber corresponding to $x$ (which is isomorphic to a matrix algebra).
We shall use $\mu_\tau$ to denote the measure associated to $\tau$.

Let $\mathcal{A}$ be a unital $C^*$-algebra with $T(\mathcal{A}) \neq \emptyset$.
By Theorem 3.2 of \cite{Thomsen-1995}, the de la Harpe-Scandalis determinant provides a continuous homomorphsim
\begin{equation}\label{E_HSdet}
\bar{\Delta} \colon U_0(M_{\infty}(\mathcal{A}))/CU(M_{\infty}(\mathcal{A}))
\rightarrow \mathrm{Aff}(T(\mathcal{A}))/\overline{\rho_A(\K_0(\mathcal{A}))}.
\end{equation}
By Corollary 3.3 of \cite{Thomsen-1995}, there is an induced split exact sequence

\begin{eqnarray}\label{E_Thomsen_es}
0 \rightarrow \mathrm{Aff}(T(\mathcal{A})) / \overline{\rho_A(\K_0(\mathcal{A}))}
\rightarrow U(M_{\infty}(\mathcal{A}))/CU(M_\infty(\mathcal{A}))
\xrightarrow[]{\pi_{\mathcal{A}}} \K_1(\mathcal{A}) \rightarrow 0.
\end{eqnarray}
The reader is referred to \cite{Thomsen-1995} for more details
of the homomorphism (\ref{E_HSdet}) and the exact sequence
(\ref{E_Thomsen_es}).

Since the exact sequence (\ref{E_Thomsen_es}) is split,
there is a homomorphism
\[
J_{\mathcal{A}} \colon \K_1(\mathcal{A}) \rightarrow U(M_{\infty}(\mathcal{A}))/CU(M_\infty(\mathcal{A}))
\]
such that $\pi_{\mathcal{A}} \circ J_{\mathcal{A}} = \mathrm{id}_{\K_1(\mathcal{A})}$.
We shall use $U_c(\K_1(\mathcal{A}))$ to denote a set of representatives
of $J_{\mathcal{A}}(\K_1(\mathcal{A}))$ in $U(M_{\infty}(\mathcal{A}))$.
For each $\mathcal{A}$, we shall fix a splitting map $J_{\mathcal{A}}$ and a set of representatives $U_c(\K_1(\mathcal{A}))$
if not mentioned explicitly.

For any two unitaries $u, v \in U(M_{\infty}(\mathcal{A}))$ such that
$uv^* \in U_0(M_{\infty}(\mathcal{A}))$, define
\[
\mathrm{dist}(\bar{u}, \bar{v}) = \| \bar{\Delta}(uv^*)\|,
\]
where the norm is the quotient norm on
$\mathrm{Aff}(T(\mathcal{A}))/\rho_{\mathcal{A}}(\K_0(\mathcal{A}))$.

The following lemma allows us to estimate the norm we just defined:
\begin{lemma}[Lemma 5.2 of \cite{Hua-Wang}] \label{L_Dist}
For any $\ep > 0$, there exists $\delta > 0$ such that, whenever $\mathcal{A}$ is a unital
$C^*$-algebra, if $u, v$ are unitaries in $U_{\infty}(\mathcal{A})$ such that $\|wuv^* - 1\| < \delta$
for some $w \in CU_{\infty}(\mathcal{A})$, then
\[
\mathrm{dist} (\bar{u}, \bar{v}) < \ep.
\]
\end{lemma}

\begin{definition} Let $X$ be a compact metric space, let $x\in X$ and let $r > 0$. Denote by $O_r(x)$
the open ball with center at $x$ and radius $r$. If $x$ is not specified, $O_r$ is an open ball of radius $r$.
\end{definition}

Let us begin with a brief outline of our strategy of proving stability of rotation relations of $n$ unitaries for $n\geq 2$. Let $\Theta\in \mathcal{T}_n.$ Suppose $u_1,u_2,\dots,u_n$ are $n$ unitaries in a unital $C^*$-algebra $\mathcal{A}$, and $u_1,u_2,\dots,u_n$ satisfy (\ref{app111}). Then there is an almost homomorphism from $\mathcal{A}_{\Theta}$ to $\mathcal{A}$. Now stability of relations (\ref{=111}) is equivalent to that this almost homomorphism is close to an actual homomorphism.

The latter problem is usually divided into two parts: the existence part and the uniqueness part. An almost homomorphism will induce an `almost' homomorphism between the invariants of the two $C^*$-algebras, where the invariant consists of the $\K$-theories. It is usually easier to show that an `almost' homomorphism of the invariants is close to an actual homomorphism. The existence part says that a homomorphism at the invariant level lifts to a homomorphism at
the $C^*$-algebra level. The uniqueness part says that, two almost homomorphisms which induces
`almost' the same maps on the invariants are almost unitarily equivalent. Therefore, conjugating
suitable unitaries, one shows that an almost homomorphism is close to an actual homomorphism.

Now we are ready to  state the uniqueness theorem and the existence theorem.
\begin{theorem}[Theorem 5.3 and Corollary 5.5 of \cite{Lin-JTA-2017}]\label{T_unique}
Let $\mathcal{C} = PM_n(C(X))P$, where $X$ is a compact metric space and $P$ is a projection in $M_n(C(X))$.
Let $\Delta \colon (0, 1) \rightarrow (0, 1)$
be a non-decreasing function such that $\lim_{r \rightarrow 0} \Delta(r) = 0$.
Let $\ep > 0$ and
let $\mathcal{F} \subset \mathcal{C}$ be a finite subset.
Then there exists $\eta> 0$, $\delta > 0$,
a finite subset $\mathcal{G} \subset \mathcal{C}$,
a finite subset $\mathcal{P} \subset \underline{\K}(\mathcal{C})$,
a finite subset $\mathcal{H} \subset \mathcal{C}_{s.a.}$ and
a finite subset $\mathcal{U} \subset U_c(\K_1(\mathcal{C}))$ satisfying the following:

Suppose that $\mathcal{A}$ is a unital separable simple $C^*$-algebra with $TR(\mathcal{A}) \leq 1$.
Suppose $L_1, L_2 \colon \mathcal{C} \rightarrow \mathcal{A}$ are two unital $\mathcal{G}$-$\delta$-multiplicative completely positive contractive
maps such that
\begin{align}
(L_1)_*\vert_\mathcal{P} & = (L_2)_*\vert_{\mathcal{P}},\phantom{aaaaaaaaaaaaaaaaaaaaa}\label{unique1}\nonumber\\
|\tau \circ L_1(g) - \tau \circ L_2(g)| & < \delta,
\quad \text{for all\,\,} g \in \mathcal{H} \nonumber\\
\mathrm{dist}(\overline{ \langle L_1(u) \rangle }, \overline{\langle L_2(u) \rangle}) & < \delta,\quad \text{for all\,\,} u \in \mathcal{U}\quad \text{and}\nonumber\\
\mu_{\tau \circ L_i}(O_r) & > \Delta(r), \quad i = 1, 2,\nonumber
\end{align}
for all $\tau \in T(\mathcal{A})$ and  all open balls $O_r$ in $X$ with radius $r\geq  \eta$. Then there exists a unitary $w \in \mathcal{A}$ such that
\[
\| \mathrm{Ad} w \circ L_1 (f) - L_2(f)\| < \ep
\text{\,\,for all\,\,} f \in \mathcal{F}.
\]
\end{theorem}

\begin{definition}Let $\mathcal{A}$ and $\mathcal{B}$ be two unital $C^*$-algebras.
Let $\kappa \in \KL_e(\mathcal{B}, \mathcal{A})^{++}$ and let $\lambda : T(\mathcal{A}) \rightarrow  T(\mathcal{B})$ be a continuous affine map.
We say that $\lambda$ is compatible with $\kappa$ if $\lambda$ is compatible with $\kappa|_{\K_0(\mathcal{B})}.$
(i.e. $\rho_{\mathcal{B}}([p])(\lambda(\tau)) = \rho_{\mathcal{A}}(\kappa([p]))(\tau)$, for any projection $p \in M_{\infty}(\mathcal{B})$ and any tracial state $\tau\in T(\mathcal{A})$).
Let
\[
\alpha \colon U(M_{\infty}(\mathcal{B}))/CU(M_{\infty}(\mathcal{B})) \rightarrow U(\mathcal{A})/CU(\mathcal{A})
\]
be a continuous homomorphism. Then by the exact sequence (\ref{E_Thomsen_es}),
there is an induced map $\af_1 \colon \K_1(\mathcal{B}) \rightarrow \K_1(\mathcal{A})$.
We say $\alpha$ and $\kappa$ are compatible, if $\alpha_1 = \kappa|_{\K_1(\mathcal{B})}$.
We say $\alpha, \lambda$ and $\kappa$ are compatible if $\lambda, \kappa$ are compatible, $\alpha, \lambda$ are compatible
and $\alpha, \kappa$ are compatible.
\end{definition}

\begin{theorem}[Theorem 6.11 of \cite{Lin-JTA-2017}]\label{T_exist}
Let $\mathcal{B}$ be a unital AH-algebra and let $\mathcal{A}$ be a unital infinite
dimensional separable simple $C^*$-algebra with $TR(\mathcal{A}) \leq 1$.
Then for any $\kappa \in \KL_e(\mathcal{B}, \mathcal{A})^{++}$,
any affine continuous map $\gamma \colon T(\mathcal{A}) \rightarrow T_f(\mathcal{B})$
and any continuous homomorphism
\[
\alpha \colon U(M_{\infty}(\mathcal{B}))/CU(M_{\infty}(\mathcal{B})) \rightarrow U(\mathcal{A})/CU(\mathcal{A})
\]
such that $\kappa, \gamma$ and $\alpha$ are compatible, there is a unital monomorphism
$\varphi \colon \mathcal{B} \rightarrow \mathcal{A}$ such that
\[
\varphi_* = \kappa, \quad \varphi_\sharp = \gamma \quad \text{and\quad}
\varphi^{\ddagger} = \alpha.
\]
\end{theorem}

Below we will always use ``$c. p. c.$'' for the abbreviation of ``completely positive contractive''.

\begin{proposition}\label{L-exist} Let $\Theta=(\theta_{jk})\in \mathcal{T}_n.$ Let $\mathfrak{u}_1,\mathfrak{u}_2,\dots,\mathfrak{u}_n$ be the canonical generators of $\mathcal{A}_{\Theta}.$ Then for any finite subset $\mathcal{G}\subset \mathcal{A}_{\Theta},$ any $\eta>0$ and any $\varepsilon>0,$ there exists $\delta>0$ such that for any unital $C^*$-algebra $\mathcal{A},$ any  $n$-tuple of unitaries $u_1,u_2,\dots,u_n$ in $\mathcal{A}$ satisfying
$$\|u_ku_j-e^{2\pi i\theta_{jk}}u_ju_k\|<\delta\,\,  for\,\, all\,\,j,k=1,2,\dots,n, $$
there is a unital $\mathcal{G}$-$\eta$-multiplicative c.p.c. map $L:\mathcal{A}_{\Theta}\rightarrow\mathcal{A}$ such that
$$\|L(\mathfrak{u}_j)-u_j\|<\varepsilon\,\, for\,\,j=1,2,\dots,n.$$
\end{proposition}
\begin{proof}
  Assume that the proposition is false. Let $(\delta_m)_{m=1}^{\infty}$ be a sequence of positive numbers decreasing to $0.$ Then there is a finite subset $\mathcal{G}\subset \mathcal{A}_{\Theta},$ some $\varepsilon,\eta>0$ such that for any integer $m>0,$ there is a unital $C^*$-algebra $\mathcal{A}_m$, an $n$-tuple of unitaries $u_1^{(m)},u_2^{(m)},\dots,u_n^{(m)}$ in $\mathcal{A}_m$ satisfying
  \begin{center}
    $\|u_k^{(m)}u_j^{(m)}-e^{2\pi i\theta_{jk}}u_j^{(m)}u_k^{(m)}\|<\delta_m,$
  \end{center}
but for any unital $\mathcal{G}$-$\eta$-multiplicative c.p.c. map $\varphi_m:\mathcal{A}_{\Theta}\rightarrow
\mathcal{A}_m,$ we have
$$\|\varphi_{m}(\mathfrak{u}_{j_m})-u_{j_m}^{(m)}\|\geq \varepsilon\,\, {\rm for}\,\, {\rm some} \,\, j_m\in \{1,2,\dots,n\}.$$

Set $\mathcal{B}=\prod_{m=1}^{\infty} \mathcal{A}_m\slash \oplus_{m=1}^{\infty}\mathcal{A}_m.$
Let $\ell:\prod_{m=1}^{\infty} \mathcal{A}_m\rightarrow \mathcal{B}$ be the canonical quotient map.
Let $u_j=(u_j^{(m)})\in \prod_{m=1}^{\infty}\mathcal{A}_m$ for $j=1,2,\dots,n.$ Then
$\{\ell(u_j)\}_{j=1}^n$ are unitaries in $\mathcal{B}$ satisfying
 $$
\ell(u_k)\ell(u_j)=e^{2\pi i\theta_{jk}}\ell(u_j)\ell(u_k).
$$
 Therefore there is a unital homomorphism $\psi:\mathcal{A}_{\Theta}\rightarrow \mathcal{B}$ such that
 $$\psi(\mathfrak{u}_j)=\ell(u_j)\,\, {\rm for}\,\, j=1,2,\dots,n.$$
  It is well known that  $\mathcal{A}_{\Theta}$ is nuclear. By the Choi--Effros lifting theorem, we can lift $\psi$ to a unital c.p.c. map
 $$\widetilde{\psi}=(\psi_1,\psi_2,\dots,\psi_m,\dots):\mathcal{A}_{\Theta}\rightarrow\prod_{m=1}^{\infty}\mathcal{A}_m$$
 such that $\psi=\ell \circ\widetilde{\psi}.$
 In particular, each coordinate map $\psi_m$ is unital completely positive, and we can also assume that they are contractive by normalization. By choosing $m$ large enough, we can make sure that $\psi_m$ is $\mathcal{G}$-$\eta$-multiplicative. From our construction,
 $$\lim\limits_{m\rightarrow \infty}\|\psi_m(\mathfrak{u}_j)-u_j^{(m)}\|=0\,\, {\rm for}\,\, j=1,2,\dots,n.$$
 This is a contradiction.
 \end{proof}

The following lemma follows from functional calculus and the fact that norm close projections (or unitaries) are equivalent.
\begin{lemma}\label{app}Let $\Theta\in\mathcal{T}_n$ be strongly totally irrational. Let $\mathfrak{u}_1,\mathfrak{u}_2,\dots,\mathfrak{u}_n$ be the canonical generators of $\mathcal{A}_{\Theta}.$ Then there is a finite subset $\mathcal{G}\subset \mathcal{A}_{\Theta},$ some $\eta>0$ and $\varepsilon>0,$ such that for any unital $C^*$-algebra $\mathcal{A},$ any $n$-tuple of unitaries $u_1,u_2,\dots,u_n$ in $\mathcal{A},$ if $L:\mathcal{A}_{\Theta}\rightarrow \mathcal{A}$ is a $\mathcal{G}$-$\eta$-multiplicative c.p.c. map such that
$$\|L(\mathfrak{u}_j)-u_j\|<\varepsilon,\quad j=1,2,\dots,n,$$
then
$$L_{*1}([\mathfrak{u}_j])=[u_j],\quad j=1,2,\dots,n$$
and
$$L_{*0}([P_{I}])=[R_{I}(u_1,u_2,\dots,u_n)],\,\, \text{for all $I$ defined as in Definition \ref{eI}}. $$
\end{lemma}

\begin{lemma}[Lemma 4.1 of \cite{LnTAM}]\label{tracial app} Let $\mathcal{A}$ be a separable unital $C^*$-algebra. For any $\varepsilon>0$ and any finite subset $\mathcal{F}\subset \mathcal{A}_{s.a.},$ there exists $\eta>0$ and a finite subset $\mathcal{G}\subset \mathcal{A}_{s.a.}$
satisfying the following: For any $\mathcal{G}$-$\eta$-multiplicative c.p.c. map $L:\mathcal{A}\rightarrow \mathcal{B}$, for any unital $C^*$-algebra $\mathcal{B}$ with $T(\mathcal{B})\neq \emptyset$, and any tracial state $t\in T(\mathcal{B}),$ there exists a
$\tau\in T(\mathcal{A})$ such that
\begin{center}
  $\|t\circ L(a)-\tau(a)\|<\varepsilon$ for all $a\in \mathcal{F}.$
\end{center}
\end{lemma}

We say that $\Theta=(\theta_{jk})\in\mathcal{T}_n$ is {\it not rational} if $\theta_{jk},1\leq j<k\leq n$ are not all rational. Note that if $\Theta$ is nondegenerate, then it is not rational. If $\Theta$ is totally irrational, then it is also not rational.



\begin{theorem}[Theorem 8.3 of \cite{Rieffel-1988}]\label{rep K1}
  If $\Theta$ is not rational, then the natural map from $U(\mathcal{A}_{\Theta})\slash U_0(\mathcal{A}_{\Theta})$ to $\K_1(\mathcal{A}_{\Theta})$ is an isomorphism.
\end{theorem}

\begin{lemma}\label{dist(u)} Let $\Theta=(\theta_{jk})\in\mathcal{T}_n$ be totally irrational. Let $\pi_{\mathcal{A}_{\Theta}}:U_{\infty}(\mathcal{A}_{\Theta})\slash CU_{\infty}(\mathcal{A}_{\Theta})\rightarrow \K_1(\mathcal{A}_{\Theta})$ be the natural surjective homomorphism. Let $[u'_j],~ j=1,2,\dots,2^{n-1}$ be a set of generators of $\K_1(\mathcal{A}_{\Theta}).$ Let $J_{\mathcal{A}_{\Theta}}:\K_1(\mathcal{A}_{\Theta})\rightarrow U_{\infty}(\mathcal{A}_{\Theta})\slash CU_{\infty}(\mathcal{A}_{\Theta})$ be the embedding induced by
$$[u_j']\rightarrow [\overline{u_j'}],\quad j=1,2,\dots,2^{n-1}.$$
Let $\{\overline{w}_1,\overline{w}_2,\dots,\overline{w}_m\}$ be a finite subset of
$J_{\mathcal{A}_{\Theta}}(\K_1(\mathcal{A}_{\Theta})).$ Then for every $\varepsilon>0,$ there is a finite subset $\mathcal{G}\subset \mathcal{A}_{\Theta}$ and a $\eta>0$ such that, for any unital $C^*$-algebra $\mathcal{A}$ and any unital $\mathcal{G}$-$\eta$-multiplicative c.p.c. map $L:\mathcal{A}_{\Theta}\rightarrow \mathcal{A},$ the homomorphism
$$\alpha_1:J_{\mathcal{A}_{\Theta}}(\K_1(\mathcal{A}_{\Theta}))\rightarrow U_{\infty}(\mathcal{A})\slash CU_{\infty}(\mathcal{A}),\quad \alpha_1(\overline{u_j'})=\overline{\langle L(u_j')\rangle},\, j=1,2,\dots,2^{n-1}$$
satisfies dist$(\overline{\langle L(w_j)\rangle},\alpha_1(\overline{w}_j))<\varepsilon,$ for $j=1,2,\dots,m.$
\end{lemma}
\begin{proof}
  Let $\varepsilon>0$ be given. Choose $\delta>0$ according to Lemma 5.2 of \cite{Hua-Wang}. We can also choose $u_j'$ to be in $\mathcal{A}_{\Theta}$ according to Theorem \ref{rep K1}. Let $\mathcal{F}=\{u_j',(u_j')^{-1}|~j=1,2,\dots,2^{n-1}\}.$ For each $k\in\mathbb{N},$ define
  $$\mathcal{F}_k=\{{\rm diag}(a_1,a_2,\dots,a_k)~|~a_j\in\mathcal{F},\,\,j=1,2,\dots,k\}.$$
  It is easy to see that, there is some $N\in \mathbb{N}$ large enough such that, for $j=1,2,\dots,m,$ we can find an element $b_j\in \mathcal{F}_{N}$ and unitaries $v_j,s_j$ such that
  $$\|w_j-b_j v_js_jv^*_{j}s_j^{*}\|<\frac{\delta}{4}\,\, {\rm and}\,\, \alpha_1(\overline{w}_j)=\alpha_1(\overline{b}_j).$$
  Since $\langle L(\cdot)\rangle$ is defined by continuous functional calculus on the set of unitaries, there is a $\eta_1$ such that whenever $L:\mathcal{A}_{\Theta}\rightarrow \mathcal{A}$
  is a unital $\mathcal{F}$-$\eta_1$-multiplicative c.p.c. map, we have
  $\|\langle L(u_j')\rangle^{-1}-\langle L((u_j')^{-1})\rangle\|<\frac{\delta}{4}$, for $j=1,2,\dots,2^{n-1}.$

  Let $\mathcal{G}_1=\{v_j,s_j,v_j^*,s_j^*| j=1,2,\dots,m\}.$ There is a $\eta_2$ such that whenever $L:\mathcal{A}_{\Theta}\rightarrow \mathcal{A}$ is a unital $\mathcal{G}_1$-$\eta_2$-multiplicative c.p.c. map, we have
  $$\|\langle L(v_js_jv_j^*s_j^*)\rangle-\langle L(v_j)\rangle \langle L(s_j)\rangle \langle L(v_j)\rangle^*\langle L(s_j)\rangle^*\|<\frac{\delta}{4}\,\,\, {\rm for}\,\,j=1,2,\dots,m.$$

  Let $\mathcal{G}=\mathcal{F}\cup \mathcal{G}_1$ and let $\eta=\min\{\eta_1,\eta_2,\frac{\delta}{4}\}.$ Let $\mathcal{A}$ be a unital $C^*$-algebra. Suppose that $L:\mathcal{A}_{\Theta}\rightarrow \mathcal{A}$ is a unital $\mathcal{G}$-$\eta$-multiplicative c.p.c. map. Define $\alpha_1$ to be the homomorphism:
  $$\alpha_1:J_{\mathcal{A}_{\Theta}}(\K_1(\mathcal{A}_{\Theta}))
  \rightarrow U_{\infty}(\mathcal{A})\slash CU_{\infty}(\mathcal{A}),\quad \alpha_1(\overline{u_j'})=\overline{\langle L(u_j')\rangle},\,j=1,2,\dots,2^{n-1}.$$
  By our construction, for any $b_j\in \mathcal{F}_{N},$ there is a unitary $c_j\in U_{\infty}(\mathcal{A})$ such that
  $$\alpha_1(\overline{b}_j)=\overline{c}_j\,\,{\rm and}\,\, \|c_j-\langle L(b_j)\rangle\|<\frac{\delta}{4}.$$
  Therefore we have $\alpha_1(\overline{w}_j)=\overline{c}_j.$ Moreover, we can compute that
  \begin{eqnarray*}
     && \|c_j-\langle L(w_j)\rangle \langle L(v_j)\rangle \langle L(s_j)\rangle\langle L(v_j)\rangle^* \langle L(s_j)\rangle^*\| \\
     &<&  \|c_j-\langle L(b_jv_js_jv_j^*s_j^*)\rangle \langle L(v_j)\rangle \langle L(s_j)\rangle\langle L(v_j)\rangle^* \langle L(s_j)\rangle^*\|+\frac{\delta}{4} \\
     &<& \eta+\frac{\delta}{4}+\frac{\delta}{4}+\|c_j-\langle L(b_j)\rangle\|<\delta.
  \end{eqnarray*}
By Lemma \ref{L_Dist}, we have dist$(\alpha_1(\overline{w}_j),\overline{\langle L(w_j)\rangle})<\varepsilon,$ for $j=1,2,\dots,m.$
\end{proof}

 \begin{lemma}[Lemma 5.11 of \cite{Hua-Wang}] \label{L_Delta}
Let $(X, d)$ be an infinite compact metric space. Let $\mathcal{C} = PM_n(C(X))P$, where $P$ is a projection in $M_n(C(X))$.
Let $\tau_0$ be a faithful tracial state on $\mathcal{C}$.
Define a function $\Delta \colon (0, 10) \rightarrow (0, 1)$ by
\[
\Delta(r) = \min \{ \mu_{\tau_0}(O_{\frac{r}{10}}(x)) \,\vert\,  x \in X\},
\quad 0 < r < 10.
\]
Then $\Delta$ is a strictly increasing function such that $\lim_{r \rightarrow 0} \Delta(r) = 0$.

Moreover, for any $\eta > 0$, there exists a finite subset $\mathcal{H} \in \mathcal{C}_+ \backslash \{ 0 \}$
and a $\delta > 0$ such that, whenever $L \colon \mathcal{C} \rightarrow \mathcal{A}$ is a c.p.c. map
and $\tau$ is a tracial state on $\mathcal{A}$, if
\[
|\tau \circ L (h) - \tau_0(h)| < \delta \quad \text{for all \,\,} h \in \mathcal{H},
\]
then $\mu_{\tau \circ L}(O_r) > \Delta(r)$ for all $r \geq \eta$.
\end{lemma}

Let $\Theta=(\theta_{jk})\in \mathcal{T}_n$ be strongly totally irrational. Let $\varepsilon>0$ be given. Let $\mathfrak{u}_1,\mathfrak{u}_2,\dots,\mathfrak{u}_n$ be the generators of $\mathcal{A}_{\Theta}$ satisfying
$$\mathfrak{u}_k\mathfrak{u}_j=e^{2\pi i\theta_{jk}}\mathfrak{u}_j\mathfrak{u}_k\,\, {\rm for}\,\, j,k=1,2,\dots,n.$$
Recall that  $\mathcal{A}_{\Theta}$
is an $\mathrm{AT}$-algebra with the unique tracial state $\tau_{\Theta},$ and $\K_0(\mathcal{A}_{\Theta})\cong \K_1(\mathcal{A}_{\Theta})\cong \mathbb{Z}^{2^{n-1}},$ by Theorem \ref{AT and K}. By Lemma 4.1.2 of \cite{Lin-book}
we may assume $\mathcal{A}_{\Theta}=\overline{\cup_{k=1}^{\infty}\mathcal{B}_k}$, $\mathcal{B}_k\subset \mathcal{B}_{k+1},k=1,2,\dots$, where each
$\mathcal{B}_k$ is a finite direct sum of $C^*$-algebras of the form $M_m(C(X)),$ where $X$ is $S^1$, an arc or a point. Set $\mathcal{F}=\{1_{\mathcal{A}_{\Theta}},\mathfrak{u_1},\mathfrak{u}_2,\dots,\mathfrak{u}_n\}.$ Choose a sufficiently large $k$ such that for every $a$ in $\mathcal{F},$ there is an $\tilde{a}\in \mathcal{B}_k$ such that $\|a-\tilde{a}\|<\varepsilon.$ Next we  denote $\mathcal{B}_k$ by $\mathcal{C}.$
For simplicity, in the proof of Lemma \ref{L} and Theorem \ref{main result} we assume that $\mathcal{C}=M_m(C(X)).$ The proof for $\mathcal{C}$ being the direct sum of such algebras is  similar. Next we use $\iota:\mathcal{C}\rightarrow \mathcal{A}_{\Theta}$ to denote the inclusion map. Define a function $\Delta:(0,10)\rightarrow (0,1)$ by
 $$\Delta(r)=\min\{\mu_{\tau_{\mathcal{C}}}(O_{\frac{r}{10}}(x))\},\quad 0<r<10,$$
 where $\tau_{\mathcal{C}}=\tau_{\mathcal{A}_{\Theta}}|_{\mathcal{C}}.$

\begin{lemma}\label{L} Let $\Theta=(\theta_{jk})\in \mathcal{T}_n$ be strongly totally irrational. Let $\mathfrak{u}_1,\mathfrak{u}_2,\dots,\mathfrak{u}_n$ be the canonical generators of $\mathcal{A}_{\Theta}.$ Set $\mathcal{F}=\{1_{\mathcal{A}_{\Theta}},\mathfrak{u_1},\mathfrak{u}_2,\dots,\mathfrak{u}_n\}.$ Let $\mathcal{C}$ and $\Delta$ be as in the above paragraph.
Then for any finite subset $\mathcal{G}\subset \mathcal{A}_{\Theta},$  and any $\varepsilon>0,$ choose $\eta>0,\delta_1$ (in place of $\delta$), a finite subset $\widetilde{\mathcal{G}}\subset \mathcal{C},$ a finite subset $\widetilde{\mathcal{P}}\subset \underline{\K}(\mathcal{C}),$ a finite subset $\widetilde{\mathcal{H}}_0\subset \mathcal{C}_{s.a.}$ and a finite subset $\widetilde{\mathcal{U}}\subset
 U_{c}(\K_1(\mathcal{C}))$ with respect to $\Delta,$ $\varepsilon$  and $\widetilde{\mathcal{F}}$ (in place of $\mathcal{F}$) according to Theorem \ref{T_unique}. There exists $\delta>0$ such that for any unital $C^*$-algebra $\mathcal{A},$ any  $n$-tuple of unitaries $u_1,u_2,\dots,u_n$ in $\mathcal{A}$ satisfying
$$\|u_ku_j-e^{2\pi i\theta_{jk}}u_ju_k\|<\delta\,\,  for\,\, all\,\,j,k=1,2,\dots,n, $$
there is  a unital
$\mathcal{G}$-$\delta_1$-multiplicative c.p.c. map $L:\mathcal{A}_{\Theta}\rightarrow\mathcal{A}$ such that:\\
(1) there is a finite subset $\widetilde{\mathcal{F}}\subset \mathcal{C}$,  such that for any $a\in \mathcal{F}$, there exists $\tilde{a}\in\widetilde{\mathcal{F}}$ such that
$$\|\iota(\tilde{a})-a\|<\varepsilon,$$\\
(2)
$$\|L(\mathfrak{u}_j)-u_j\|<\varepsilon\,\, for\,\,j=1,2,\dots,n,$$
and
 $$L_{*1}([\mathfrak{u}_j])=[u_j]\,\, for\,\, j=1,2,\dots,n,$$
and
$$L_{*0}([P_I])=[R_{I}(u_1,u_2,\dots,u_n)]\,\,\text {for all $I$ defined as in Definition \ref{eI}},$$
where $P_I$ is as in Theorem \ref{PI} and  $R_{I}(u_1,u_2,\dots,u_n)$ is defined as in Definition \ref{RI}.\\
 (3) the homomorphism
$$\alpha_1:J_{\mathcal{A}_{\Theta}}(\K_1(\mathcal{A}_{\Theta}))\rightarrow U_{\infty}(\mathcal{A})\slash CU_{\infty}(\mathcal{A}),\quad \alpha_1(\overline{\mathfrak{u}}_j)=\overline{\langle L(\mathfrak{u}_j)\rangle},\,\, {\rm for}\,\, j=1,2,\dots,n$$
satisfies dist$(\overline{\langle L(\iota(s))\rangle},\alpha_1(\overline{\iota(s)}))<\delta_1$ for all $s\in \widetilde{\mathcal{U}}.$ \\
(4) put  $\widetilde{L}=L\circ \iota$, we have $$\|\tau\circ\widetilde{L}(h)-\tau_{\Theta}(\iota(h))\|<\delta_1 \,\, {\rm for}\, {\rm all}\,\, h\in \widetilde{\mathcal{H}}_0.$$
  Moreover,
$$\mu_{\tau\circ \widetilde{L}}(O_r)>\Delta (r)$$ for all $\tau \in T(\mathcal{A})$ and  all open balls $O_r$ in $X$ with radius $r\geq  \eta.$
\end{lemma}
\begin{proof} For (1), it is obvious according to the discussion above this lemma.

 Next we use $[u_1'],[u_2'],\dots,[u_{2^{n-1}}']$ to denote a set of generators of $\K_1(\mathcal{A}_{\Theta})$.
We define a splitting map:
$$J_{\mathcal{A}_{\Theta}}:\K_1(\mathcal{A}_{\Theta})\rightarrow U_{\infty}(\mathcal{A}_{\Theta})\slash CU_{\infty}(\mathcal{A}_{\Theta}),\quad J_{\mathcal{A}_{\Theta}}([u_j'])=\overline{u_j'},\, j=1,2,\dots,2^{n-1}.$$
 Then there is a splitting map $J_{\mathcal{C}}:\K_1(\mathcal{C})\rightarrow U_{\infty}(\mathcal{C})\slash CU_{\infty}(\mathcal{C})$ such that the following diagram commutes:
$$\xymatrix{
 U_{\infty}(\mathcal{C})\slash CU_{\infty}(\mathcal{C}) \ar[d]^{\iota^{\ddag}} & \ar[l]_{\phantom{aaaaa}J_{\mathcal{C}}} K_1(\mathcal{C})
    \ar[d]^{\iota_*} \\
 U_{\infty}(\mathcal{A}_{\Theta})\slash CU_{\infty}(\mathcal{A}_{\Theta}) &\ar[l]_{\phantom{aaaaaa}J_{\mathcal{A}_{\Theta}}} \K_1(\mathcal{A}_{\Theta}).}$$

 Choose $\eta>0,\delta_1$ (in place of $\delta$), a finite subset $\widetilde{\mathcal{G}}\subset \mathcal{C},$ a finite subset $\widetilde{\mathcal{P}}\subset \underline{\K}(\mathcal{C}),$ a finite subset $\widetilde{\mathcal{H}}_0\subset \mathcal{C}_{s.a.}$ and a finite subset $\widetilde{\mathcal{U}}\subset
 U_{c}(\K_1(\mathcal{C}))$ with respect to $\Delta,$ $\varepsilon$  and $\widetilde{\mathcal{F}}$ (in place of $\mathcal{F}$) according to Theorem \ref{T_unique}.

  By Lemma \ref{L_Delta}, there exists a finite subset $\widetilde{\mathcal{H}}_1\in \mathcal{C}_{+}\backslash \{0\}$ and a $\delta_2\leq \delta_1$ such that, for any $C^*$-algebra $\mathcal{A}$, whenever $\widetilde{L}:\mathcal{C}\rightarrow \mathcal{A}$ is a c.p.c. map and $\tau$ is a tracial state on $\mathcal{A},$ if
 \begin{center}$\|\tau\circ \widetilde{L}(h)-\tau_{\mathcal{C}}(h)\|<\delta_2$ for all $h\in \widetilde{\mathcal{H}}_1,$
 \end{center}
 then $\mu_{\tau\circ \widetilde{L}}(O_r)> \Delta(r),$ for all $r\geq \eta.$

 Let $\mathcal{H}=\iota(\widetilde{\mathcal{H}}_0\cup\widetilde{\mathcal{H}}_1)$. Choose a finite subset $\mathcal{G}_0\subset \mathcal{A}_{\Theta}$, a $\delta_3>0$ and a positive integer $N$ with respect to $\mathcal{H}$ and $\delta_2$ (in place of $\varepsilon$) according to Lemma \ref{tracial app}.

 Let $\overline{\mathcal{U}}=\{\overline{\iota(s)}|s\in \widetilde{\mathcal{U}}\}.$ Then $\overline{\mathcal{U}}$ is a finite subset of
 $J_{\mathcal{A}_{\Theta}}(\K_1(\mathcal{A}_{\Theta})).$ By Lemma \ref{dist(u)}, there is a finite subset $\mathcal{G}_1$ and a $\delta_4$ such that, for any unital $C^*$-algebra $\mathcal{A}$ and any unital $\mathcal{G}_1$-$\delta_4$-multiplicative c.p.c. map $L:\mathcal{A}_{\Theta}\rightarrow \mathcal{A},$ there is a homomorphism
 $$\alpha_1:J_{\mathcal{A}_{\Theta}}(\K_1(\mathcal{A}_{\Theta}))\rightarrow U_{\infty}(\mathcal{A})\slash CU_{\infty}(\mathcal{A})$$
 such that dist$(\overline{\langle L(\iota(s))\rangle},\alpha_1(\overline{\iota(s)}))<\delta_1$ for all $s\in \widetilde{\mathcal{U}}.$

 Let $\mathcal{P}=\iota_*(\widetilde{\mathcal{P}}).$ Since $\K_*(\mathcal{A}_{\Theta})$ is torsion free, by Proposition 2.4 of \cite{Schochet-1984}, we have $\K_*(\mathcal{A}_{\Theta};\mathbb{Z}\slash m\mathbb{Z})\cong \K_*(\mathcal{A}_{\Theta})\otimes \mathbb{Z}\slash m\mathbb{Z}$ for any $m\in \mathbb{Z}_+.$ By the description of $\K_*(\mathcal{A}_{\Theta}),$ we may assume without loss of generality that
 $\mathcal{P}=\{[1_{\mathcal{A}_{\Theta}}],[u_1'],[u_2'],\dots,[u_{2^{n-1}}']\}\bigcup
 \{ [P_I]| I=(i_1,i_2,\dots,i_{l}), 1\leq i_1<i_2<\cdots<i_l\leq \lfloor\frac{n}{2}\rfloor,l=2m, m\in\{1,2,\dots,\lfloor\frac{n}{2}\rfloor\}\}$.

 By Lemma \ref{app}, we can choose a finite subset $\mathcal{G}_2\subset \mathcal{A}_{\Theta}$, $0<\varepsilon_0<\varepsilon$ and $\delta_5>0$ so that
 whenever $u_1,u_2,\dots,u_n$ are unitaries in $\mathcal{A}$ and $L:\mathcal{A}_{\Theta}\rightarrow \mathcal{A}$ is a $\mathcal{G}_2$-$\delta_5$-multiplicative c.p.c. map such that
 $$\|L(\mathfrak{u}_j)-u_j\|<\varepsilon_0\,\, {\rm for}\,\, j=1,2,\dots,n,$$
 then
 $$L_{*0}([P_I])=[R_{I}(u_1,u_2,\dots,u_n)]\,\,{\rm for}\,\, {\rm every\,\, possible}\,\, I,$$
 and
 $$L_{*1}([\mathfrak{u}_j])=[u_j]\,\,{\rm for}\,\, j=1,2,\dots,n.$$

Set $\mathcal{G}=\iota(\widetilde{\mathcal{G}})\cup \mathcal{G}_0\cup \mathcal{G}_1\cup \mathcal{G}_2.$ Let $\delta_6=\min\{\delta_1,\delta_2,\delta_3,\delta_4,\delta_5\}.$ Find a positive number $\delta\leq \delta_6$ according to $\mathcal{G},$ $\delta_6$ (in place of $\eta$) and $\varepsilon_0>0$ (in place of $\varepsilon$) as in Proposition \ref{L-exist}.

Now suppose that $\mathcal{A}$ is a unital simple $C^*$-algebra with tracial rank at most one. Let $u_1,u_2,\dots,u_n\in \mathcal{A}$ be unitaries such that
\begin{eqnarray*}
\,\|u_ku_j-e^{2\pi i\theta_{jk}}u_ju_k\|<\delta,\,\,j,k=1,2,\dots,n.
\end{eqnarray*}

  By Proposition \ref{L-exist}, there is a unital $\mathcal{G}$-$\delta_6$-multiplicative c.p.c. map
$L:\mathcal{A}_{\Theta}\rightarrow \mathcal{A}$ such that $\|L(\mathfrak{u}_j)-u_j\|<\varepsilon_0$ for $j=1,2,\dots,n$. By Lemma \ref{dist(u)}, the homomorphism
$$\alpha_1:J_{\mathcal{A}_{\Theta}}(\K_1(\mathcal{A}_{\Theta}))\rightarrow U_{\infty}(\mathcal{A})\slash CU_{\infty}(\mathcal{A}),\quad \alpha_1(\overline{\mathfrak{u}}_j)=\overline{\langle L(\mathfrak{u}_j)\rangle},\,\, {\rm for}\,\, j=1,2,\dots,n$$
satisfies dist$(\overline{\langle L(\iota(s))\rangle},\alpha_1(\overline{\iota(s)}))<\delta_1$ for all $s\in \widetilde{\mathcal{U}}.$

Since $\tau_{\Theta}$ is the only tracial state on $\mathcal{A}_{\Theta}$, by Lemma \ref{tracial app}, we have
$$\|\tau\circ L(h)-\tau_{\Theta}(h)\|<\delta_2\,\, {\rm for}\, {\rm all}\,\, h\in \mathcal{H}.$$

Therefore
$$\|\tau\circ\widetilde{L}(h)-\tau_{\Theta}(\iota(h))\|=\|\tau\circ L(\iota(h))-\tau\circ \widetilde{\phi}(h)\|<\delta_2\leq \delta_1$$
for all $\tau\in T(\mathcal{A})$ and $h\in \widetilde{\mathcal{H}}_0.$

The same computation shows that
$$\|\tau\circ\widetilde{L}(h)-\tau_{\Theta}(\iota(h))\|<\delta_2 \,\, {\rm for}\, {\rm all}\,\, h\in \widetilde{\mathcal{H}}_1.$$

So by Lemma \ref{L_Delta}, we have
$$\mu_{\tau\circ \widetilde{L}}(O_r)>\Delta (r)\,\, {\rm for}\,{\rm all}\,\,r\geq \eta.$$
\end{proof}

Now let us prove our main result.
\begin{theorem}\label{main result}
  Let $\Theta=(\theta_{jk})\in \mathcal{T}_n$ be strongly totally irrational. Then for any $\varepsilon>0,$ there exists $\delta>0$ satisfying the following: For any unital simple separable $C^*$-algebra $\mathcal{A}$ with tracial rank at most one, any $n$-tuple of unitaries $u_1,u_2,\dots,u_n$ in $\mathcal{A}$ such that
 \begin{eqnarray}\|u_ku_j-e^{2\pi i\theta_{jk}}u_ju_k\|<\delta,\,\,\,j,k=1,2,\dots,n,\label{app11}\nonumber
 \end{eqnarray}
  and
 \begin{eqnarray}\label{t11}\tau(R_{I}(u_1,u_2,\dots,u_n))={\rm pf}(M_{I}^{\Theta})+\sum\limits_{0<|J|< l}{\rm pf}(M_{J}^{\Theta})k_{I_J}+k_{I_0}\label{assump2}
  \end{eqnarray}
  for all possible  $I=(i_1,i_2,\dots,i_{l})$, $ 1\leq i_1<i_2<\cdots<i_l\leq n,$ $l=2m,$ $m\in\{1,2,\dots,\lfloor\frac{n}{2}\rfloor\},$ and all tracial states $\tau$ on $\mathcal{A}$, where  $R_{I}(u_1,u_2,\dots,u_n)$ is defined as in Definition \ref{RI},
  $k_{I_J},k_{I_0}\in \mathbb{Z}$  is the same as  in Corollay \ref{PI},
  $J=(i_1,i_2,\dots,i_{s})$,  $|J|=s;$ there exists an $n$-tuple of unitaries $\tilde{u}_1,\tilde{u}_2,\dots,\tilde{u}_n$ in $\mathcal{A}$ such that
 \begin{eqnarray}\tilde{u}_k\tilde{u}_j=e^{2\pi i\theta_{jk}}\tilde{u}_j\tilde{u}_k\,\, {\rm and}\,\, \|\tilde{u}_j-u_j\|<\varepsilon,\,\, j,k=1,2,\dots,n.\label{=}\nonumber
 \end{eqnarray}
\end{theorem}
\begin{proof}Let $\varepsilon>0$ be given. Let $\mathfrak{u}_1,\mathfrak{u}_2,\dots,\mathfrak{u}_n$ be the generators of $\mathcal{A}_{\Theta}$ satisfying
$$\mathfrak{u}_k\mathfrak{u}_j=e^{2\pi i\theta_{jk}}\mathfrak{u}_j\mathfrak{u}_k\,\, {\rm for}\,\, j,k=1,2,\dots,n.$$
Set $\mathcal{F}=\{1_{\mathcal{A}_{\Theta}},\mathfrak{u_1},\mathfrak{u}_2,\dots,\mathfrak{u}_n\}.$  Choose $\mathcal{C}$ using Lemma \ref{L} (replace $\varepsilon$ with $\frac{\varepsilon}{3})$, so there is a finite subset $\widetilde{\mathcal{F}}\subset\mathcal{C}$ such that for every $a$ in $\mathcal{F},$ there is an $\tilde{a}\in \widetilde{\mathcal{F}}$ such that $\|a-\tilde{a}\|<\frac{\varepsilon}{3}.$ We use $\iota:\mathcal{C}\rightarrow \mathcal{A}_{\Theta}$ to denote the inclusion map.

 Note that $[1]$ and $\{[P_{I}]|I=(i_1,i_2,\dots,i_{l}), 1\leq i_1<i_2<\cdots<i_l\leq n, l=2m, m=1,2,\dots,\lfloor\frac{n}{2}\rfloor\}$ are the generators of $\K_0(\mathcal{A}_{\Theta})$ by Theorem \ref{PI}. We define $\kappa_0:\K_0(\mathcal{A}_{\Theta})\rightarrow \K_0(\mathcal{A})$ to be the homomorphism induced by
$$\kappa_0([1_{\mathcal{A}_{\Theta}}])=[1_{\mathcal{A}}],\, \, \kappa_0([P_I])=[R_{I}(u_1,u_2,\dots,u_n)]\,\,{\rm for\,\,all\,\, possible}\,\, I.$$
We claim that this is a positive homomorphism. Indeed, let $[p]\in \K_0(\mathcal{A}_{\Theta})$ be a positive element. Then $(\tau_{\Theta})_*([p])>0.$
There are  integers $n_0,n_{I}$ for every possible $I$ such that $$[p]=n_0[1_{\mathcal{A}_{\Theta}}]+\sum\limits_I n_I [P_I].$$ So for any tracial state $\tau$ on $\mathcal{A}$, by (\ref{pI1}) and (\ref{assump2}) we compute that
\begin{eqnarray*}
  \kappa_0([p])(\tau) &=& n_0 \tau_*([1_{\mathcal{A}}])+\sum\limits_{I}n_I\tau_*([R_{I}(u_1,u_2,\dots,u_n)]) \\
   &=& n_0 (\tau_{\Theta})_*([1_{\mathcal{A}_{\Theta}}])+\sum\limits_{I}n_I(\tau_{\Theta})_*([P_I])=(\tau_{\mathcal{B}})_*([p])>0.
\end{eqnarray*}
By Theorem 3.7.2 of \cite{Lin-book} $\mathcal{A}$ has strict comparison, which shows that $\kappa_0([p])$ is positive.

Define a map
$$\gamma:T(\mathcal{A})\rightarrow T(\mathcal{A}_{\Theta}),\,\, \gamma(\tau)=\tau_{\Theta},\,\, {\rm for}\,{\rm all}\, \tau\in T(\mathcal{A}).$$
It induces an affine map $\gamma_{\sharp}:{\rm Aff}(T(\mathcal{A}_{\Theta}))\rightarrow {\rm Aff}(T(\mathcal{A})).$ For any tracial state $\tau\in\mathcal{A},$ we can compute that
$$\gamma_{\sharp}(\hat{p})(\tau)=\hat{p}(\gamma(\tau))=\hat{p}(\tau_{\Theta})
=\tau_{\Theta}(p)=\kappa_0([p])(\tau).$$
It follows that $\gamma_{\sharp}(\overline{\rho_{\mathcal{A}_{\Theta}}(\K_0(\mathcal{A}_{\Theta}))})\subset \overline{\rho_{\mathcal{A}}(\K_0(\mathcal{A}))}.$ Therefore there is an induced homomorphism (denoted by the same symbol)
$$\gamma_{\sharp}:{\rm Aff}(T(\mathcal{A}_{\Theta}))\slash \overline{\rho_{\mathcal{A}_{\Theta}}(\K_0(\mathcal{A}_{\Theta}))}\rightarrow {\rm Aff}(T(\mathcal{A}))\slash \overline{\rho_{\mathcal{A}}(\K_0(\mathcal{A}))}.$$

Since $\K_1(\mathcal{A}_{\Theta})$ is free abelian, it is clear that there is a homomorphism
$\kappa_1:\K_1(\mathcal{A}_{\Theta})\rightarrow \K_1(\mathcal{A})$ such that the following diagram commutes:
$$\xymatrix{
U_{\infty}(\mathcal{A}_{\Theta})\slash CU_{\infty}(\mathcal{A}_{\Theta}) \ar[d]^{\alpha_1} &\ar[l]_{\phantom{aaaaaaa}J_{\mathcal{A}_{\Theta}}} \K_1(\mathcal{A}_{\Theta})
    \ar[d]^{\kappa_1} \\
 U_{\infty}(\mathcal{A})\slash CU_{\infty}(\mathcal{A}) &\ar[l]_{\phantom{aaaaaa}J_{\mathcal{A}}} \K_1(\mathcal{A}).}$$

 Use the split exact sequence (\ref{E_Thomsen_es}), we can find a homomorphism
 $\alpha:U_{\infty}(\mathcal{A}_{\Theta})\slash CU_{\infty}(\mathcal{A}_{\Theta})\rightarrow  U_{\infty}(\mathcal{A})\slash CU_{\infty}(\mathcal{A})$ so that the following
 diagram commutes:
$$\xymatrix{
  0  \ar[r]^{} & {\rm Aff}(T(\mathcal{A}_{\Theta}))\slash \overline{\rho_{\mathcal{A}_{\Theta}}(\K_0(\mathcal{A}_{\Theta}))} \ar[d]_{\gamma_{\sharp}} \ar[r]^{} & U_{\infty}(\mathcal{A}_{\Theta})\slash CU_{\infty}(\mathcal{A}_{\Theta}) \ar@{.>}[d]_{\alpha} {\phantom a}{\phantom a}{\phantom a}
  \overset{\pi_{\mathcal{A}_{\Theta}}}{\underset{J_{\mathcal{A}_{\Theta}}}\leftrightarrows} & \K_{1}(\mathcal{A}_{\Theta}) \ar[r]^{}\ar[d]_{\kappa_1} & 0  \phantom{,}\\
  0 \ar[r]^{} & {\rm Aff}(T(\mathcal{A}))\slash \overline{\rho_{\mathcal{A}}(\K_0(\mathcal{A}))} \ar[r]^{} & U_{\infty}(\mathcal{A})\slash CU_{\infty}(\mathcal{A}) \ar[r]^{\phantom{aaaaaa}\pi_{\mathcal{A}}} & \K_{1}(\mathcal{A}) \ar[r]^{} & 0. }$$
Let $\kappa=(\kappa_0,\kappa_1).$ Then $\kappa\in \KL_{e}(\mathcal{A}_{\Theta},\mathcal{A})^{++}.$ It follows from the commutative diagram
that $\kappa,\gamma$ and $\alpha$ are compatible. Therefore, by Theorem \ref{T_exist} there is a unital monomorphism $\phi:\mathcal{A}_{\Theta}\rightarrow \mathcal{A}$ such that $$\phi_*=\kappa,\,\, \phi_{\sharp}=\gamma\,\, {\rm and}\,\, \phi^{\ddag}=\alpha.$$

Now for any finite subset $\mathcal{G}\subset \mathcal{A}_{\Theta},$  and any $\varepsilon>0,$ choose $\eta>0,\delta_1$ , a finite subset $\widetilde{\mathcal{G}}\subset \mathcal{C},$ a finite subset $\widetilde{\mathcal{P}}\subset \underline{\K}(\mathcal{C}),$ a finite subset $\widetilde{\mathcal{H}}_0\subset \mathcal{C}_{s.a.}$ and a finite subset $\widetilde{\mathcal{U}}\subset
 U_{c}(\K_1(\mathcal{C}))$ with respect to $\Delta,$ $\varepsilon\slash 3$  and $\widetilde{\mathcal{F}}$  in Lemma \ref{L}. There exists $\delta>0$ such that for any unital $C^*$-algebra $\mathcal{A},$ any  $n$-tuple of unitaries $u_1,u_2,\dots,u_n$ in $\mathcal{A}$ satisfying
$$\|u_ku_j-e^{2\pi i\theta_{jk}}u_ju_k\|<\delta\,\,  for\,\, all\,\,j,k=1,2,\dots,n, $$
there is  a unital
$\mathcal{G}$-$\delta_1$-multiplicative c.p.c. map $L:\mathcal{A}_{\Theta}\rightarrow\mathcal{A}$ such that (1),(2),(3),(4) of Lemma \ref{L} hold.

We now compare the two maps $\widetilde{L}=L\circ \iota$ and $\widetilde{\phi}=\phi\circ \iota.$ It is easy to see from (1),(2),(3),(4) of Lemma \ref{L} and the construction of $\phi$ that
$$\widetilde{L}_*|_{\widetilde{\mathcal{P}}}=\widetilde{\phi}_*|_{\widetilde{\mathcal{P}}}$$
and
$${\rm dist}(\overline{\langle \widetilde{L}(s)\rangle}, \overline{\langle\widetilde{\phi}(s)}\rangle)={\rm dist}(\overline{\langle L(\iota(s))\rangle},\overline{\phi(\iota(s))})<\delta_1\,\,{\rm for}\,{\rm all}\,\, s\in \widetilde{\mathcal{\mathcal{U}}}.$$

Since $\tau_{\Theta}$ is the only tracial state on $\mathcal{A}_{\Theta}$,  we have
$$\|\tau\circ\widetilde{L}(g)-\tau\circ \widetilde{\phi}(g)\|=\|\tau\circ L(\iota(g))-\tau_{\Theta}(\iota(g))\|< \delta_1$$
for all $\tau\in T(\mathcal{A})$ and $g\in \mathcal{H}_0.$

Note that
$$\|\tau\circ \widetilde{\phi}(h)-\tau_{\Theta}(\iota(h))\|=0<\delta_2 \,\, {\rm for}\, {\rm all}\,\, h\in \mathcal{C}.$$
So by Lemma \ref{L_Delta}, we have
$$ \mu_{\tau\circ \widetilde{\phi}}(O_r)>\Delta (r)\,\, {\rm for}\,{\rm all}\,\,r\geq \eta.$$
Note that $\mu_{\tau\circ \widetilde{L}}(O_r)>\Delta (r)$ for all $r\geq \eta$ by (4) of Lemma \ref{L}.

Now by Theorem \ref{T_unique}, there is a unitary $s\in \mathcal{A}$ such that
$$\|s^*\widetilde{\phi}(a)s-\widetilde{L}(a)\|<\frac{\varepsilon}{3}\,\,\, {\rm for}\,{\rm all}\,\, a\in \widetilde{\mathcal{F}}.$$
So we can find $v_j\in \widetilde{\mathcal{F}}$ for $j=1,2,\dots,n$ such that
$$\|\iota(v_j)-\mathfrak{u}_j\|<\frac{\varepsilon}{6}\,\,{\rm for}\,\, j=1,2,\dots,n.$$
We can then compute that
\begin{eqnarray*}
  \|s^* \phi(\mathfrak{u}_j)s-u_j\| &\leq& \|s^* \phi(\iota(v_j))s-L(\mathfrak{u}_j)\|+\frac{\varepsilon}{3} \\
   &\leq & \|s^*\widetilde{\phi}(v_j)s-L(\iota(v_j))\|+\frac{\varepsilon}{3}+\frac{\varepsilon}{3} \\
   &\leq& \|s^*\widetilde{\phi}(v_j)s-\widetilde{L}(v_j)\|+\frac{2\varepsilon}{3}<\varepsilon,
\end{eqnarray*}
for $j=1,2,\dots,n.$

Let $\tilde{u}_j=s^*\phi(\mathfrak{u}_j)s$ for $j=1,2,\dots,n.$ Then
$$\tilde{u}_k\tilde{u}_j=e^{2\pi i\theta_{jk}}\tilde{u}_j\tilde{u}_k\quad {\rm and}\quad \|\tilde{u}_j-u_j\|<\varepsilon,\,\, j,k=1,2,\dots,n.$$
\end{proof}

Next let us see some low-dimensional examples, i.e. the cases of $n=2,3,4$.

\

{\bf (\romannumeral1) Case $n=2.$}
\begin{corollary}Let $\theta$ be an irrational number in $(0,1).$ Then for any $\varepsilon>0,$ there exists $\delta>0$ satisfying the following: For any unital simple separable $C^*$-algebra $\mathcal{A}$ with tracial rank at most one, any pair of unitaries $u_1,u_2$ in $\mathcal{A}$ with
$$\|u_2u_1-e^{2\pi i\theta}u_1u_2\|<\delta,$$and
$$\tau(\log_{\theta}(u_2u_1u_2^*u_1^*))=2\pi i\theta\,\, {\rm for}\,\,{\rm all}\,\,\tau\in T(\mathcal{A}),$$
there exists a pair of unitaries $\tilde{u}_1,\tilde{u}_2$ in $\mathcal{A}$ such that
$$\tilde{u}_2\tilde{u}_1=e^{2\pi i\theta}\tilde{u}_1\tilde{u}_2\,\, {\rm and}\, \, \|\tilde{u}_1-u_1\|<\varepsilon,\,\|\tilde{u}_2-u_2\|<\varepsilon.$$

\end{corollary}
\begin{proof}Since $n=2,$ the only possible value of $I$ is $(1,2)$ in Theorem \ref{main result}.
By the generalized Exel trace formula (\ref{Exel formula}), $\tau(R_{I}(u_1,u_2))=\frac{1}{2\pi i}\tau(\log_{\theta}(u_2 u_1 u_2^* u_1^*))=\theta$ for all $\tau\in T(\mathcal{A}).$ So $u_1$ and $u_2$ satisfy the conditions of Theorem \ref{main result}, we get the conclusion.
\end{proof}

This case is a generalization of any unital simple separable $C^*$-algebra with tracial rank zero in Theorem \ref{HH} to any unital simple separable $C^*$-algebra with tracial rank at most one when $\theta$ is an irrational number.

\

{\bf (\romannumeral2) Case $n=3.$}
\begin{corollary}Let $\Theta=(\theta_{jk})_{3\times 3}$ be a totally irrational real skew-symmetric matrix in $\mathcal{T}_3$ and $\theta_{jk}\in (0,1)$ for $j,k=1,2,3.$ In this case total irrationality means that the three numbers $\theta_{12}, \theta_{13}, \theta_{23}$ are independent over $\Q.$
Then for any $\varepsilon>0,$ there exists $\delta>0$ satisfying the following: For any unital simple separable $C^*$-algebra $\mathcal{A}$ with tracial rank at most one, any triple of unitaries $u_1,u_2,u_3$ in $\mathcal{A}$ with
$$\|u_ku_j-e^{2\pi i\theta_{jk}}u_ju_k\|<\delta\,\, {\rm for}\,\, j,k=1,2,3,$$
and
$$\tau(\log_{\theta_{jk}}(u_ku_ju_k^*u_j^*))=2\pi i\theta_{jk}\,\, {\rm for}\,\,{\rm all}\,\,\tau\in T(\mathcal{A})\,\,{\rm and}\,\, j,k=1,2,3,$$
there exists a triple of unitaries $\tilde{u}_1,\tilde{u}_2,\tilde{u}_3$ in $\mathcal{A}$ such that$$\tilde{u}_k\tilde{u}_j=e^{2\pi i\theta_{jk}}\tilde{u}_j\tilde{u}_k\,\, {\rm and}\,\, \|\tilde{u}_j-u_j\|<\varepsilon\,\, {\rm for}\,\, j,k=1,2,3.$$
\end{corollary}
\begin{proof}Since $n=3,$ the possible values of $I$ are $(1,2),(1,3)$ and $(2,3)$ in Theorem \ref{main result}.
By the generalized Exel trace formula (\ref{Exel formula}), $\tau(R_{(j,k)}(u_1,u_2,u_3))=\frac{1}{2\pi i}\tau(\log_{\theta_{jk}}(u_k u_j u_k^* u_j^*))=\theta_{jk}$ for $1\leq j<k\leq 3$ and all $\tau\in T(\mathcal{A}).$ So $u_1,u_2$ and $u_3$ satisfy the conditions of Theorem \ref{main result}, we get the conclusion.
\end{proof}

Of course, this case is included in Theorem \ref{nondegenerate stability}.

\

{\bf (\romannumeral3) Case $n=4.$}
\begin{corollary}Let $\Theta=(\theta_{jk})_{4\times 4}$ be a strongly totally irrational real skew-symmetric matrix in $\mathcal{T}_4,$ where $\theta_{jk}\in [0,1)$ for $j,k=1,2,3,4.$
Then for any $\varepsilon>0,$ there exists $\delta>0$ satisfying the following: For any unital simple separable $C^*$-algebra $\mathcal{A}$ with tracial rank at most one, any $4$-tuple of unitaries $u_1,u_2,u_3,u_4$ in $\mathcal{A}$ such that
$$\|u_ku_j-e^{2\pi i\theta_{jk}}u_ju_k\|<\delta,\, {\rm for}\, j,k=1,2,3,4,$$
there exists a $4$-tuple of unitaries $\tilde{u}_1,\tilde{u}_2,\tilde{u}_3,\tilde{u}_4$ in $\mathcal{A}$ such that
$$\tilde{u}_k\tilde{u}_j=e^{2\pi i\theta_{jk}}\tilde{u}_j\tilde{u}_k\,\, {\rm and}\,\, \|\tilde{u}_j-u_j\|<\varepsilon,\,\, {\rm for}\,\, j,k=1,2,3,4$$
if and only if
\begin{eqnarray}\tau(\log_{\theta_{jk}}(u_ku_ju_k^*u_j^*))=2\pi i\theta_{jk}\,\, {\rm for}\,\,{\rm all}\,\,\tau\in T(\mathcal{A})\,\,{\rm and}\,\, j,k=1,2,3,4,\label{c1}
\end{eqnarray}
and
\begin{eqnarray}\label{P4}
\tau(R_{(1,2,3,4)}(u_1,u_2,u_3,u_4))=\theta_{12}\theta_{34}-\theta_{13}\theta_{24}+\theta_{14}\theta_{23}+k\theta_{12}\,\, {\rm for}\,{\rm all}\,\tau\in T(\mathcal{A}),
\end{eqnarray}
where $R_{(1,2,3,4)}(u_1,u_2,u_3,u_4)$ is defined as in Definition \ref{RI}, $k$ is an integer as in (\ref{pI1}).
\end{corollary}
\begin{proof}Since $n=4,$ the possible values of $I$ are $\{(j,k)|1\leq j<k\leq4\}\cup \{(1,2,3,4)\}$ in Theorem \ref{main result}.
By the generalized Exel trace formula (\ref{Exel formula}) and the assumption (\ref{c1}), $\tau(R_{(j,k)})=\frac{1}{2\pi i}\tau(\log_{\theta_{jk}}(u_k u_j u_k^* u_j^*))=\theta_{jk}$ for $1\leq j<k\leq 4,$ and by the assumption (\ref{P4}) $\tau(R_{(1,2,3,4)}(u_1,u_2,u_3,u_4))=\theta_{12}\theta_{34}-\theta_{13}\theta_{24}+\theta_{14}\theta_{23}+k\theta_{12}\,\,$ for all $\tau\in T(\mathcal{A})$. So $u_1,u_2,u_3,u_4$ satisfy the conditions of Theorem \ref{main result}, we get the conclusion.
\end{proof}

Finally, we show that  the trace conditions (\ref{t11})
are necessary to prove the stability of rotation relations in Theorem \ref{main result}.

\begin{proposition}
  Let $\Theta=(\theta_{jk})\in \mathcal{T}_n$ be strongly totally irrational.
  There exists a $\delta>0$ such that: For any unital $C^*$-algebra $\mathcal{A}$
  with $T(\mathcal{A})\neq \emptyset$, any unitaries $u_1,u_2,\dots,u_n$ and $\tilde{u}_1,\tilde{u}_2,\dots,\tilde{u}_n$
  in $\mathcal{A}$ satisfying the following:\\
  (1)\,\,$\|u_ku_j-e^{2\pi i\theta_{jk}}u_ju_k\|<2\,\,{\rm for}\,\,j,k=1,2,\dots,n,$\\
  (2)\,\,$\tilde{u}_k\tilde{u}_j=e^{2\pi i\theta_{jk}}\tilde{u}_j\tilde{u}_k\,\,{\rm for}\,\,j,k=1,2,\dots,n,$\\
  (3)\,\,$\|\tilde{u}_j-u_j\|<\delta\,\, {\rm for}\,\,j=1,2,\dots,n,$\\
  we have \begin{eqnarray}\label{t111}\tau(R_{I}(u_1,u_2,\dots,u_n))={\rm pf}(M_{I}^{\Theta})+\sum\limits_{0<|J|< l}{\rm pf}(M_{J}^{\Theta})k_{I_J}+k_{I_0}\nonumber
  \end{eqnarray}
  for all possible  $I=(i_1,i_2,\dots,i_{l})$, $1\leq i_1<i_2<\cdots<i_l\leq n,$ $l=2m,$ $m\in\{1,2,\dots,\lfloor\frac{n}{2}\rfloor\},$ and all tracial states $\tau$ on $\mathcal{A}$, where  $R_{I}(u_1,u_2,\dots,u_n)$ is defined as in Definition \ref{RI},
  $k_{I_J},k_{I_0}\in \mathbb{Z}$  is the same as  in Corollary \ref{PI},
  $J=(i_1,i_2,\dots,i_{s})$,  $|J|=s.$
 \end{proposition}
 \begin{proof}Choose $\delta_0$ according to Proposition \ref{gap}. Since
 $R_{I}(u_1,u_2,\dots,u_n)$ is defined in terms of continuous functional calculus of $u_1,u_2,\dots,u_n$, there is a $\delta_1$ such that whenever
 $\|\tilde{u}_j-u_j\|<\delta_1$ for $j=1,2,\dots,n,$ we have
 $\|R_{I}(u_1,u_2,\dots,u_n)-R_{I}(\tilde{u}_1,\tilde{u}_2,\dots,\tilde{u}_n)\|<\frac{1}{2}$.
 Let $\delta=\min\{\delta_0,\delta_1\}.$  Then for any unitaries $u_1,u_2,\dots,u_n$ and $\tilde{u}_1,\tilde{u}_2,\dots,\tilde{u}_n$
  in $\mathcal{A}$ satisfies conditions (1), (2), (3), we have \begin{eqnarray}\label{w1}[R_{I}(u_1,u_2,\dots,u_n)]=[R_{I}(\tilde{u}_1,\tilde{u}_2,\dots,\tilde{u}_n)]
  \end{eqnarray}
  in $\K_0(\mathcal{A}).$
  Since $\tilde{u}_k\tilde{u}_j=e^{2\pi i\theta_{jk}}\tilde{u}_j\tilde{u}_k$ for $j,k=1,2,\dots,n,$ we have
  \begin{eqnarray}\label{w2}R_{I}(\tilde{u}_1,\tilde{u}_2,\dots,\tilde{u}_n)=
  e_{I}(\tilde{u}_1,\tilde{u}_2,\dots,\tilde{u}_n).
   \end{eqnarray}
   By Remark \ref{remark1}, there exists a homomorphism $\varphi:\mathcal{A}_{\Theta}\rightarrow \mathcal{A}$ such that $\varphi(\mathfrak{u}_j)=\tilde{u}_j$ for $j=1,2,\dots,n.$ Thus $\varphi(P_I)= e_{I}(\tilde{u}_1,\tilde{u}_2,\dots,\tilde{u}_n).$ From (\ref{w1}), (\ref{w2}) and (\ref{pI1}),
   we have
  \begin{eqnarray}
  \tau_*([R_{I}(u_1,u_2,\dots,u_n)])&=&\tau_*([R_{I}(\tilde{u}_1,\tilde{u}_2,\dots,\tilde{u}_n)])\nonumber\\
  &=&\tau_*([\chi_{(\frac{1}{2},+\infty)} (e_{I}(\tilde{u}_1,\tilde{u}_2,\dots,\tilde{u}_n))])\nonumber\\
  &\overset{(\ref{b1})}{=}&\tau_*([\varphi(\chi_{(\frac{1}{2},+\infty)}(\widetilde{P}_I))])\nonumber\\
  &=&\tau_*([\varphi(P_I)])\nonumber\\
  &=&(\tau_{\Theta})_*([P_I])\nonumber\\
  &=&{\rm pf}(M_{I}^{\Theta})+\sum\limits_{0<|J|< l}{\rm pf}(M_{J}^{\Theta})k_{I_J}+k_{I_0}\nonumber
  \end{eqnarray}
  for all possible $I$ and all tracial states $\tau$ on $\mathcal{A}$, where  $R_{I}(u_1,u_2,\dots,u_n)$ is as in Definition \ref{RI}, $P_{I}$ is as in Theorem \ref{PI},   $\widetilde{P}_{I}$ is as in (\ref{Nk}), $I=(i_1,i_2,\dots,i_{l})$, $ 1\leq i_1<i_2<\cdots<i_l\leq n,$ $l=2m,$ $m\in\{1,2,\dots,\lfloor\frac{n}{2}\rfloor\},$ $k_{I_J},k_{I_0}\in \mathbb{Z}$  is the same as  in Corollay \ref{PI},
  $J=(i_1,i_2,\dots,i_{s})$,  $|J|=s.$
 \end{proof}

\section*{Appendix  \uppercase\expandafter{\romannumeral1}}
In this section we will provide a large class of examples of $n\times n$ skew-symmetric matrices, generalising  Proposition \ref{h3}, which satisfy the conditions of Theorem \ref{PI1}.

Let $s=\{s_i\}_i$ be a sequence of integers such that $s_i>\sum\limits_{j=1}^{i-1}s_j$ with $s_1=1.$ We call such a sequence a super-increasing sequence. For $\alpha \in (0,1)$ define the $n \times n$ antisymmetric matrix $\Theta(n)=(\Theta(n)_{ij})$ by induction: $\Theta(2)=\left(\begin{array}{cc}0 & \alpha^{s_1} \\ -\alpha^{s_1} & 0\end{array}\right);$ $\Theta(n)_{ij} = \Theta(n-1)_{ij},$ for $1<i<j<n,$ and $\Theta(n)_{in}:= \alpha^{s_{p+i}},$ for $i=1,\cdots, n-1,$ where $p=\frac{(n-1)(n-2)}{2}.$

   Hence

     $\Theta(3)=\left(\begin{array}{ccc}0 & \alpha^{s_1} & \alpha^{s_2}\\ -\alpha^{s_1} & 0 & \alpha^{s_3} \\ -\alpha^{s_2} & -\alpha^{s_3} & 0 \end{array}\right)$,
     $\Theta(4)=\left(\begin{array}{cccc}0 & \alpha^{s_1} & \alpha^{s_2} & \alpha^{s_4}\\ -\alpha^{s_1} & 0 & \alpha^{s_3} & \alpha^{s_5} \\ -\alpha^{s_2} & -\alpha^{s_3} & 0 & \alpha^{s_6} \\ -\alpha^{s_4} & -\alpha^{s_5}  & -\alpha^{s_6}& 0 \end{array}\right).$

 \begin{remark}\label{rem:submatrices_super}
 	
 Note that any sub-matrix $M_I^{\Theta(n)}$  of $\Theta(n)$ is like $\Theta(m),$ for some $m \le n,$ but for a different super increasing sequence $s'=\{s'_i\}_i,$ which is a subsequence of $s.$
   \end{remark}
 We then have $\pf(\Theta(2))=\alpha^{s_1},$ and

$$\pf(\Theta(4))=\alpha^{s_1+s_6}-\alpha^{s_2+s_5}+\alpha^{s_4+s_3}.$$
Note that, using the super-increasing property of $s,$ we have that

$$s_1+s_6>s_2+s_5>s_4+s_3.$$
Let us first recall the following recursive definition of pfaffian from \cite[Page 116]{Fulton-Pragacz}. Let $\mathrm{pf}^{i j}(A)$ denote the pfaffian of the skew-symmetric matrix obtained from $A=(a_{mn})$ by removing the $i^{t h}$, $j^{t h}$ row and the $i^{t h}$, $j^{t h}$ column. Let $n$ be even. Then for a fixed integer $j, 1 \leq j \leq n$, one has

$$
\operatorname{pf}(A)=\sum_{i<j}(-1)^{i+j-1} a_{i j} \operatorname{pf}^{i j}(A)+\sum_{i>j}(-1)^{i+j} a_{i j} \operatorname{pf}^{i j}(A).
$$

In particular, when $j=n,$ we have

\begin{equation}\label{eq:pfaffian_recursive}
\operatorname{pf}(A)=\sum_{i=1}^{n-1}(-1)^{i+1} a_{i n} \operatorname{pf}^{i n}(A).	
\end{equation}

\begin{lemma}\label{lem:pfaffian_in_0_1}

	$$0<\mathrm{pf}(\Theta(n))<\mathrm{pf}(\Theta(n-2))<1,$$ for an even $n.$ \end{lemma}
	\noindent We will give the proof after Remark~\ref{rem:appen2}.

	\begin{corollary}\label{cor:submatrices_pfaffian_in_0_1}
	With the notations of Theorem \ref{PI1},
	$$0<\operatorname{pf}\left(F^{j}\left(M_{I}^{\Theta(n)}\right)_{11}\right)<1,$$  for all $I$  with  $2 \leq|I| \leq 2 \lfloor \frac{n}{2}\rfloor ,$ and for all  $j=0,1, \ldots, m-1,$ where $|I|=2 m.$ 	\end{corollary}
	
	\begin{proof}
		Follows from the above lemma and Remark~\ref{rem:submatrices_super} with the explicit expression of \\ $\operatorname{pf}\left(F^{j}\left(M_{I}^{\Theta(n)}\right)_{11}\right)$ (using Lemma \ref{h2}) in hand.
	\end{proof}
	
	Choose the super-increasing sequence $\{s_i=2^{i-1}\}_i.$ When $\alpha$ is a transcendental number,  it is well known that the numbers
$\alpha,\alpha^2,\dots,\alpha^{2^i},\dots$  as well as any
product the numbers are linear independent over $\mathbb{Q}$. So we have that $\Theta(n)$ is totally
irrational.
Then by using the above corollary, we get the following corollary.
\begin{corollary}\label{a}
  Let $\alpha\in (0,1)$  be a transcendental number and let
  $s:=\{s_i=2^{i-1}\}_i.$ Let $\Theta(n)$ be the $n\times n$
  antisymmetric matrix involving $\alpha$ and $s.$ Then we have that
  $\Theta(n)$ is a strongly totally irrational matrix for $n\geq2.$
\end{corollary}

It is clear that the above corollary is  a generalisation of Proposition \ref{h3}.
 This corollary gives general examples of strongly totally irrational matrices.

	The above lemma (Lemma \ref{lem:pfaffian_in_0_1}) will follow from the next lemma.
	
	\begin{lemma}\label{lem:pfaffian_incresing} For an even $n$ we have
	$$\mathrm{pf}(\Theta(n))=\alpha^{M(n)_1}-\alpha^{M(n)_2}+\alpha^{M(n)_3}-\alpha^{M(n)_4}+\cdots +\alpha^{M(n)_{R(n)}},$$ for a strictly decreasing sequence of numbers  $M(n)_1,M(n)_2,\cdots, M(n)_{R(n)},$ where $R(n)=(n-1)!!$\footnote{double factorial}.
	\end{lemma}

	\begin{proof}
		We prove this by induction on $n.$ We have seen that the statement is true for $n=2,4$ for all super-increasing sequences. Now assume the statement is true for $n-2,$ for all super-increasing sequences. Then we must prove that the statement is true for $n,$ for all super-increasing sequences. Fix a super-increasing sequence $s=\{s_i\}_i.$ Write $\Theta=\Theta(n).$ From (\ref{eq:pfaffian_recursive}), we have
		
		\begin{align*}
			\operatorname{pf}(\Theta)&=\sum_{i=1}^{n-1}(-1)^{i+1} \Theta_{i n} \operatorname{pf}^{i n}(\Theta)\\
			&= \Theta_{(n-1) n} \operatorname{pf}^{(n-1) n}(\Theta)-\Theta_{(n-2) n} \operatorname{pf}^{(n-2) n}(\Theta)+\Theta_{(n-3) n} \operatorname{pf}^{(n-3) n}(\Theta)-\cdots\\&\quad -\Theta_{2 n} \operatorname{pf}^{2 n}(\Theta)+\Theta_{1 n} \operatorname{pf}^{1 n}(\Theta)\\
			&= \alpha^{s_{(n-1), n}} \operatorname{pf}^{(n-1) n}(\Theta)-\alpha^{s_{(n-2), n}} \operatorname{pf}^{(n-2) n}(\Theta)+\alpha^{s_{(n-3), n}} \operatorname{pf}^{(n-3) n}(\Theta)-\cdots\\&\textcolor{red}\quad-\alpha^{s_{2, n}} \operatorname{pf}^{2 n}(\Theta)+\alpha^{s_{1, n}} \operatorname{pf}^{1 n}(\Theta),
		\end{align*}
		where $s_{i,j}$ is the exponent of $\alpha$ in $\Theta_{ij}.$ Now by the induction hypothesis we have that all $\operatorname{pf}^{i n}(\Theta), i=1,2, \ldots ,n-1,$ are of the form as in the statement with $R(n-2)=(n-3)!!.$ Now if we expand all $\operatorname{pf}^{i n}(\Theta), i=1,2, \ldots,n-1,$ keeping the form, we see that the expression $$\alpha^{s_{(n-1), n}} \operatorname{pf}^{(n-1) n}(\Theta)-\alpha^{s_{(n-2), n}} \operatorname{pf}^{(n-2) n}(\Theta)+\alpha^{s_{(n-3), n}} \operatorname{pf}^{(n-3) n}(\Theta)-\cdots\\-\alpha^{s_{2, n}} \operatorname{pf}^{2 n}(\Theta)+\alpha^{s_{1, n}} \operatorname{pf}^{1 n}(\Theta)$$ is already of the required form, using the super-increasing property of $s,$ and noting that the  exponents of $\alpha$ in the expressions of $\operatorname{pf}^{i n}(\Theta), i=1,2, \ldots ,n-1,$ contain no term involving $s_{*,n}.$
		Note that the total number of terms (after expanding) in the above expression is $(n-1)(n-3)!!=(n-1)!!.$
	\end{proof}

\begin{remark}\label{rem:appen2}
	It is also clear from above, again using the super-increasing property of $s,$ that $M(n)_{R(n)}>M(n-2)_1.$
\end{remark}

  \begin{proof}[Proof of Lemma~\ref{lem:pfaffian_in_0_1}]
  Since $$\mathrm{pf}(\Theta(n))=\alpha^{M(n)_1}-\alpha^{M(n)_2}+\alpha^{M(n)_3}-\alpha^{M(n)_4}+\cdots +\alpha^{M(n)_{R(n)}},$$ we have that

  $$\mathrm{pf}(\Theta(n))>\alpha^{M(n)_1}>0,$$ using $-\alpha^{M(n)_{2i}}+\alpha^{M(n)_{2i+1}}>0$ (since $\alpha \in (0,1)$). To show $\mathrm{pf}(\Theta(n))<\mathrm{pf}(\Theta(n-2)),$ we look at the expression $\mathrm{pf}(\Theta(n))-\mathrm{pf}(\Theta(n-2)),$ which is

  $$\alpha^{M(n)_1}-\alpha^{M(n)_2}+\alpha^{M(n)_3}-\cdots +\alpha^{M(n)_{R(n)}}-\alpha^{M(n-2)_1}+\alpha^{M(n-2)_2}-\alpha^{M(n-2)_3}+\cdots -\alpha^{M(n-2)_{R(n-2)}}.$$
  Note that $M(n)_1,M(n)_2,\cdots, M(n)_{R(n)},M(n-2)_1,M(n-2)_2,\cdots, M(n-2)_{R(n-2)}$ is still a strictly decreasing finite sequence due to the above remark. Now using $\alpha^{M(n)_i}-\alpha^{M(n)_{i+1}}<0,$ $\alpha^{M(n)_{R(n)}}-\alpha^{M(n-2)_{1}}<0,$ and $\alpha^{M(n-2)_{2i}}-\alpha^{M(n-2)_{2i+1}}<0,$ we get the above expression less than zero.

  The last inequality in the statement of Lemma~\ref{lem:pfaffian_in_0_1} follows from an argument using induction on $n.$
   \end{proof}

\section*{Appendix  \uppercase\expandafter{\romannumeral2}}

In this section we will give a description of the map $\psi$ (for a general $n$) used in Theorem~\ref{even n projection} of Section~\ref{sec:rie}, and an explicit expression of the Rieffel-type projection for $n=4,$ which was constructed in Theorem~\ref{even n projection}. This section essentially resembles the construction of \cite[Remark of page 198]{Boca}, however no explicit expression of the Rieffel-type projection (for $n=4$) is given in \cite{Boca}.

Let us first write down the Morita equivalence construction of Section~\ref{sec:rie} explicitly for $p=1,$ i.e. when $\mathcal{M}$ is the group $\R\times \Z^q,$ $q=n-2.$
As before, write $$\Theta=\left(\begin{matrix}
                 \Theta_{11} &  \Theta_{12} \\
               \Theta_{21}& \Theta_{22}
                 \end{matrix}\right)= \left(\begin{matrix}
               \Theta_{11} & \Theta_{12} \\
                -\Theta_{12}^t &\Theta_{22}
                 \end{matrix}\right)\in\mathcal{T}_n,$$
                 where
                 \begin{eqnarray}
  \Theta_{11}=\left(\begin{matrix}
              0 &  \theta_{12} \\
               \theta_{21}& 0
                 \end{matrix}\right)=\left(\begin{matrix}
                 0 &  \theta_{12} \\
              - \theta_{12}& 0
                 \end{matrix}\right)\in\mathcal{T}_2\nonumber
\end{eqnarray} is an invertible $2\times2$ matrix. Then consider the matrix
$T=\left(\begin{matrix}
                  T_{11} & 0 \\
                 0 & {\rm id}_q\\
                 T_{31}&T_{32}
                \end{matrix}
\right),$
where $T_{11}=\left(\begin{matrix}
                  \theta_{12} & 0 \\
                 0 & 1
                 \end{matrix}
\right),$
$T_{31}=\Theta_{21}$ and $T_{32}=\frac{\Theta_{22}}{2}.$
Also consider $S=\left(\begin{matrix}
                  S_0 &  -S_0 T_{31}^t \\
                 0& {\rm id}_q\\
                 0& T_{32}^t
                 \end{matrix}
\right),$
where
$S_0=\left(\begin{matrix}
                  0 & 1 \\
                 \theta_{12}^{-1} & 0
                 \end{matrix}
\right).$
Then \cref{eq:module1,eq:module2,eq:module3,eq:module4} make $\mathcal{S}(\R \times \Z^q)$  an $\mathcal{A}^\infty_{\Theta'}-\mathcal{A}^\infty_{\Theta}$
Morita equivalence bimodule, where
$\Theta'=\left(\begin{matrix}
                  \Theta_{11}^{-1} & -\Theta_{11}^{-1}\Theta_{12} \\
                  \Theta_{21}\Theta_{11}^{-1} & \Theta_{22}-\Theta_{21}\Theta_{11}^{-1}\Theta_{12}
                \end{matrix}
\right).$ Let us denote by $\mathfrak{u_1}, \mathfrak{u_2}, \ldots, \mathfrak{u_n}$ and $\mathfrak{v_1}, \mathfrak{v_2}, \ldots,\mathfrak{v_n}$ the canonical generators of $\mathcal{A}_{\Theta}$ and $\mathcal{A}_{\Theta'},$ respectively.  For the following, we consider the conjugate space of $\mathcal{S}(\R \times \Z^q)$ as an $\mathcal{A}^\infty_{\Theta}-\mathcal{A}^\infty_{\Theta'}$ bimodule.

Following Boca (\cite[page 190]{Boca}), if we can find a $f\in \mathcal{S}(\R \times \Z^q)$ such that
$\langle f,f\rangle_{\mathcal{A}^\infty_{\Theta'}}=1,$ then $e=_{\mathcal{A}^\infty_{\Theta}}\langle f,f\rangle$ is a projection of trace $\theta_{12}$ in $\mathcal{A}^\infty_{\Theta}\subset \mathcal{A}_{\Theta},$ and we have an isomorphism $\psi: \mathcal{A}_{\Theta'} \rightarrow e\mathcal{A}_{\Theta}e,$ given by $\psi(a)= _{\mathcal{A}_{\Theta}}\langle fa,f\rangle.$ Choose a function $\phi$ on $\R$ as in the definition of $f_2$ in Definition~$\ref{def of Rieffel P}$ for $\theta_{12},$ assuming $\theta_{12}\in (0,1).$ By using a standard regularization argument as in \cite[Lemma 2.1]{Boca}, we can assume that $\phi$ is smooth and $\sqrt{\phi}$ is also smooth. Define a function $f\in \mathcal{S}(\R \times \Z^q),$ given by $f(x,l)=c\sqrt{\phi}(x),$ $c=\frac{1}{\sqrt{K\theta_{12}}}$ ($K$ is as in Equation~(\ref{eq:module4})), when $\Z^q\ni l=0$ and $f(x,l)=0$ otherwise. Let us first show that $\langle f,f\rangle_{\mathcal{A}^\infty_{\Theta'}}=1.$ To this end, from Equation~(\ref{eq:module4}), we have the following:

$$\langle f,f\rangle_{\mathcal{A}^\infty_{\Theta'}}(m)=Ke^{-2\pi i\langle S(m),J'S(m)\slash2\rangle}\int_{\R \times \Z^q}\overline{\langle x,S''(m)\rangle}f(x+S'(m))f(x)dx,$$ for $m=(m_1,m_2,\ldots,m_n) \in \Z^n,$ noting that $f$ is real valued. Using the formula of $S'$ and $f,$ it is clear that the above expression is zero for any non-zero values of $(m_3,m_4,\ldots,m_n)\in \Z^q.$ Also when $m_2\neq 0,$ the above expression for $m=(m_1,m_2,0,\ldots,0)\in \Z^n$ is zero since $x+S'(m)= (x_0+m_2,x_1,\ldots,x_q),$ where $x=(x_0,x_1,\ldots,x_q),$ and $\phi$ is supported in the unit interval. Finally for an $m= (m_1,0,\ldots,0)$ the expression becomes
\begin{eqnarray*}
&&Kc^2\int_{\R}e^{\frac{-2\pi ix_0 m_1} {\theta_{12}}}\phi(x_0)dx_0\\
&=&c^2K\theta_{12}\int_{\R}e^{-2\pi ix m_1}\phi(x\theta_{12})dx\\
&=&\int_{\R}e^{-2\pi ix m_1}\widetilde{\phi}(x)dx,
\end{eqnarray*}
where $\widetilde{\phi}(x)=\phi(x\theta_{12}).$ Then the expression is

$$\int_{\R}e^{-2\pi ix m_1}\widetilde{\phi}(x)dx=\widehat{\tilde{\phi}}(m_1)=\widehat{\Phi}(m_1),$$
where $~\widehat{}~$ denotes the Fourier transform and $\Phi$ is the periodic function defined by $\Phi(x)= \sum\limits_{n\in \Z} \widetilde{\phi}(x+n),$ $x\in \R.$ But $\sum\limits_{n\in \Z} \widetilde{\phi}(x+n)=\sum\limits_{n\in \Z} \phi(\theta_{12}x+\theta_{12}n)=1$ for all $x\in \R.$ Hence we get $\langle f,f\rangle_{\mathcal{A}^\infty_{\Theta'}}(m)=\Phi(\mathfrak{v}_1)=1.$
Now let us write an explicit expression for $e=_{\mathcal{A}^\infty_{\Theta}}\langle f,f\rangle.$ From Equation~(\ref{eq:module2}), we have $$_{\mathcal{A}^\infty_{\Theta}}\langle f,f\rangle(m)=e^{2\pi i\langle T(m),J'T(m)\slash2\rangle}\int_{\R \times \Z^q}\langle x,T''(m)\rangle f(x+T'(m))f(x)dx.$$ If we write $m=(m_1,m_2,m_3,\ldots,m_n),$ an easy observation shows that the above expression is zero unless $m_3=m_4=\cdots=m_n=0.$ For $m= (m_1,m_2,0,\ldots,0)$ we get

\begin{eqnarray*}
	_{\mathcal{A}^\infty_{\Theta}}\langle f,f\rangle(m)&=&c^2e^{\pi i\theta_{12}m_1m_2}\int_{\R}e^{2\pi ix_0m_2} \sqrt{\phi}(x_0+\theta_{12}m_1)\sqrt{\phi}(x_0)dx_0\\
&=&c^2\int_{\R}e^{2\pi ix_0m_2} \sqrt{\phi}(x_0-\frac{\theta_{12}m_1}{2})\sqrt{\phi}(x_0+\frac{\theta_{12}m_1}{2})dx_0.
\end{eqnarray*}
The above expression is non-zero only when $m_1=0,\pm 1,$ and the expression of $e=_{\mathcal{A}^\infty_{\Theta}}\langle f,f\rangle$ becomes $G_{-}(\mathfrak{u}_2^{-1})\mathfrak{u}_1^{-1}+G(\mathfrak{u}_2^{-1})+G_{+}(\mathfrak{u}_2^{-1})\mathfrak{u}_1,$
which is a projection of trace $\theta_{12}$ (this gives $c=1$ and $K=1/\theta_{12}$), where $G$ and $G_{\pm}$ are the following: $G(x)= \sum\limits_{n\in \Z} \phi(x+n),$ $x\in \R,$ $G_{\pm}(x)= \sum\limits_{n\in \Z} \sqrt{\phi(x+n+\frac{\theta_{12}}{2})\phi(x+n-\frac{\theta_{12}}{2}}),$ $x\in \R.$

Next we want specialize to the case $n=4,$ where we will to compute the Rieffel-type projection explicitly. Although the computations are messy, they are quite straight forward. As discussed in Section~\ref{sec:rie}, the image $e_1$ of a Rieffel projection $e'\in C^*(\mathfrak{v}_3,\mathfrak{v}_4)\subset \mathcal{A}_{\Theta'}$ of trace $\theta'=\frac{\mathrm{pf}(\Theta)}{\theta_{12}}+k$ by the isomorphism $ \mathcal{A}_{\Theta'}\simeq e\mathcal{A}_{\Theta}e $ is a projection in $e\mathcal{A}_{\Theta}e$ of trace $\mathrm{pf}(\Theta)+ k\theta_{12}$
in $\mathcal{A}_{\Theta},$ for some $k\in \mathbb{Z}.$ Since the isomorphism between $\mathcal{A}_{\Theta'}$ and $e\mathcal{A}_{\Theta}e$ is given by $\psi(x)=_{\mathcal{A}_{\Theta}}\langle f x,f\rangle,$ $ x\in \mathcal{A}_{\Theta'},$ we must compute the expression $_{\mathcal{A}_{\Theta}}\langle f e',f\rangle,$   and we have
\begin{eqnarray}
 _{\mathcal{A}_{\Theta}}\langle f e',f\rangle=e_1 = \psi(e') &=&\psi(G_{-}(\mathfrak{v}_4^{-1})\mathfrak{v}_3^{-1})+G(\mathfrak{v}_4^{-1})+G_{+}(\mathfrak{v}_4^{-1})\mathfrak{v}_3) \nonumber\\
   &=& G_{-}(\psi(\mathfrak{v}_4^{-1}))\psi(\mathfrak{v}_3^{-1})+G(\psi(\mathfrak{v}_4^{-1}))+G_{+}(\psi(\mathfrak{v}_4^{-1}))\psi(\mathfrak{v}_3).\nonumber
\end{eqnarray}
  Hence if we can compute the expressions of  $\psi(\mathfrak{v}_3)$ and $\psi(\mathfrak{v}_4)$ explicitly, we will get an explicit expression of $e_1.$ Let us give an explicit expression of $\psi(\mathfrak{v}_3).$ Similar computations will also hold for $\psi(\mathfrak{v}_4).$
  From Equation~(\ref{eq:module3}), $$(f\mathfrak{v}_3)(x)=e^{2\pi i\langle S(l),J'S(l)\slash2\rangle}\langle -x,S''(l)\rangle f(x-S'(l)),$$ where $l=(0,0,-1,0).$ A direct computation gives

  $$(f\mathfrak{v}_3)(x_0,x_1,x_2)=e^{-\pi i \frac{\theta_{23}\theta_{13}}{\theta_{12}}}e^{-2\pi ix_0 \frac{\theta_{13}}{\theta_{12}}-x_2\frac{\theta_{34}}{2}}  f(x_0+\theta_{23},x_1+1,x_2).$$
Finally, from Equation~(\ref{eq:module2})

  $$ _{\mathcal{A}_{\Theta}}\langle f\mathfrak{v}_3,f\rangle(m)=e^{2\pi i\langle T(m),J'T(m)\slash2\rangle}\int_{\R \times \Z^2}\langle x,T''(m)\rangle f(x+T'(m))f\mathfrak{v}_3(x)dx$$

  The above expression is zero unless $m=(m_1,m_2,1,0),$ and when $m=(m_1,m_2,1,0),$ the expression becomes


$$e^{2\pi i\theta_{13}m_1}\int_{\R}e^{2\pi im_2x}e^{-2\pi ix\frac{\theta_{13}}{\theta_{12}}}  \sqrt{\phi}(x+\frac{\theta_{12}m_1}{2}-\frac{\theta_{23}}{2})\sqrt{\phi}(x-\frac{\theta_{12}m_1}{2}+\frac{\theta_{23}}{2})dx,$$
which is non-zero only when $|m_1\theta_{12}-\theta_{23}|<2\theta_{12}$, i.e. when
 $-2+\frac{\theta_{23}}{\theta_{12}}<m_1<2+\frac{\theta_{23}}{\theta_{12}}.$ Let $t_1=\lfloor-2+\frac{\theta_{23}}{\theta_{12}}\rfloor$ and $t_2=\lfloor2+\frac{\theta_{23}}{\theta_{12}}\rfloor.$Then $\psi(\mathfrak{v}_3)=\sum_{a=t_1}^{t_2}\mathfrak{u}_3 H^{-}_{-a}(\mathfrak{u}_2^{-1})\mathfrak{u}_1^{a},$
where
$H^{-}_{-a}(x)=\sum_{n\in \mathbb{Z}}\Phi_a(x+n),$  $$\Phi_a(x)= e^{2\pi i\theta_{13}a}e^{-2\pi ix\frac{\theta_{13}}{\theta_{12}}}  \sqrt{\phi}(x+\frac{\theta_{12}a}{2}-\frac{\theta_{23}}{2})\sqrt{\phi}(x-\frac{\theta_{12}a}{2}+\frac{\theta_{23}}{2}).$$

\section*{Acknowledgments}
The first named author acknowledges the support of DST, Government of India (\emph{DST-INSPIRE
Faculty Scheme} with Faculty Reg. No. IFA19-MA139). The second named author was supported by the National Natural Science Foundation of China (No. 11401256) and the Zhejiang Provincial Natural Science Foundation of China (No. LQ13A010016).

\end{document}